\documentclass[11pt]{article}%
\usepackage{amssymb}
\usepackage{amsfonts}
\usepackage{amsmath}
\usepackage{amssymb}
\usepackage{color}
\usepackage{graphicx}
\usepackage{caption}%
\usepackage[authoryear]{natbib}
\usepackage{hyperref}
\usepackage{float}
\usepackage{subcaption}
\usepackage{multirow}

\usepackage{tabularx}
\usepackage{algorithm}
\usepackage{algpseudocode}
\providecommand{\U}[1]{\protect\rule{.1in}{.1in}}
\newtheorem{theorem}{Result}

\newenvironment{proof}[1][Proof]{\noindent \textbf{#1.} }{\  \rule{0.5em}{0.5em}}
\begin{document}
	
	\title{\textbf{Robust inference for an interval-monitored step-stress experiment under proportional hazards  }}
	\author{Narayanaswamy Balakrishnan,  María Jaenada and Leandro Pardo}
	\date{ }
	\maketitle
	
	\begin{abstract}
		
		Accelerated life tests (ALTs) play a crucial role in reliability analyses, providing lifetime estimates of highly reliable products under normal conditions. Among the different types of ALTs, the step-stress design incrementally increases the stress level at predefined times, while maintaining a constant stress level between successive changes. This approach accelerates the occurrence of failures, effectively reducing experimental duration and cost.
		While many studies focusing on ALTs assume a specific form for the lifetime distribution; in certain applications %, any known parametric form might be suitable. 
		instead a general form satisfying certain properties, 
		such as the proportional hazards requirement, should be preferred. In particular, the proportional hazard model assumes that applied stresses act multiplicatively on the hazard rate, and thus the hazards function under increased stress may be divided into two factors, with one representing the effect of the stress, and the other representing the baseline hazard.
		In this work, we examine particular forms of baseline hazards, namely, linear and quadratic forms. Moreover, certain experiments may face practical constraints, making continuous monitoring of devices infeasible. Instead, devices under test are inspected at predetermined intervals, and resulting data are then grouped as counts of failures, leading to interval-censoring.
		On the other hand, recent works have shown an appealing trade-off between the efficiency and robustness of divergence-based estimators in parametric inference under interval-censored data. This paper introduces the step-stress ALT model under proportional hazards and presents a robust family of minimum density power divergence estimators (MDPDEs)  for estimating device reliability and related lifetime characteristics such as mean lifetime and distributional quantiles. The MDPDEs of related lifetime characteristics and their asymptotic distributions are derived, providing approximate confidence intervals for the parameters of interest.
		Empirical evaluations through Monte Carlo simulations demonstrate their performance in terms of robustness and efficiency. Finrally, an illustrative example is provided to demonstrate the usefulness of the model and the associated method developed here.

	\end{abstract}
	
	\section{Introduction \label{sec:Introduction}} 
	
	%ALTs
	Life testing analysis is a major concern in reliability, with applications in diverse fields including engineering and biomedical sciences.
	Moreover,  driven by customer expectations, nowadays products are designed to be highly reliable with substantially long
	lifetimes. Consequently, life-testing experiments for such highly reliable products  under normal operating conditions (NOC)  would usually  lead to few failures (if at all) in the experiment.
	As a result, to achieve accurate inference under NOC, large experimental times, and consequently, high cost,s are necessitated. 
	To address this limitation, accelerated life-tests (ALTs) reduce the time to failure by exposing the experimental units to stronger test conditions than NOC, causing a greater number of observed failures. The failure data under increased stress levels can be accurately analyzed and subsequently, the results from the reliability analysis conducted under increased stress conditions can be extrapolated to make inferences about the product's performance under NOC.
	%	Among the different ALT plans, 
	In particular, step-stress ALT plans are noteworthy. These plans involve increasing the stress factor at which units are exposed at certain (typically pre-specified) times, thereby subjecting all devices under test to the same stress patterns and heightening the probability of failure. That is,  the stress is increased and held constant between two successive stress levels for all non-failed units.
	 The step-stress ALT design enables efficient data collection and accurate inference with smaller number of devices than  under the popular constant-stress testing, and also would result in shorter experimental times. But,  it necessitates a model relating the relationship between the wear caused by preceding stresses and the current lifetime distribution.
	
	%PH model
	Several known distributions have been explored in the literature to describe the lifetimes of tested products, including the exponential, Weibull, gamma, log-normal or generalized gamma distributions. However, the true underlying distribution for the data is generally unknown and so none of these parametric distributions may be adequate to fit the observed data. 
	Alternatively, \cite{cox1972regression}'s proportional hazards (PH)  model offers a semi-parametric multiple regression approach for reliability estimation, in which the baseline hazard function is multiplicatively affected  by the applied stresses. The PH model	is distribution-free and relies on the assumption that the hazard rate ratio between two stress levels remains constant over time. 
	 This model can be particularly useful when  no expert knowledge may be available to select a specific parametric lifetime distribution family, but certain natural constraints, such as the constant hazard rate ratio, can be made. In this paper, we adopt the PH Cox model with linear and quadratic forms for the baseline hazard function.

	%interval-censoring
	Although step-stress tests are typically conducted with continuous monitoring of unit lifetimes, there are situations where continuous monitoring is not feasible, but units are only inspected for failures at specific time intervals. In such cases,
	the observed data are grouped failure counts between two consecutive inspection times, resulting in interval-censoring.
	This interval-censored set-up may arise in experiments where functional check requires a manual test, or in industrial experiments where cost limitations make it more feasible to intermittently monitor the lifetime of units at specific intervals rather than on a time-continuous manner.
	For example, \cite{balakrishnan2023robust} applied the interval-censored step-stress model for analyzing lifetime data from non-destructive tests with one-shot devices. 
	Under this scenario, all data observed are censored, and so specific inferential methods need to be developed to handle such situations effectively.
	
	%MLE
	In parametric, inference the model parameters are typically estimated  using maximum likelihood estimators (MLEs) due its good asymptotic properties.
	However, this method suffers from lack of robustness and can be significantly affected by outliers present in the data
	The accuracy of the inference procedure profoundly affects the reliability estimates at normal operating conditions and the subsequent
	decisions regarding product configuration, warranties or preventive maintenance schedule. Therefore, contamination in the data can exert a strong negative impact on the results when employing the MLE.
	
	In recent years, considerable efforts have been made towards developing robust methods  for interval-censored ALT, most of them based on divergence measures (see, for example, \cite{balakrishnan2023robust, balakrishnan2023step, balakrishnan2023Non}). In this paper, we extend the robust approach proposed in the above mentioned papers	and develop here robust estimators %and tests 
	for interval-monitored ALTs based on	divergence measures under the PH  model.	To best of our knowledge, no research on inference for interval-censored ALTs within the PH model framework, either based on MLE or robust estimators, has been carried out so far.
	
	The rest of the paper is organized as follows:	Section \ref{sec:hazard} revisits the step-stress ALT model under the proportional hazards assumption and introduces two functional forms for the baseline hazard—namely, linear and quadratic baselines.
	In Section \ref{sec:MDPDE}, a family of robust estimators based on density power divergence (DPD) is defined, and explicit expressions of the model for the linear and quadratic forms are obtained. Additionally, the estimating equations and asymptotic distribution of the robust estimators are derived.
	In Section \ref{sec:char}, three of the main lifetime characteristics, mean lifetime, reliability, and distributional quantiles, under both linear and quadratic baseline hazards are estimated,  and associated asymptotic distributions and confidence intervals are obtained by applying the Delta method. The performance of the proposed estimators is then empirically evaluated in terms of accuracy and robustness in Section \ref{sec:simulation}.
	An illustrative example demonstrates the usefulness of the model and the method developed here in  Section \ref{sec:realdata}. Finally, some concluding remarks are made in Section \ref{sec:concluding}.

	\section{Step-stress ATL with proportional hazards \label{sec:hazard}}

	Life testing of highly reliable products poses a challenge, but if a specific environmental factor influencing device reliability is identified, an ALT can be conducted to intentionally induce more failures compared to a NOC.
	When a lifetime distribution is assumed, the functional relationship between the stress level and the devices reliability function can be established in terms of the model parameters. For example, for scale parametric families, the scale parameter at a constant stress $x$, say $\lambda,$ is usually assumed to be log-linearly related to the the stress level of the form
	\begin{equation} \label{eq:loglinear}
		\lambda(x; \boldsymbol{a}) = \log(a_0 + a_1x),
	\end{equation}
	with $\boldsymbol{a}= (a_0, a_1)^T \in \mathbb{R}^2.$
	Proportional hazards (PH) models, first proposed  by \cite{cox1972regression}, facilitate a semi-parametric regression
	approach for reliability estimation, in which the baseline hazard function is
	affected multiplicatively by the  applied stresses. That is, if we denote by $\lambda(t, x)$ the hazard rate function of a product lifetime under a constant stress $x,$ then we can multiplicatively separate the effects of the natural  product reliability and the stress level experienced as 
	\begin{equation}\label{eq:PH}
		h(t, x) = h_0(t) \lambda(x; \boldsymbol{a}),
	\end{equation}
	where $\lambda(x; \boldsymbol{a})$ is the log-linear relation stated in (\ref{eq:loglinear}).
	The PH model assumes that the ratio of hazard rates between two stress levels is constant over time, i.e.,
		\begin{equation*}
		\frac{h(t, x_1)}{h(t, x_2)} =\frac{\lambda(x_1; \boldsymbol{a})}{\lambda(x_2; \boldsymbol{a})}
	\end{equation*}
	for any two stress levels $x_1$ and $x_2.$
	One of the advantages of the PH model is its flexibility to extend it to time-dependent covariates. 
	On the other hand, since the stress level is increased during the experimentation in step-stress ALTs, a model relating the lifetime distribution of a device at a specific (current) and preceding stress levels becomes necessary.
	The natural underlying assumption is that the reliability of devices under test at a given stress level is somehow influenced by the stress levels experienced earlier, causing it to wear and tear. 
	Three main models have been discussed in the literature for a such purpose.
	\cite{degroot1979bayesian} proposed the tampered random variable model and discussed optimal tests under a Bayesian framework.
	\cite{bhattacharyya1989tampered} proposed the tampered	failure rate model which assumes that the effect of change of stress is to multiply the initial failure rate function by a factor subsequent to the stress change time. Finally, \cite{Sedyakin1966}, \cite{Bagdonavicius1978} and \cite{nelson1990accelerated} all adopted the cumulative exposure	model. The natural assumption underlying the cumulative exposure model has led to its widespread use adoption.
	 \cite{elsayed2002optimal} adopted the cumulative exposure model for optimal step-stress ALT under PH, and it is the  approach that is employed here.
	%
	 %Additionally, since an increased stress level is supposed to increase  the failure rate, a model relating the lifetime distribution of the device with the 
	 Of course, the appropriate form of these two models, the first relating the stress level to the failure rate (at constant stress) and the other relating the change on the reliability when increasing the stress will depend on the nature of the product and the assumptions made about its degradation or failure mechanism.

	The cumulative exposure model states that the residual life of an experimental unit depends only on the cumulative exposure it has experienced, with no memory of how this exposure was accumulated.
	That is, the effect of increasing the stress level from  $x_1$ to $x_2$ at time $\tau$ on the  lifetime distribution of the device is mathematically defined by a shift on the lifetime distribution such that  continuity is ensured, i.e.,
	$$F_1(\tau) = F_2(\tau+s),$$
	where $F_i(\cdot)$ denotes the cumulative lifetime distribution (CDF) at constant stress level $x_i$ and  the shifting time $s <0$ represents the accumulated damage until the time of stress change.	Alternately,  the previous condition can be stated in terms of the cumulative hazard function at constant stress level $x_i$, denoted in the following by $H_i(\cdot), i=1,2,$ as follows:
	$$H_1(\tau) = H_2(\tau+s),$$
	where $H_i(t) = -\log(1-F_i(t)), i=1,2,$ and $s$ is the shifting time.
	
	Bringing together the PH and the cumulative exposure assumptions, the shifting time $s$ must satisfy the following relationship
	\begin{equation} \label{eq:CEPH}
		H_1(\tau) = \lambda(x_1, \boldsymbol{a}) \int_0^\tau h_0(t)dt  =  \lambda(x_2, \boldsymbol{a}) \int_0^{\tau+s} h_0(t)dt  = 		H_2(\tau+s).
	\end{equation}
	Note that the hazard rate of the lifetime over the step-stress experiment  may not be continuous at time $\tau,$ and the continuity assumption is inherited only by the cumulative hazard function.
	% $h_T(\cdot)$ %solving in (\ref{eq:CEPH}), we have that the shifting time under PH can be computed as
	%$$ s = h_0^{-1}\left(h_0(\tau) \lambda(x_1; \boldsymbol{a}) /\lambda(x_2; \boldsymbol{a}) \right)- \tau.$$
%
	From the above formula, the cumulative hazard function of the device lifetime, $T,$ under a simple step-stress model with stress levels $x_1$ and $x_2,$ is given by
	\begin{equation}
		H_T(t) = \begin{cases}
			H_1(t) = \lambda(x_1,\boldsymbol{a}) \int_{0}^{t} h_0(u)du, & 0 < t < \tau,\\
			H_2(t+s) = \lambda(x_2,\boldsymbol{a}) \int_{0}^{t+s} h_0(u)du,  &  \tau < t < \infty,
		\end{cases}
	\end{equation}
%	where the shifting time $s = H_2^{-1}(H_1(\tau)) - \tau$,
	 and consequently the reliability of the device can be obtained as
	\begin{equation}
		R_T(t) = \begin{cases}
			R_1(t) = \exp(-H_1(t)), & 0 < t < \tau,\\
			R_2(t+s) = \exp(-H_2(t+s)), &  \tau < t < \infty,
		\end{cases}
	\end{equation}
	where the shifting time $s = H_2^{-1}(H_1(\tau)) - \tau$.  
%	From the hazard rate function we can readily obtain the cumulative hazard of the lifetime as
%	\begin{equation}
%		h_T(t) = \begin{cases}
%			H_1(t) = \int_{0}^{t} h_1(u)du & 0 < t < \tau,\\
%			H_2(t+s) =  H_1(\tau) + \int_{\tau}^{t+s} h_2(u)du&  \tau < t < \infty,
%		\end{cases}
%	\end{equation}
%	and the reliability function of the lifetime as
%	\begin{equation}
%		R_T(t) = \begin{cases}
%			R_1(t) = \exp(-H_1(t)) & 0 < t < \tau,\\
%			R_2(t+s) = \exp(-H_2(t+s)) &  \tau < t < \infty,
%		\end{cases}
%	\end{equation}
%
	%If the baseline hazard belongs to a parametric family with scale parameter $\theta$, then the shifting time can be computed as 	$$s = \frac{\theta_2}{\theta_1}\tau - \tau$$
	%
	In practical use, the shifting time could be estimated non-parametrically or by assuming a specific form for the baseline hazard. Because the observed data are usually censored, the non-parametric estimator may not be accurate. For this reason, we will assume here linear and quadratic functions for the baseline hazard.

	The assumption of linear and quadratic forms is made here with the hope that the true (but unknown) hazard may be satisfactorily approximated by a polynomial of degree one or two (low degree).
	A reasonable justification is derived from the fact that some hazard functions from known distributions  may be expanded in a converging power series. Therefore, the linear and quadratic forms can be seen as first- and second-order Taylor approximations of any underlying general hazard function.
	Furthermore, in some sense, the use of a polynomial transformation is a non-parametric technique, and consequently may be considered as an alternative to other non-parametric estimates.
%	%Electric Institute (1942), An Appraisal of Methods for Estimating Service Lives of Utility Properties (mimeo), New York
	%(NARUC) National Association of Railroad and Utilities Commissioners (1938), Report of the Committee on Depreciation, Washington: NARUC.
	 %The procedure is akin to the 'normalization' of frequency functions by polynomial trans- formations (cf. Kendall and Stuart [1958, p. 163]). 
	 %
	 %Another parallel may be found in the choice of metameters in bioassay, although the polynomial trans- formation is seldom used in this application (but see Finney [1952, p. 162]).  
	 %
	 %In a sense, the use of a polynomial transformation is a non-parametric technique, and consequently may be considered as an alternative to the product-limit estimate and other non-parametric estimates described by Kaplan and Meier [1958]. (The author is indebted to a referee for calling attention to the relevance of this work in similar contexts.) The latter methods have not, to our knowledge, been compared with polynomial transformation estimates obtained by the NARUC [1938] or AGA-EEI [19421 techniques or the present treatmen
	
	The assumption of a polynomial hazard rate is equivalent to the assumption of polynomial exponential transformation of the lifetime adopted in \cite{krane1963analysis} and, as pointed out therein, it is suitable for industrial property survival data (see \cite{NARUC, AGAEEI} reports).

	\subsection{Linear hazards}
	
	%Weibull ??
	Let us consider linear form for the baseline hazard as
	\begin{equation} \label{eq:baselinehazardlinear}
		h_0(t) = \gamma_0 + \gamma_1 t, \hspace{0.3cm} t>0,
	\end{equation} 
	where $\boldsymbol{\gamma} = (\gamma_0, \gamma_1)$ are the linear relationship coefficients. 
	Using the same notation as in Section \ref{sec:hazard}, the hazard rate function of a device lifetime subjected to constant stress $x$ under PH is given by
	\begin{equation} \label{eq:hazardlinear}
		h(t, x) = ( \gamma_0 + \gamma_1 t) \exp\left(a_1 x\right),
	\end{equation}
	where $\boldsymbol{a} = (a_0, a_1) \in \mathbb{R}^2$ are the coefficients of the log-linear relationships in  (\ref{eq:loglinear}) relating the stress level to the devices lifetimes.
	However, note that the intercept $a_0$ becomes superfluous in Eq. (\ref{eq:hazardlinear}) under linear baseline assumption, since its multiplicative role can be incorporated into the baseline hazard parameters $\boldsymbol{\gamma}.$ Therefore, for parametric inference on the reliability of the product under a constant stress, we would need to estimate both relationships coefficients in (\ref{eq:hazardlinear}) resulting in a $3$-dimensional model parameter $\boldsymbol{\theta} = (\boldsymbol{\gamma}, a_1).$ 
	
	Now, the corresponding cumulative hazard and reliability functions at constant stress $x$ are given, respectively, by
	\begin{equation*}
		%\begin{aligned}
			H_{\boldsymbol{\theta}}(t, x) = \left( \gamma_0t + \gamma_1 \frac{t^2}{2}\right) \exp\left(a_1 x\right) 
			\hspace{0.2cm} \text{and} \hspace{0.2cm}
			R_{\boldsymbol{\theta}}(t, x) = \exp\left[-\left( \gamma_0t + \gamma_1 \frac{t^2}{2}\right) \exp\left(a_1 x\right) \right].
		%\end{aligned}
	\end{equation*}

	Under the simple step-stress experiment with two stress levels $x_1$ and $x_2,$ the cumulative lifetime hazards at constant stresses are combined to yield the step-stress cumulative hazard function as
	\begin{equation}
		H_{\boldsymbol{\theta}}(t) = \begin{cases}
		 \exp\left( a_1 x_1\right)  \left( \gamma_0t + \gamma_1 \frac{t^2}{2}\right), & 0 < t < \tau,\\
		\exp\left(a_1 x_2\right)  \left( \gamma_0(t+s) + \gamma_1 \frac{(t+s)^2}{2}\right), &  \tau < t < \infty,
		\end{cases}
	\end{equation}
	 where $s$ is the solution of the second-order equation
	\begin{equation}\label{eq:secondorderequation}
   \frac{\gamma_1}{2}(\tau+s)^2 +\gamma_0(\tau+s) - \exp\left( a_1 (x_1-x_2)\right) \left( \gamma_0 \tau + \gamma_1 \frac{\tau^2}{2}\right) = 0.
   \end{equation}

The above equation has at least one real root, since all coefficients are positive and so, the discriminant of the second-order equation is also positive.
%$$\gamma_0^2 + 4  \frac{\gamma_1}{2} \frac{\lambda(x_1,\boldsymbol{a})}{	 \lambda(x_2,\boldsymbol{a})} \left( \gamma_0 \tau + \gamma_1 \frac{\tau^2}{2}\right) > 0.$$
Moreover, it has at least one negative solution
%more than one negative solution,
%mirar bien pero creo que siempre tiene una de cada signo
and so we will choose the greater negative solution of the above equation to ensure that $s<0$ and $\tau + s >0.$ 

The reliability function of a device under test will be given as
	\begin{equation} \label{eq:reliabilitylinearhazard}
	R_{\boldsymbol{\theta}}(t) = \begin{cases}
		\exp\left(-\exp\left( a_1 x_1\right) \left( \gamma_0t + \gamma_1 \frac{t^2}{2}\right)\right), & 0 < t < \tau,\\
		\exp\left(-\exp\left(a_1 x_2\right) \left( \gamma_0(t+s) + \gamma_1 \frac{(t+s)^2}{2}\right)\right), &  \tau < t < \infty,
	\end{cases}
\end{equation}
where $s$ satisfies the second-order equation in (\ref{eq:secondorderequation}). 

The model parameter should satisfy some natural restrictions. Firstly, because the hazard function is positive at any time $t>0,$ both coefficients $\gamma_0$ and $\gamma_1$ in (\ref{eq:baselinehazardlinear}) must be positive. Moreover, under positive and increasing stress levels, the reliability of the devices should decrease when increasing the stress level. Then, the regression coefficient in (\ref{eq:loglinear}), namely $a_1,$ should also be positive. Summarizing, assuming increasing and positive stress levels, the parameter space now becomes $\Theta = \mathbb{R}^{3+}.$ 

%$$\exp\left(-\exp\left(a_0 + a_1 (x+y)\right)  \left( \gamma_0t + \gamma_1 \frac{t^2}{2}\right)\right) < \exp\left(-\exp\left(a_0 + a_1 x \right)  \left( \gamma_0t + \gamma_1 \frac{t^2}{2}\right)\right)$$

\subsection{Quadratic hazards} \label{subsec:quadratichazards}
	
	Let us now assume a quadratic form for the baseline hazard function, as
	\begin{equation} \label{eq:baselinehazardquadratic}
		h_0(t) = \gamma_0 + \gamma_1 t + \gamma_2t^2, \hspace{0.3cm} t>0,
	\end{equation} 
	where $\boldsymbol{\gamma} = (\gamma_0, \gamma_1, \gamma_2)$ are the quadratic relationship coefficients. 
	Applying the PH condition, the hazard rate function of a device lifetime subjected to a constant stress $x$ is then given by
	$$h(t, x) = ( \gamma_0 + \gamma_1 t + \gamma_2t^2) \exp\left(a_1 x\right)$$
	%where $\boldsymbol{a} = (a_0, a_1) \in \mathbb{R}^{},$ 
	and the corresponding cumulative hazard function is 
	\begin{equation}
		H_{\boldsymbol{\theta}}(t) = \begin{cases}
			\exp\left(a_1 x\right) \left( \gamma_0t + \gamma_1 \frac{t^2}{2} + \gamma_2 \frac{t^3}{3}\right), & 0 < t < \tau,\\
			\exp\left(a_1 x\right)\left( \gamma_0(t+s) + \gamma_1 \frac{(t+s)^2}{2}  + \gamma_2 \frac{(t+s)^3}{3}\right), &  \tau < t < \infty,
		\end{cases}
	\end{equation}
	where $s$ is the solution of the equation
	\begin{equation}\label{eq:thirdorderequation}
		\frac{\gamma_2}{3}(\tau+s)^3 + \frac{\gamma_1}{2}(\tau+s)^2 +\gamma_0(\tau+s) - \exp\left( a_1 (x_1 - x_2)\right) \left( \gamma_0 \tau + \gamma_1 \frac{\tau^2}{2} + \gamma_2 \frac{\tau^3}{3}\right) = 0.
	\end{equation}
	The intercept of the log-linear relationship has again been omitted here because its multiplicative effect is represented by the coefficients of the quadratic baseline hazard function.
	Equation (\ref{eq:thirdorderequation}) has at least one real solution. As before, we may choose $s$ as the greatest negative solution of the equation.  
	Note that
	under quadratic baseline hazards, the model parameter is a 4-dimensional vector $\boldsymbol{\theta} = (\boldsymbol{\gamma}, a),$ with $\boldsymbol{\gamma}$ and $a_1$ being the coefficients defining the multiplicative factors in Equations (\ref{eq:baselinehazardquadratic}) and (\ref{eq:loglinear}). As before, there exist natural restrictions on the parameter space;
	%First, the regression coefficient of the log-linear relationship $a_1$ must be positive to ensure that the reliability decreases when increasing the stress level. Secondly, since the baseline hazard is a positive function in $t$, the intercept term $\gamma_0$ in (\ref{eq:hazardquadratic}) must also be positive. 
	the reliability of the devices should decrease when increasing the stress level, and the baseline hazard must be positive. To simplify the parameter constraints, we assume that all the coefficients in Equation (\ref{eq:baselinehazardquadratic}) are positive and so is $a_1$ for the log-linear factor. This is a strong assumption, but realistic in many situations, as suggested in \cite{krane1963analysis}. Therefore, the parameter space of $\boldsymbol{\theta} = (\boldsymbol{\gamma}, a_1)$ gets reduced to
	$\Theta = \mathbb{R}^{4+}.$ 
	
	% Moreover, the positive hazard condition implies that the discriminant $\Delta = \gamma_1^2 - 4\gamma_2 \gamma_0 $ must also to be positive.
	Exponentiating the negative cumulative hazard function we obtain the reliability function of a device under test under quadratic baseline hazards  as
	\begin{equation} \label{eq:reliabilitylinearhazard}
		R_{\boldsymbol{\theta}}(t) = \begin{cases}
			\exp\left(-\exp\left( a_1 x_1\right) \left( \gamma_0t + \gamma_1 \frac{t^2}{2} + \gamma_2\frac{t^3}{3}\right)\right), & 0 < t < \tau,\\
			\exp\left(-\exp\left(a_1 x_2\right) \left( \gamma_0(t+s) + \gamma_1 \frac{(t+s)^2}{2}+ \gamma_2\frac{(t+s)^3}{3}\right)\right), &  \tau < t < \infty,
		\end{cases}
	\end{equation}
	where the shifting time $s$ satisfies Equation (\ref{eq:thirdorderequation}).

%	Further, if the estimate of the quadratic coefficient $\gamma_2$ is zero, then a linear form of the baseline hazard underlies. That is, we can test the significance of the quadratic form as $\operatorname{H}_0: \gamma_2=0.$ We will discuss this test in detail in Section \ref{sec:waldtest}
	
	\section{Robust estimators \label{sec:MDPDE}}
	
	Consider a interval-monitored simple step-stress ALT model  with two stress levels, $x_1$ and $x_2,$ switched at a pre-fixed time $\tau.$ A total of $N$ devices are tested and the number of failures is recorded at $L$ fixed inspections times $IT \in \{0 = t_0 < t_1 < \cdots < t_L\},$ including the time of stress change $\tau = t_{k}$ for a certain $k < L.$ Here, $k$ is the number of inspection times under the first stress and $L-k$ is the number of inspection times under the second stress level.
	Note that the experiment termination time is also fixed as $t_L$ (Type I censoring design).
	Then, the numbers of devices that have	failed within the $j$-th inspection interval, $[t_{j-1},t_j)$ are recorded as  $n_j$  and the remaining $n_{L+1} = N-\sum_{j=1}^{L} n_j$ is the number of surviving units after the termination time. The  data  recorded from such as step-stress experiment is summarized in Table \ref{table:design}.
		\begin{table}[htb]
		\centering
		\begin{tabular}{ccc}
			IT & Stress level & Number of failures \\
			\hline
			$t_0=0$ & $x_1$ & - \\
			$t_1$ & $x_1$ & $n_1$  \\
			\vdots &\vdots &\vdots  \\
			$t_{k-1}$ & $x_1$ & $n_{k-1}$ \\
			%\hline
			$t_{k}$ & $x_2$ & $n_{k}$  \\
			\vdots &\vdots &\vdots  \\
			$t_{L}$ & $x_2$ & $n_{L}$  \\
			\hline
			%& & 
		\end{tabular}
	\caption{Step-stress ALT design with interval-monitoring}
	\label{table:design}
	\end{table}

	Because the observed data are interval-censored, the complete likelihood of the model will not be useful. Instead, we can define the incomplete likelihood based on a multinomial model of the grouped data as follows. For each device under test, we consider the $L+1$ possible experimental results corresponding to failing within  inspected intervals $[t_{j-1},t_j), j=1,...,L,$ and surviving after the experiment termination time.
	The probability of failure within an interval $[t_{j-1},t_j)$ can be determined by evaluating the reliability of the lifetime as
	\begin{equation} \label{eq:th}
		\pi_j(\boldsymbol{\theta}) = R_{\boldsymbol{\theta}}(t_{j-1}) - R_{\boldsymbol{\theta}}(t_j) = \exp \left(-H_{\boldsymbol{\theta}}(t_{j-1})\right) - \exp \left(-H_{\boldsymbol{\theta}}(t_j)\right),
	\end{equation}
	and the probability of survival at termination time $t_L$ is given by
	$$\pi_{L+1}(\boldsymbol{\theta}) =  R_{\boldsymbol{\theta}}(t_L) = \exp \left(-H_{\boldsymbol{\theta}}(t_{L})\right). $$
	Using the expressions given in Equations (\ref{eq:baselinehazardlinear}) and (\ref{eq:baselinehazardquadratic})  for linear and quadratic baseline hazards, respectively, we can compute the probabilities of each experimental result under both linear and quadratic baseline hazard assumptions. %$\pi_j(\boldsymbol{\theta}).$ 
	
	Then, adopting the multinomial approach with probability vector $\boldsymbol{\pi} (\boldsymbol{\theta}) = \left(\pi_1(\boldsymbol{\theta}),..., \pi_{L+1}(\boldsymbol{\theta})\right)^T$,
	the incomplete log-likelihood of the data is given by 
	\begin{equation}\label{eq:likelihood}
		\mathcal{L}(\boldsymbol{\theta}) \propto \sum_{j=1}^{L+1} n_j \log\left(\pi_{j}(\boldsymbol{\theta}) \right),
	\end{equation}
	and consequently the maximum likelihood estimator (MLE) is the maximizer (or equivalent by minimizer) of the previous log-likelihood (or negative log-likelihood). The MLE is a commonly used estimator due to its well known asymptotic properties; it is a consistent and efficient estimator. However, it is highly sensitive to outliers, which can lead to erroneous analysis and results in the presence of contamination in data.
	Alternatively, divergence-based estimators have a great advantage in terms of robustness with a small loss in efficiency for general statistical models. In particular, 
	\cite{balakrishnan2023robust, balakrishnan2023step, balakrishnan2023Non} developed robust minimum density power divergence estimators (MDPDEs) for non-destructive one-shot devices tested under step-stress experiments under the assumption of exponential, gamma and lognormal lifetime distributions, respectively. They demonstrated that the MDPDE family provides a bridge between efficient and robust estimators.
	Following on from previous work, we develop here the MDPDE under the PH assumption. 
	
	The DPD between two distributions aims to quantify the statistical closeness between the two associated random variables; The greater DPD they have, the more distinct the variables would be. In contrast, the DPD would be small (but always positive) for similar distributions. Then, when assuming a parametric model, we should look for a distribution that is as close as possible to the true underlying distribution.
	The DPD between two probability mass functions $\boldsymbol{p} = \left(p_1,...,p_{L+1}\right)^T$ and $\boldsymbol{q} = \left(q_1,...,q_{L+1}\right)^T$   is given, for $\beta >0$, by
	\begin{equation*}
		d_\beta(\boldsymbol{p}, \boldsymbol{q}) = \sum_{j=1}^{L+1} q_j^{\beta+1} - \left( 1 + \frac{1}{\beta}\right)\sum_{j=1}^{L+1} q_j^{\beta} p_j + \frac{1}{\beta}\sum_{j=1}^{L+1} p_j^{\beta+1}.
	\end{equation*}
	The DPD can be defined at $\beta = 0$ by taking continuous limit as $\beta \rightarrow 0$ and it yields the well-known Kullback-Leibler divergence given by
	\begin{equation*}
		d_{KL}(\boldsymbol{p}, \boldsymbol{q}) = \sum_{j=1}^{L+1} p_j\log\left(\frac{p_j}{q_j}\right).
	\end{equation*}
	
	As the true  underlying distribution is unknown, we can approximate it non-parametrically through its empirical probability vector $\widehat{\boldsymbol{p}} = \left(\frac{n_1}{N},...,\frac{n_L}{N}, \frac{n_{L+1}}{N}\right)^T$ and the DPD between the empirical and the assumed distribution, $\widehat{\boldsymbol{p}}$ and $\boldsymbol{\pi}(\boldsymbol{\theta}),$  is then given by
		\begin{equation}\label{eq:DPDloss}
		d_\beta(\widehat{\boldsymbol{p}}, \boldsymbol{\pi}(\boldsymbol{\theta})) = \sum_{j=1}^{L+1} \begin{cases}
		 \pi_j(\boldsymbol{\theta})^{\beta+1} - \left( 1 + \frac{1}{\beta}\right) \pi_j(\boldsymbol{\theta})^{\beta} \widehat{p}_j + \frac{1}{\beta} \widehat{p}_j^{\beta+1}, & \beta >0,\\
	 \widehat{p}_j\log\left(\frac{\widehat{p}_j}{\pi_j(\boldsymbol{\theta})}\right), & \beta=0.
		\end{cases}
	\end{equation}
	
	The closer the empirical probability and the assumed probability vector are, the better fit the assumed model provides for the observed data. Therefore, the MDPDE is defined as the minimizer of the DPD loss as follows:
	\begin{equation*}
		\widehat{\boldsymbol{\theta}}^\beta = \operatorname{min}_{\boldsymbol{\theta} \in \Theta} d_\beta(\widehat{\boldsymbol{p}}, \boldsymbol{\pi}(\boldsymbol{\theta}))
	\end{equation*}
	with $\Theta = \mathbb{R}^{2+}  \times \mathbb{R} \times \mathbb{R}^{+},$ for linear baseline and $\Theta = \mathbb{R}^{3+}  \times \mathbb{R} \times \mathbb{R}^{+}$  for the quadratic baseline. 
	Observe that the negative log-likelihood in (\ref{eq:likelihood}) and Kullback-Leibler divergence ($\beta=0$) in (\ref{eq:DPDloss}) are equivalent as objective functions. Moreover, the MLE is included in the MDPDE family as a particular case for $\beta=0,$ and all the asymptotic properties of the MLE are derived in the following section by specializing at $\beta=0.$

	\subsection{Linear baseline}
	
	% asymptotic
	
	We study the asymptotic properties of the MDPDEs for interval-monitored step-stress ALTs under linear baseline PH first. The asymptotic distribution of the MLE as well as an explicit expression of the Fisher information matrix of the incomplete model will be obtained as particular cases for our general results.
	
	First,  we introduce some useful notation. 
	Following Expression (\ref{eq:th}), the probability of failure within an inspected interval $(t_{j-1}, t_j]$ is given by
	\begin{equation} \label{pijquadratic}
		\pi_j(\boldsymbol{\theta})  = 
		\begin{cases}
			\exp \left( -\exp(a_1x_1) \left( \gamma_0t_{j-1} + \gamma_1 \frac{t_{j-1}^2}{2}\right) \right) -
			\exp \left( -\exp(a_1x_1) \left( \gamma_0t_j + \gamma_1 \frac{t_j^2}{2}\right) \right),
				 & j \leq k,\\
			\exp\left(-\exp(a_1x_2) \left( \gamma_0(t_{j-1}+s) + \gamma_1 \frac{(t_{j-1}+s)^2}{2}\right)\right) \\
			\hspace{2cm} -
			\exp\left(-\exp(a_1x_2) \left( \gamma_0(t_{j}+s) + \gamma_1 \frac{(t_{j}+s)^2}{2}\right)\right),
					&   k \leq j \leq L,\\
			\exp\left(-\exp(a_1x_2) \left( \gamma_0(t_j+s) + \gamma_1 \frac{(t_j+s)^2}{2}\right)\right),
				&  j=L+1 ,
		\end{cases}
	\end{equation}
where  the shifting time $s$ is a negative solution of Equation (\ref{eq:secondorderequation}). 
Because the MDPDE with tuning parameter $\beta$ is a minimum of the DPD loss, as defined in (\ref{eq:DPDloss}), it must annul the first derivatives of the loss. The following result presents the estimating equations of the MDPDE for interval-monitored step-stress ALTs under the assumption of linear proportional hazards.
\begin{theorem} \label{thm:estimatinglinear}
	Let us consider a sample $(n_1,...n_{L+1})$ from a simple interval-monitored step-stress ALT with $L$ inspection times and $N=\sum_{j=1}^{L+1}n_j $ devices under test. We define the non-parametric estimator of the probability as $\widehat{\boldsymbol{p}} = \left(\frac{n_1}{N},...,\frac{n_L}{N}, \frac{n_{L+1}}{N}\right).$ Then,
	the MDPDE  with parameter $\beta$ under the assumption of linear proportional hazards, $\widehat{\boldsymbol{\theta}}_\beta,$ must annul the $\beta$-score function defined  by
	\begin{equation} \label{eq:scorelinear}
		\boldsymbol{U}_\beta (\boldsymbol{\theta}) = \boldsymbol{W}(\boldsymbol{\theta})^T \boldsymbol{D}(\boldsymbol{\theta})^{\beta-1}\left(\widehat{\boldsymbol{p}} - \boldsymbol{\pi}(\boldsymbol{\theta})\right),
	\end{equation}
	where the matrix $\boldsymbol{W}(\boldsymbol{\theta})$ is the $4 \times (L+1)$ Jacobian matrix of the theoretical probability of failure $\boldsymbol{\pi}(\boldsymbol{\theta})$  as defined in (\ref{eq:th}),  with columns given by
	$$\boldsymbol{w}_j = \begin{cases}
		\frac{\partial R_{\boldsymbol{\theta}}(t_{j-1},x_1)}{\partial \boldsymbol{\theta}} - \frac{\partial R_{\boldsymbol{\theta}}(t_{j},x_1)}{\partial \boldsymbol{\theta}}, & t_{j-1} < \tau, \\
		\frac{\partial R_{\boldsymbol{\theta}}(t_{j-1},x_2)}{\partial \boldsymbol{\theta}} - \frac{\partial R_{\boldsymbol{\theta}}(t_{j},x_2)}{\partial \boldsymbol{\theta}}, & \tau \leq t_{j-1} ,
	\end{cases}$$
	where the four components of the gradient vector of the reliability function,
	$$\frac{\partial R_{\boldsymbol{\theta}}(t,x)}{\partial \boldsymbol{\theta}}  = \left(\frac{\partial R_{\boldsymbol{\theta}}(t,x)}{\partial \gamma_0},\frac{\partial R_{\boldsymbol{\theta}}(t,x)}{\partial \gamma_1},
	% \frac{\partial R_{\boldsymbol{\theta}}(t,x)}{\partial a_0}, 
	\frac{\partial R_{\boldsymbol{\theta}}(t,x)}{\partial a_1} \right)^T$$
 	are given, under the first stress level $x_1$, by
	\begin{center}
	\begin{equation*}
			\begin{aligned}
				\frac{\partial  R_{\boldsymbol{\theta}}(t,x_1)}{\partial \gamma_0} &= -	R_{\boldsymbol{\theta}}(t,x_1)
				\exp(a_1x_1) t ,\\
				\frac{\partial R_{\boldsymbol{\theta}}(t,x_1)}{\partial \gamma_1} &= -R_{\boldsymbol{\theta}}(t,x_1)
				\exp(a_1x_1) \frac{t^2}{2} ,\\
				%
				%\frac{\partial R_{\boldsymbol{\theta}}(t,x_1)}{\partial a_0} &= 	-R_{\boldsymbol{\theta}}(t,x_2) \exp(a_1x_1) tk_2(t),\\
				%
				\frac{\partial R_{\boldsymbol{\theta}}(t,x_1)}{\partial a_1} &= 	-R_{\boldsymbol{\theta}}(t,x_1)  \exp(a_1x_1) tk_2(t) x_2,\\
			\end{aligned}
	\end{equation*}
	\end{center}
	and under the increased stress level $x_2$ by
		\begin{center}
	\begin{equation*}
		\begin{aligned}
			\frac{\partial  R_{\boldsymbol{\theta}}(t,x_2)}{\partial \gamma_0} &= -	R_{\boldsymbol{\theta}}(t,x_2) 
			\exp(a_1x_2) \left( (t+s)+ 
			\frac{ \gamma_1}{2}\frac{(\tau+s) s}{k_2(\tau)}\frac{k_1(t+s)}{k_1(\tau+s)}
			\right),\\
			\frac{\partial R_{\boldsymbol{\theta}}(t,x_2)}{\partial \gamma_1} &= -R_{\boldsymbol{\theta}}(t,x_2)
			\exp(a_1x_2) \left(\frac{(t+s)^2}{2}-    \frac{\gamma_0}{2}\frac{(\tau+s)s}{k_2(\tau)} \frac{k_1(t+s)}{k_1(\tau+s)}          \right),\\
			%
			%\frac{\partial R_{\boldsymbol{\theta}}(t,x_2)}{\partial a_0} &= 	-R_{\boldsymbol{\theta}}(t,x_2) \exp(a_1x_2) (t+s)k_2(t+s),\\
			%
			\frac{\partial R_{\boldsymbol{\theta}}(t,x_2)}{\partial a_1} &= 	- R_{\boldsymbol{\theta}}(t,x_2) \exp(a_1x_2)\left[  (t+s)k_2(t+s) x_2
			+ \frac{k_1(t+s)}{k_1(\tau+s)} (\tau+s)k_2(\tau+s)(x_1-x_2) \right]\
		\end{aligned}
	\end{equation*}
	\end{center}
	and the functions $k_1(t)$ and $k_2(t)$ are $$k_1(t) = \gamma_1t+\gamma_0 \hspace{1cm}\text{and}\hspace{1cm} k_2(t) = \frac{\gamma_1}{2}t + \gamma_0.$$
	By $\boldsymbol{D}(\boldsymbol{\theta}),$ we denote a diagonal matrix with diagonal entries given by the probability vector $\boldsymbol{\pi}(\boldsymbol{\theta})$.
\end{theorem}
\begin{proof}
	See Appendix \ref{app:proof1}
\end{proof}

From the above result, the MPDPE can be defined as a solution of the system of equations
		$$	\boldsymbol{U}_\beta  (\boldsymbol{\theta}) = \boldsymbol{0}_4,$$
known as the estimating equations associated with the MDPDE, and thus  belongs to the family of M-estimators.
Moreover, we can derive the score function of the MLE by setting $\beta=0$ in Equation (\ref{eq:scorelinear}), yielding 
\begin{equation*} 
	\boldsymbol{U}_{\beta=0} (\boldsymbol{\theta}) = \boldsymbol{W}(\boldsymbol{\theta})^T \boldsymbol{D}(\boldsymbol{\theta})^{-1}\left(\widehat{\boldsymbol{p}} - \boldsymbol{\pi}(\boldsymbol{\theta})\right).
\end{equation*}
Notably,  the matrices $\boldsymbol{W}(\boldsymbol{\theta})$ and $\boldsymbol{D}(\boldsymbol{\theta})$ do not depend on the tuning parameter $\beta$ of the DPD loss, and thus the MDPDE estimating equations are a weighted version of the maximum likelihood estimating equations with weights equal to $\pi_j(\boldsymbol{\theta})^\beta,$ representing the probability  of failure within the $j$-th interval. Thus, the MDPDE down-weighs the observations on intervals with low probabilities. %; if an interval  appear to have higher-than-expected number of failures (outlying)

	\subsection{Quadratic baseline}
	
	In this subsection, we  derive the estimating equations and the asymptotic distribution of the MDPDE under quadratic baseline hazard. The primary equations can be obtained in a manner similar to the linear case, but it is important to note that the probabilities of failure (and consequently their derivatives) differ and must be computed separately. The parameter $\boldsymbol{\theta}$ under the assumption of quadratic hazard defined in (\ref{eq:baselinehazardquadratic}) refers to a $4-$dimensional vector with entries $(\gamma_0, \gamma_1, \gamma_2, a_1),$ the first three representing the baseline hazard and the  last entry representing the relationship of the lifetime distribution and the stress factor to which the units are subjected to.
	
	Using the notation in Section \ref{subsec:quadratichazards}, under quadratic baseline hazard, the probability of failure within an interval $(t_{j-1},t_j]$ is given by
		\begin{equation} \label{eq:thquadratic}
		\pi_j(\boldsymbol{\theta})  = 
		\begin{cases}
			\exp \left( -\exp(a_1x_1) \left( \gamma_0t_{j-1} + \gamma_1 \frac{t_{j-1}^2}{2} + \gamma_2\frac{t_{j-1}^2}{3}\right) \right)\\
			\hspace{2cm} -
			\exp \left( -\exp(a_1x_1) \left( \gamma_0t_j + \gamma_1 \frac{t_j^2}{2} + \gamma_2 t_{j}^{3}\right) \right),
			& j \leq k,\\
			\exp\left(-\exp(a_1x_2) \left( \gamma_0(t_{j-1}+s) + \gamma_1 \frac{(t_{j-1}+s)^2}{2} + \gamma_2 \frac{(t_{j-1}+s)^3}{3} \right)\right) \\
			\hspace{2cm} -
			\exp\left(-\exp(a_1x_2) \left( \gamma_0(t_{j}+s) + \gamma_1 \frac{(t_{j}+s)^2}{2} + \gamma_2\frac{(t_{j}+s)^3}{3}  \right)\right),
			&   k \leq j \leq L,\\
			\exp\left(-\exp(a_1x_2) \left( \gamma_0(t_j+s) + \gamma_1 \frac{(t_j+s)^2}{2} + \gamma_2 \frac{(t_j+s)^3}{3} \right) \right),
			&  j=L+1 ,
		\end{cases}
	\end{equation}
	where $s$ is the shifting time obtained as a solution of Equation (\ref{eq:thirdorderequation}).
	
	% Moreover, under the quadratic assumption the Jacobian matrix of the theoretical probability of failure vector, $\boldsymbol{\pi}(\boldsymbol{\theta})$ is given in
	\begin{theorem} \label{thm:estimatingquadratic}
		Given a sample $(n_1,...n_{L+1})$ from a simple interval-monitored step-stress ALT with $L$ inspection times and $N=\sum_{j=1}^{L+1}n_j $ devices under test,
		the MDPDE  with parameter $\beta$ under the assumption of quadratic proportional hazard, $\widehat{\boldsymbol{\theta}}_\beta,$ must annul the $\beta$-score function defined  by
		\begin{equation} \label{eq:score}
			\boldsymbol{U}_\beta (\boldsymbol{\theta}) = \boldsymbol{W}(\boldsymbol{\theta})^T \boldsymbol{D}(\boldsymbol{\theta})^{\beta-1}\left(\widehat{\boldsymbol{p}} - \boldsymbol{\pi}(\boldsymbol{\theta})\right),
		\end{equation}
		where $\widehat{\boldsymbol{p}} = \left(\frac{n_1}{N},...,\frac{n_L}{N}, \frac{n_{L+1}}{N}\right)^T$ is the empirical probability of failure,
		the matrix $\boldsymbol{W}(\boldsymbol{\theta})$ is the $5 \times (L+1)$ Jacobian matrix of the theoretical probability of failure $\boldsymbol{\pi}(\boldsymbol{\theta})$  defined in Equation (\ref{eq:th}),  with columns given by
		$$\boldsymbol{w}_j = \begin{cases}
			\frac{\partial R_{\boldsymbol{\theta}}(t_{j-1},x_1)}{\partial \boldsymbol{\theta}} - \frac{\partial R_{\boldsymbol{\theta}}(t_{j},x_1)}{\partial \boldsymbol{\theta}}, & t_{j-1} < \tau, \\
			\frac{\partial R_{\boldsymbol{\theta}}(t_{j-1},x_2)}{\partial \boldsymbol{\theta}} - \frac{\partial R_{\boldsymbol{\theta}}(t_{j},x_2)}{\partial \boldsymbol{\theta}}, & \tau \leq t_{j-1},
		\end{cases}$$
		where the five components of the gradient vector of the reliability function,
		$$\frac{\partial R_{\boldsymbol{\theta}}(t,x)}{\partial \boldsymbol{\theta}}  = \left(\frac{\partial R_{\boldsymbol{\theta}}(t,x)}{\partial \gamma_0},\frac{\partial R_{\boldsymbol{\theta}}(t,x)}{\partial \gamma_1}, \frac{\partial R_{\boldsymbol{\theta}}(t,x)}{\partial \gamma_2},
		% \frac{\partial R_{\boldsymbol{\theta}}(t,x)}{\partial a_0}, 
		\frac{\partial R_{\boldsymbol{\theta}}(t,x)}{\partial a_1} \right)^T$$
		are given, under the first stress level $x_1$, by
		\begin{center}
			\begin{equation*}
			\begin{aligned}
				\frac{\partial  R_{\boldsymbol{\theta}}(t,x_1)}{\partial \gamma_0} &= -	R_{\boldsymbol{\theta}}(t,x_1)
				\exp(a_1x_1) t ,\\
				\frac{\partial R_{\boldsymbol{\theta}}(t,x_1)}{\partial \gamma_1} &= -R_{\boldsymbol{\theta}}(t,x_1)
				\exp(a_1x_1) \frac{t^2}{2} ,\\
				\frac{\partial R_{\boldsymbol{\theta}}(t,x_1)}{\partial \gamma_2} &= -R_{\boldsymbol{\theta}}(t,x_1)
				\exp(a_1x_1) \frac{t^3}{3} ,\\
				%&\frac{\partial R_{\boldsymbol{\theta}}(t,x_1)}{\partial a_0} &= 	-R_{\boldsymbol{\theta}}(t,x_1) \exp(a_1x_1)tk_2(t),\\
				%
				\frac{\partial R_{\boldsymbol{\theta}}(t,x_1)}{\partial a_1} &= 	-R_{\boldsymbol{\theta}}(t,x_1) \exp(a_1x_1)tk_2(t) x_2,\\
			\end{aligned}
		\end{equation*}
		\end{center}
		and under the increased stress level $x_2$ by
		\begin{center}
			\begin{equation*}
			\begin{aligned}
				\frac{\partial  R_{\boldsymbol{\theta}}(t,x_2)}{\partial \gamma_0} &= -	R_{\boldsymbol{\theta}}(t,x_2) 
				\exp(a_1x_2) \left( (t+s)+  (\tau+s)\frac{k_1(t+s) }{k_1(\tau+s)}\left(\frac{k_2(\tau+s)}{k_2(\tau)}-1\right)
				\right),\\
				\frac{\partial R_{\boldsymbol{\theta}}(t,x_2)}{\partial \gamma_1} &= -R_{\boldsymbol{\theta}}(t,x_2)
				\exp(a_1x_2) \left(\frac{(t+s)^2}{2}+\frac{(\tau+s)}{2}\frac{ k_1(t+s)}{k_1(\tau+s)}\left(\tau\frac{k_2(\tau+s)}{k_2(\tau)}-(\tau+s)\right)          \right),\\
				\frac{\partial R_{\boldsymbol{\theta}}(t,x_2)}{\partial \gamma_2} &= -R_{\boldsymbol{\theta}}(t,x_2)
				\exp(a_1x_2) \left(\frac{(t+s)^3}{3}+  \frac{(\tau+s)}{3}\frac{k_1(t+s)}{k_1(\tau+s)}\left(\tau^2\frac{k_2(\tau+s)}{k_2(\tau)}-(\tau+s)^2\right)        \right),\\
				%
				%\frac{\partial R_{\boldsymbol{\theta}}(t,x_2)}{\partial a_0} &= 	-R_{\boldsymbol{\theta}}(t,x_2)\exp(a_1x_2)(t+s)k_2(t+s),\\
				%
				\frac{\partial R_{\boldsymbol{\theta}}(t,x_2)}{\partial a_1} &=  - R_{\boldsymbol{\theta}}(t,x_2) \exp(a_1x_2) \left[(t+s)k_2(t+s) x_2 +
				\frac{k_1(t+s)}{k_1(\tau+s)}(\tau+s)k_2(\tau+s)(x_1-x_2) \right],
			\end{aligned}
		\end{equation*}
		\end{center}
		and the functions $k_1(t)$ and $k_2(t)$ are $$k_1(t) = \gamma_2t^2+\gamma_1t+\gamma_0 \hspace{1cm}\text{and}\hspace{1cm} k_2(t) = \frac{\gamma_2}{3} t^2 + \frac{\gamma_1}{2}t + \gamma_0.$$
		Moreover,  $\boldsymbol{D}(\boldsymbol{\theta})$ denotes a diagonal matrix with diagonal entries given by the probability vector $\boldsymbol{\pi}(\boldsymbol{\theta})$.
	\end{theorem}
	\begin{proof}
		See Appendix \ref{app:proof2}
	\end{proof}
	
	As in the linear case, the score function of the MLE can be obtained by specializing at $\beta=0$  and the MDPDE estimating equations
		$ \boldsymbol{U}_\beta  (\boldsymbol{\theta}) = \boldsymbol{0}_5$
	 can be viewed as a weighted version of the efficient maximum likelihood score, with weights $\pi_j(\boldsymbol{\theta}).$
	That is, for positive values of the tuning parameter $\beta$, it provides a relative-to-the-model down-weighting for outlying observations; higher-than-expected failure counts that are considerably discrepant with respect to the model will get nearly zero weights. 
	In the most efficient case ($\beta=0$), all observations, including very severe outliers, get weights equal to one.

	\subsection{Asymptotic properties}
	
	We next derive the asymptotic distribution of the MDPDE under linear and quadratic baseline hazards. 
	This asymptotic convergence is a key result for two reasons: Firstly, it demonstrates the consistency of the estimator, and secondly it facilitates the derivation of asymptotic confidence intervals for parameters of interest.
	\begin{theorem} \label{thm:asymptotic}
		Let $\boldsymbol{\theta}^{\ast}$ be the true value of the parameter $\boldsymbol{\theta}$ and $\boldsymbol{\pi} (\boldsymbol{\theta}^{\ast})$ denote the true probability mass vector of the multinomial model.
		Then, the asymptotic distribution of the
		MDPDE, $\widehat{\boldsymbol{\theta}}^\beta,$ for the interval-monitored step-stress ALT model under PH, is given by
		$$\sqrt{N}\left(\widehat{\boldsymbol{\theta}}^\beta - \boldsymbol{\theta}^{\ast} \right) \xrightarrow[N\rightarrow \infty]{L} \mathcal{N} \left(\boldsymbol{0}, \boldsymbol{\Sigma}_\beta(\boldsymbol{\theta}^{\ast})\right),
		$$
		where the asymptotic variance-covariance matrix is given by
		\begin{equation}\label{eq:Sigmamatrix}
			\boldsymbol{\Sigma}_\beta(\boldsymbol{\theta}) = \boldsymbol{J}_\beta^{-1}(\boldsymbol{\theta})
			\boldsymbol{K}_\beta(\boldsymbol{\theta})
			\boldsymbol{J}_\beta^{-1}(\boldsymbol{\theta})
		\end{equation}
		with
		\begin{equation}
			\begin{aligned}
				\boldsymbol{J}_\beta(\boldsymbol{\theta}) &= \boldsymbol{W}(\boldsymbol{\theta})^T \boldsymbol{D}(\boldsymbol{\theta})^{\beta-1} \boldsymbol{W}(\boldsymbol{\theta}),\\
				%\hspace{0.3cm} \text{and} \hspace{0.3cm}
				\boldsymbol{K}_\beta(\boldsymbol{\theta}) &= \boldsymbol{W}(\boldsymbol{\theta})^T \left(\boldsymbol{D}(\boldsymbol{\theta})^{2\beta-1} - \boldsymbol{\pi}(\boldsymbol{\theta})^\beta (\boldsymbol{\pi}(\boldsymbol{\theta})^\beta)^T \right) \boldsymbol{W}(\boldsymbol{\theta}),\\
				%\boldsymbol{W}_\beta(\boldsymbol{\theta}) &= \frac{\partial \boldsymbol{\pi}(\boldsymbol{\theta})}{\partial \boldsymbol{\theta}}
				%\hspace{0.3cm} \text{and} \hspace{0.3cm}
				%\boldsymbol{D}(\boldsymbol{\theta}) = \operatorname{diag}\left(\boldsymbol{\pi} (\boldsymbol{\theta}) \right)
			\end{aligned}
		\end{equation}
		and $\boldsymbol{W}(\boldsymbol{\theta})$ is as defined  in Result \ref{thm:estimatinglinear} for the linear baseline hazard and as in Result \ref{thm:estimatingquadratic} for the quadratic baseline hazard.
	\end{theorem}
	
	\begin{proof}
		 The proof follows similar lines to the proof of Result 3 in \cite{balakrishnan2023robust}.
	\end{proof}
	%estimating equations
	%score function
	
	Because the MDPDE at $\beta=0$ coincides with the MLE, the inverse of the Fisher information matrix of the model is given by
	\begin{equation}
		\boldsymbol{I}_F^{-1}(\boldsymbol{\theta}) = 	\boldsymbol{\Sigma}_{\beta=0}(\boldsymbol{\theta}) = \boldsymbol{J}_{\beta=0}^{-1}(\boldsymbol{\theta})
		\boldsymbol{K}_{\beta=0}(\boldsymbol{\theta})
		\boldsymbol{J}_{\beta=0}^{-1}(\boldsymbol{\theta}) = \boldsymbol{W}(\boldsymbol{\theta})^T \boldsymbol{D}(\boldsymbol{\theta})^{-1} \boldsymbol{W}(\boldsymbol{\theta}).
	\end{equation}
	
	Moreover, from the above asymptotic distribution, the standard errors for all MDPDEs with tuning parameter $\beta$, $\widehat{\gamma}^\beta_i, i=0,1,2,$ for the baseline hazard parameter or $\widehat{a}^\beta_1$ for the log-linear relation with the stress level, can be obtained as the diagonal entries of the asymptotic variance-covariance matrix $\boldsymbol{\Sigma}_\beta(\boldsymbol{\theta})$ given in (\ref{eq:Sigmamatrix}). 
	These standard errors will be denoted by $$\sigma^2_\beta(\gamma_i^\ast) = \boldsymbol{\Sigma}_\beta(\boldsymbol{\theta}^\ast)_{i+1,i+1}, i=0,1,2,$$
	for the baseline hazard parameters, and 
	%  \hspace{0.3cm}\text{and} \hspace{0.3cm} 
	$$\sigma^2_\beta(a_1^\ast) = \boldsymbol{\Sigma}_\beta(\boldsymbol{\theta}^\ast)_{3,3}  \hspace{0.3cm}\text{or} \hspace{0.3cm} \sigma^2_\beta(a_1^\ast) = \boldsymbol{\Sigma}_\beta(\boldsymbol{\theta}^\ast)_{4,4}$$ 
	for the log-linear relationship parameters under linear and quadratic baseline hazard functions, respectively.
	A consistent estimator of these standard  errors can be estimated by plugging-in any consistent estimator of the true parameter. Here, we will choose the same value of the tuning parameter $\beta$ for the MDPDE, ensuring robust estimation of the standard estimation error for positive values of the tuning parameter. However, any other estimator could be used instead. For example, for a more efficient (but non-robust) estimation, we could make use of  the MLE for estimating the standard error of any MDPDE.
	
	From the above discussion, we can build approximate $100(1-\alpha)\%$ confidence intervals for the baseline hazard parameters $\gamma_i$ (for $i=0,1, 2$) and the stress-factor relationship $\widehat{a}^\beta_1$ as follows.
	For the baseline hazard parameters, $\gamma_i,$ for $i=0,1, 2,$ the approximate confidence intervals are given by
	$$IC_{1-\alpha}\left[ \gamma_i \right] = \left[\widehat{\gamma}_i^\beta \pm z_{\alpha/2} \sigma_\beta(\widehat{\gamma}_i^\beta)\right],$$
	and similarly,	for the stress-factor relationship  parameter, $\widehat{a}^\beta_1,$  we have
	$$IC_{1-\alpha}\left[ a_1 \right] = \left[\widehat{a}_1^\beta \pm z_{\alpha/2} \sigma_\beta(\widehat{a}_1^\beta)\right].$$
	Here, $z_{\alpha}$ denotes the upper $\alpha$-quantile of a standard normal distribution.
	If any of these estimated intervals extend beyond the boundaries of a parameter space, such as the requirement for baseline hazard parameters to be positive, the extreme of intervals need to be truncated.

%	\end{theorem}

	% asymptotic
	%estimating equations
	%score function

	\section{Point estimation of lifetime characteristics \label{sec:char}}
	
	Many reliability analyses aim to estimate specific lifetime characteristics rather than the complete lifetime distribution function. Specifically, inferring the median, mean lifetime, distribution quantiles, and reliability at fixed mission times are often of interest under NOC.
	Once the model parameters $\boldsymbol{\theta} = (\boldsymbol{\gamma}, \boldsymbol{a})$  have been estimated, these quantities can be readily estimated by substituting the parameter estimates into the distribution function and then computing the corresponding lifetime characteristic of interest.
	As we assume general linear or quadratic baseline hazard functions, the derivation of explicit expressions for the previously mentioned lifetime characteristics under both baseline forms  is necessary. In this section, we develop explicit expressions for estimating the mean lifetime, distribution quantiles (including the median of the distribution), and reliability at fixed mission times under NOC based on the MDPDEs. We also present their asymptotic properties, inherited  from the robust estimators, as well as asymptotic confidence intervals.
	
	\subsection{Linear baseline hazard} \label{sec:char_linear}
	
	We first assume linear baseline hazard function defined in Equation (\ref{eq:baselinehazardlinear}), with unknown parameters $\boldsymbol{\gamma} = (\gamma_0,\gamma_1).$ The cumulative distribution function of the lifetime $T$ under normal operating conditions $x_0$ is then given by
	\begin{equation} \label{eq:cdflinear}
		F_T(t) = 	1- \exp\left(-\exp\left( a_1 x_0\right) \left( \gamma_0t + \gamma_1 \frac{t^2}{2}\right)\right), \hspace{0.3cm} 0 < t < \infty,
	\end{equation}
	where $\exp\left( a_1 x_0\right)$ represents the multiplicative effect of the stress level under NOC.
	
	\subsubsection{Mean lifetime to failure}	
	
	The mean lifetime is often a key parameter in reliability analyses because it provides valuable insight into the performance and durability of the considered product. It may play an important role in decision-making, risk assessment, and quality control across various industries.
	 Using the expression of the c.d.f. in (\ref{eq:cdflinear}), the mean of the lifetime of the product is given by
	\begin{equation} \label{eq:meanlinear}
	%	\begin{aligned}
			E_{\boldsymbol{\theta}}(T) %&= 	\int_0^\infty \exp\left(-\exp\left( a_1 x_0\right) \left( \gamma_0t + \gamma_1 \frac{t^2}{2}\right)\right) dt \\
			= \frac{1}{2} \sqrt{\frac{2\pi}{\gamma_1 \exp(a_1x_0)}} \exp\left(\frac{\exp(a_1x_0)\gamma_0^2}{2\gamma_1}\right).
	%	\end{aligned}
	\end{equation}
Detailed calculations for the mean lifetime are provided in  Appendix \ref{app:meanlinear}. Consequently, the MDPDE  of the mean lifetime with parameter $\beta$ can be computed readily as
	\begin{equation*} \label{eq:meanlinearest}
		E_{\widehat{\boldsymbol{\theta}}^\beta}[T] = \frac{1}{2} \sqrt{\frac{2\pi}{\widehat{\gamma}^\beta_1 \exp(\widehat{a}^\beta_1x_0)}} \exp\left(\frac{\exp(\widehat{a}^\beta_1x_0)(\widehat{\gamma}^\beta_0)^2}{2\widehat{\gamma}^\beta_1}\right),
\end{equation*}
where $\widehat{\boldsymbol{\theta}}^\beta = (\widehat{\gamma}^\beta_0, \widehat{\gamma}^\beta_1, \widehat{a}^\beta_1)$ is the MDPDE of the model parameter $\boldsymbol{\theta}.$
As the estimated mean lifetime can be expressed in terms of a differentiable function of the MDPDE,
the asymptotic distribution of the MDPDE of the mean lifetime, $E_T(\widehat{\boldsymbol{\theta}}^\beta),$ can be obtained directly by applying the Delta-method. In particular, it can be shown that the MDPDE of the mean lifetime is such that
$$
\sqrt{N}\left(E_{\widehat{\boldsymbol{\theta}}^\beta}[T]- E_{\boldsymbol{\theta}^\ast}[T] \right) \xrightarrow[N\rightarrow \infty]{L} \mathcal{N} \left(\boldsymbol{0}, \sigma_\beta(E_T(\boldsymbol{\theta}^{\ast}))\right),
$$
where $\boldsymbol{\theta}^{\ast} = (\gamma_0^{\ast}, \gamma_1^{\ast},  a_1^{\ast})$ is the true value of the model parameters,
$$ \sigma^2_\beta(E_T(\boldsymbol{\theta}^{\ast})) = \left[\frac{\partial E_{\boldsymbol{\theta}}[T] }{\partial \boldsymbol{\theta}}\right]^T_{\boldsymbol{\theta} = \boldsymbol{\theta}^{\ast}} \boldsymbol{\Sigma}_\beta(\boldsymbol{\theta}^{\ast}) \left[\frac{\partial E_{\boldsymbol{\theta}}[T] }{\partial \boldsymbol{\theta}}\right]_{\boldsymbol{\theta} = \boldsymbol{\theta}^{\ast}},$$
with $\boldsymbol{\Sigma}_\beta(\boldsymbol{\theta}) $ being defined in Equation (\ref{eq:Sigmamatrix}) 
and 
\begin{equation}
	\begin{aligned}
	\left[\frac{\partial E_{\boldsymbol{\theta}}[T] }{\partial \boldsymbol{\theta}}\right]_{\boldsymbol{\theta} = \boldsymbol{\theta}^{\ast}} =&\frac{1}{2}E_T(\boldsymbol{\theta}^{\ast})
	\left(\frac{\gamma^{\ast}_0}{\gamma^{\ast}_1} \exp(+a^{\ast}_1x_0), \frac{1}{\gamma^{\ast}_1}\left[-1 - \frac{\gamma_0^2}{\gamma_1}\exp(+a^{\ast}_1x_0)\right]\right.,\\
	&\hspace{2cm}
%	  -1 + \frac{(\gamma^{\ast}_0)^2}{\gamma^{\ast}_1}\exp(+a^{\ast}_1x_0) ,
	\left. \left[ -1 + \frac{(\gamma^{\ast}_0)^2}{\gamma^{\ast}_1}\exp(+a^{\ast}_1x_0)\right]x_0 \right)^T.
\end{aligned}
\end{equation}
Explicit derivation of the above derivatives are given in Appendix \ref{app:meanlinear}. As the MDPDE $\widehat{\boldsymbol{\theta}}$ is a consistent estimator of the true model parameter, the standard error of the estimate can be approximated by  $\sigma^2_\beta(E_{\widehat{\boldsymbol{\theta}}_\beta}[T])$

Now, based on the asymptotic distribution of the MDPDE of the mean lifetime, an approximate $100(1-\alpha)\%$-confidence interval for the mean lifetime can be obtained as
$$IC_{1-\alpha}\left[E_{\boldsymbol{\theta}}[T] \right] = \left[ E_{\widehat{\boldsymbol{\theta}}_\beta}[T] \pm z_{\alpha/2} \sigma_\beta(E_{\widehat{\boldsymbol{\theta}}_\beta}[T])\right].$$
Because the mean lifetime is a positive quantity, if the lower bound of the approximate confidence interval is negative, it should be truncated to $0.$ 

\subsubsection{Reliability}

Estimating the reliability of the product at a certain mission time may help manufacturers to ensure product performance, safety or cost-effectiveness of the product at a particular time horizon of interest.
Let us consider the reliability function of the product at constant stress $x_0$ under the assumption of linear baseline hazard function as in (\ref{eq:reliabilitylinearhazard}). The estimated reliability of the product based on the MDPDE $\widehat{\boldsymbol{\theta}}_\beta$ is then given by
\begin{equation} \label{eq:estimatedreliabilitylinearhazard}
	R_{\widehat{\boldsymbol{\theta}}_\beta}(t_0) = 
		\exp\left(-\exp\left(\widehat{a}_1^\beta x_0\right) \left( \widehat{\gamma}^\beta_0t_0 + \widehat{\gamma}^\beta_1 \frac{t_0^2}{2}\right)\right), \hspace{0.5cm} 0 < t < \infty,
\end{equation}
where $t_0$ denotes the fixed mission time. 

To derive approximate confidence intervals for the product's reliability at a specific mission time, we can employ the asymptotic distribution of the function $R_{\widehat{\boldsymbol{\theta}}_\beta}(t_0).$ This distribution can be derived using the Delta method as follows:
$$
\sqrt{N}\left(R_{\widehat{\boldsymbol{\theta}}_\beta}(t_0)- R_{\boldsymbol{\theta}^\ast}(t_0) \right) \xrightarrow[N\rightarrow \infty]{L} \mathcal{N} \left(\boldsymbol{0}, \sigma_\beta(R_{\boldsymbol{\theta}^\ast}(t_0))\right),
$$
where $\boldsymbol{\theta}^{\ast} = (\gamma_0^{\ast}, \gamma_1^{\ast}, a_1^{\ast})$ is the true value of the model parameters,
$$ \sigma^2_\beta(R_{\boldsymbol{\theta}^\ast}(t_0)) = \left[\frac{\partial R_{\boldsymbol{\theta}}(t_0) }{\partial \boldsymbol{\theta}}\right]^T_{\boldsymbol{\theta} = \boldsymbol{\theta}^{\ast}} \boldsymbol{\Sigma}_\beta(\boldsymbol{\theta}^{\ast}) \left[\frac{\partial R_{\boldsymbol{\theta}}(t_0) }{\partial \boldsymbol{\theta}}\right]_{\boldsymbol{\theta} = \boldsymbol{\theta}^{\ast}},$$
with $\boldsymbol{\Sigma}_\beta(\boldsymbol{\theta}) $ being as defined in Equation (\ref{eq:Sigmamatrix}) 
and 
\begin{equation}
	\frac{\partial R_{\boldsymbol{\theta}}(t_0)}{\partial \boldsymbol{\theta}} = -R_{\boldsymbol{\theta}}(t_0)\exp\left(a_0^\beta + a_1 x_0\right)
	\left(   t_0,  \frac{t_0^2}{2},  \left( \gamma_0t_0 + \gamma_1 \frac{t_0^2}{2}\right)x_0 \right)^T.
\end{equation} 

The standard error of the reliability estimate, denoted above as $\sigma_\beta(R_{\boldsymbol{\theta}^\ast}(t_0)),$ can be robustly and consistently estimated using a MDPDE. We can then build an approximate $100(1-\alpha)\%$ confidence interval as follows:
$$IC_{1-\alpha}\left[ R_{\boldsymbol{\theta}}(t_0) \right] = \left[ R_{\widehat{\boldsymbol{\theta}}_\beta}(t_0) \pm z_{\alpha/2} \sigma_\beta(R_{\widehat{\boldsymbol{\theta}}_\beta}(t_0))\right],$$
where $z_{\alpha}$ denotes the upper $\alpha$ quantile of a standard normal distribution.
Given that reliability values fall within the interval $[0,1]$, if  either of the approximated interval bounds exceed this range, they should be truncated accordingly.

\subsubsection{Quantiles} \label{sec:char_quantile_linear}

Certain distribution quantiles, particularly those representing the lower and upper tails, offer valuable insights into the characteristics of the tails of the lifetime distribution. Besides, the median life serves as a crucial measure of central tendency, as  it gets less affected by extreme values, unlike the mean lifetime. This makes the median a more robust measure of centrality than the mean lifetime to failure. Hence, estimating quantiles for the lower tail, upper tail, and central tendency can help manufacturers in comprehending the overall behaviour of the lifetimes of the products. Here, we will develop robust estimators for an $\alpha$-quantile based on the MDPDEs.

The $\alpha$-quantile of the lifetime distribution represents the time at which $100\alpha\%$ of the devices would have failed. Under the PH hazard assumption with linear baseline function and constant stress $x_0$, the quantile function is obtained by inverting the distribution function in (\ref{eq:cdflinear}), that is,
$
	Q_{\boldsymbol{\theta}}(\alpha) = F_T^{-1}(\alpha).
$
Therefore, equating the distribution function to $\alpha$, we obtain the $\alpha$-quantile function to satisfy the following equation:
%$$\alpha = \exp\left(-\exp\left(\widehat{a}_1^\beta x_0\right) \left( \widehat{\gamma}^\beta_0t + \widehat{\gamma}^\beta_1 \frac{t^2}{2}\right)\right)$$
\begin{equation}  \label{eq:quantileec}
	\frac{\gamma_1 }{2}Q_{\boldsymbol{\theta}}(\alpha)^2 + \gamma_0Q_{\boldsymbol{\theta}}(\alpha) + \log(\alpha)\exp\left( -a_1^\beta x_0\right) = 0. 
\end{equation}

Note that the  above equation  always has a real solution due to the constraint $0 < \alpha < 1,$ which ensures that the term $\log(\alpha)$ is negative. Moreover, all remaining quantities $\gamma_0$, $\gamma_1$ and $\exp(a_1x_0)$ are positive, and so  the discriminant of the second-order solution formula is positive and the solutions of the equation are therefore real. Then, solving for $Q_{\boldsymbol{\theta}}(\alpha),$ we explicitly obtain the $\alpha$-quantile as
\begin{equation}
	Q_{\boldsymbol{\theta}}(\alpha) = \frac{-\gamma_0 + \sqrt{\gamma_0^2-2\gamma_1 \log(\alpha)\exp\left( -a_1^\beta x_0\right)}}{\gamma_1}.
\end{equation}
Although a second-order equation may have up to two solutions, in the context of quantiles where the result should be a positive quantity, we have selected the unique positive solution provided by the above equation.

Now, applying the delta-method, we can readily stablish the asymptotic distribution of the $\alpha-$quantile as follows:
$$
\sqrt{N}\left(Q_{\widehat{\boldsymbol{\theta}}_\beta}(\alpha)- Q_{\boldsymbol{\theta}^\ast}(\alpha) \right) \xrightarrow[N\rightarrow \infty]{L} \mathcal{N} \left(\boldsymbol{0}, \sigma_\beta(Q_{\boldsymbol{\theta}^\ast}(\alpha))\right),
$$
where $\boldsymbol{\theta}^{\ast} = (\gamma_0^{\ast}, \gamma_1^{\ast}, a_1^{\ast})$ is the true value of the model parameters,
$$ \sigma^2_\beta(Q_{\boldsymbol{\theta}^\ast}(\alpha)) = \left[\frac{\partial Q_{\boldsymbol{\theta}}(\alpha) }{\partial \boldsymbol{\theta}}\right]^T_{\boldsymbol{\theta} = \boldsymbol{\theta}^{\ast}} \boldsymbol{\Sigma}_\beta(\boldsymbol{\theta}^{\ast}) \left[\frac{\partial Q_{\boldsymbol{\theta}}(\alpha) }{\partial \boldsymbol{\theta}}\right]_{\boldsymbol{\theta} = \boldsymbol{\theta}^{\ast}},$$
with $\boldsymbol{\Sigma}_\beta(\boldsymbol{\theta}) $ being as defined in Equation (\ref{eq:Sigmamatrix}) 
and 
%\begin{equation}
%		\frac{\partial Q_{\boldsymbol{\theta}}(\alpha)}{\partial \boldsymbol{\theta}} =  - \frac{ Q_{\boldsymbol{\theta}}(\alpha) }{\gamma_1  Q_{\boldsymbol{\theta}}(\alpha)  + \gamma_0}\left(1,  \frac{Q_{\boldsymbol{\theta}}(\alpha)}{2}, \frac{\gamma_1 }{2}Q_{\boldsymbol{\theta}}(\alpha) + \gamma_0, \frac{\gamma_1 }{2}Q_{\boldsymbol{\theta}}(\alpha)x_0 + \gamma_0x_0 \right)^T,
%\end{equation}
\begin{equation}
	\frac{\partial Q_{\boldsymbol{\theta}}(\alpha)}{\partial \boldsymbol{\theta}} =  - \frac{ Q_{\boldsymbol{\theta}}(\alpha) }{k_1 (Q_{\boldsymbol{\theta}}(\alpha))}\left(1,  \frac{Q_{\boldsymbol{\theta}}(\alpha)}{2}, k_2(Q_{\boldsymbol{\theta}}(\alpha))x_0 \right)^T,
\end{equation}
with $k_1(t) = \gamma_1t+\gamma_0 $ and $ k_2(t) = \frac{\gamma_1}{2}t + \gamma_0.$

The explicit forms of the derivatives of the $\alpha$-quantile have been obtained by taking implicit derivatives of Equation (\ref{eq:quantileec}) with respect to each parameter and subsequently solving each corresponding equation. Detailed calculations are provided in Appendix \ref{app:quantilelinear}.
Estimating the standard errors of the estimates by plugging-in the MDPDE as before,  an approximate $100(1-\alpha)\%$ confidence interval can be given as follows:
$$IC_{1-\alpha}\left[Q_{\boldsymbol{\theta}}(\alpha_0) \right] = \left[ Q_{\widehat{\boldsymbol{\theta}}_\beta}(\alpha_0) \pm z_{\alpha/2} \sigma_\beta(Q_{\widehat{\boldsymbol{\theta}}_\beta}(\alpha_0))\right],$$
where $z_{\alpha}$ denotes the upper $\alpha$ quantile of a standard normal distribution.
As the domain of the distribution function is restricted to positive values, any approximate lower bound that falls below zero needs to be be truncated. 

\subsection{Quadratic baseline hazard} \label{sec:char_quadratic}

We will now discuss in detail the case of quadratic baseline hazard. Estimating the lifetime characteristics under quadratic baseline hazard is a more complex task, as not all characteristics have an explicit expression, and they need to be computed numerically. Consequently, while adopting a quadratic lifetime assumption provides a more flexible statistical model, the estimation of lifetime characteristics poses a big challenge. 
Let us assume a quadratic hazard function as in Equation (\ref{eq:baselinehazardquadratic}), with unknown parameters $\boldsymbol{\theta} = \left(\gamma_0, \gamma_1, \gamma_2, a_1 \right).$ The cumulative distribution function of the lifetime $T$ under a constant stress $x_0$ is then given by
	\begin{equation} \label{eq:cdfquadratic}
	F_T(t) = 	1- \exp\left(-\exp\left( a_1 x_0\right) \left( \gamma_0t + \gamma_1 \frac{t^2}{2} + \gamma_2 \frac{t^3}{3}\right)\right), \hspace{0.3cm} 0 < t < \infty,
\end{equation}
where $\exp\left( a_1 x_0\right)$ represents the multiplicative effect of the stress level.  
Consequently, the probability density function of the lifetime is given  by
\begin{equation} \label{eq:pdfquadratic}
	f_T(t) = 	\exp\left( a_1 x_0\right) \left( \gamma_0 + \gamma_1 t + \gamma_2 t^2 \right) \exp\left(-\exp\left( a_1 x_0\right) \left( \gamma_0t + \gamma_1 \frac{t^2}{2} + \gamma_2 \frac{t^3}{3}\right)\right).
\end{equation}

\subsubsection{Mean lifetime to failure}

The mean lifetime of a device under the quadratic hazard assumption and constant stress level $x_0$ can be computed as
\begin{equation} \label{eq:ET}
	E_{\boldsymbol{\theta}}[T] = \exp\left( a_1 x_0\right) \int_{0}^{\infty}  \left( \gamma_0t + \gamma_1 t^2 + \gamma_2 t^3 \right) \exp\left(-\exp\left( a_1 x_0\right) \left( \gamma_0t + \gamma_1 \frac{t^2}{2} + \gamma_2 \frac{t^3}{3}\right)\right) dt.
\end{equation}

The definite integral in (\ref{eq:ET}) does not have a closed form and so, given a certain vector model parameter,  it needs to be estimated numerically.
In particular, the MDPDE of the mean lifetime for a fixed $\beta,$ $E_{\widehat{\boldsymbol{\theta}}_\beta}[T],$ can be obtained by plugging-in the MDPDE estimator $\widehat{\boldsymbol{\theta}}_\beta$ into the previous expression and subsequently approximating its value numerically.
To compute the improper integrals with infinite domain through Monte Carlo techniques, the infinite integrand interval can be mapped to a finite interval  and then  any regular numerical quadrature routine may be used for the modified finite integral, or use Gaussian quadrature rules.
Although the mean lifetime to failure does not have an explicit form, it is a differentiable function of $\boldsymbol{\theta}$ and its derivatives can be computed using derivative under integral sign rules. Therefore,
its asymptotic distribution can be obtained through the Delta-method, as in the linear case. In particular, we can establish the following asymptotic result:
$$
\sqrt{N}\left(E_{\widehat{\boldsymbol{\theta}}_\beta}[T]- E_{\boldsymbol{\theta}^\ast}[T] \right) \xrightarrow[N\rightarrow \infty]{L} \mathcal{N} \left(\boldsymbol{0}, \sigma_\beta E_{\boldsymbol{\theta}^\ast}[T]\right),
$$
where $\boldsymbol{\theta}^{\ast} = (\gamma_0^{\ast}, \gamma_1^{\ast}, a_0^{\ast}, a_1^{\ast})$ is the true value of the model parameters,
$$ \sigma^2_\beta(E_{\boldsymbol{\theta}^\ast}[T]) = \left[\frac{\partial E_{\boldsymbol{\theta}}[T] }{\partial \boldsymbol{\theta}}\right]^T_{\boldsymbol{\theta} = \boldsymbol{\theta}^{\ast}} \boldsymbol{\Sigma}_\beta(\boldsymbol{\theta}^{\ast}) \left[\frac{\partial E_{\boldsymbol{\theta}}[T] }{\partial \boldsymbol{\theta}}\right]_{\boldsymbol{\theta} = \boldsymbol{\theta}^{\ast}},$$
with $\boldsymbol{\Sigma}_\beta(\boldsymbol{\theta}) $ being as in Equation (\ref{eq:Sigmamatrix}) 
and $\frac{\partial E_{\boldsymbol{\theta}}[T] }{\partial \boldsymbol{\theta}} = \left(\frac{\partial E_{\boldsymbol{\theta}}[T] }{\partial \gamma_0}, \frac{\partial E_{\boldsymbol{\theta}}[T] }{\partial \gamma_1}, \frac{\partial E_{\boldsymbol{\theta}}[T] }{\partial \gamma_2}, \frac{\partial E_{\boldsymbol{\theta}}[T] }{\partial a_1} \right)^T$
is the vector of derivatives with
\begin{equation}
	\begin{aligned}
	\frac{\partial E_{\boldsymbol{\theta}}[T] }{\partial \gamma_0} & = \exp\left( a_1 x_0\right) \int_{0}^{\infty}  t(1-th_{\boldsymbol{\theta}}(t))R_{\boldsymbol{\theta}}(t)dt, \\
	\frac{\partial E_{\boldsymbol{\theta}}[T] }{\partial \gamma_1} & = \exp\left( a_1 x_0\right) \int_{0}^{\infty}  t^2(1-\frac{t}{2})h_{\boldsymbol{\theta}}(t) )R_{\boldsymbol{\theta}}(t)dt,\\
	\frac{\partial E_{\boldsymbol{\theta}}[T] }{\partial \gamma_2} & = \exp\left( a_1 x_0\right) \int_{0}^{\infty}  t^3(1-\frac{t}{3}h_{\boldsymbol{\theta}}(t))R_{\boldsymbol{\theta}}(t)dt,\\
	%
	%\frac{\partial E_{\boldsymbol{\theta}}[T] }{\partial a_0} &=  (1-\exp\left( a_1 x_0\right)) E_{\boldsymbol{\theta}}[T]\\
	%
	\frac{\partial E_{\boldsymbol{\theta}}[T] }{\partial a_1} &=   (1-\exp\left( a_1 x_0\right)) E_{\boldsymbol{\theta}}[T]  x_0.
	\end{aligned}
\end{equation}
Detailed calculations of the above derivatives are given in Appendix \ref{app:meanquadratic}.
Given a MDPDE $\widehat{\boldsymbol{\theta}}_\beta,$ we can numerically approximate 
the estimated mean lifetime $E_{\widehat{\boldsymbol{\theta}}_\beta}[T]$ and estimated standard error
$\sigma_\beta(E_{\widehat{\boldsymbol{\theta}}_\beta}[T])$ and consequently
 an approximate $100(1-\alpha)\%$-confidence interval for the mean lifetime can be given as
$$IC_{1-\alpha}\left[E_{\boldsymbol{\theta}}[T] \right] = \left[ E_{\widehat{\boldsymbol{\theta}}_\beta}[T] \pm z_{\alpha/2} \sigma_\beta(E_{\widehat{\boldsymbol{\theta}}_\beta}[T])\right].$$
Again, it's important to note that the mean lifetime is a positive value. Therefore, if the lower bound of the approximate confidence interval happens to be negative, it should be  truncated to a minimum value of 0.

	\subsubsection{Reliability}
	
	The reliability function of the lifetime of the device under quadratic hazard at a mission time $t_0$ and constant stress level $x_0$ is given by
	\begin{equation*} %\label{eq:cdfquadratic}
		R_{\boldsymbol{\theta}}(t_0) = 	\exp\left(-\exp\left( a_1 x_0\right) \left( \gamma_0t_0 + \gamma_1 \frac{t_0^2}{2} + \gamma_2 \frac{t_0^3}{3}\right)\right).
	\end{equation*}
	Therefore, the MDPDE of the reliability with tuning parameter $\beta$ is defined as
	\begin{equation*} %\label{eq:cdfquadratic}
		R_{\widehat{\boldsymbol{\theta}}_\beta}(t_0) = 	\exp\left(-\exp\left(\widehat{a}^\beta_1 x_0\right) \left( \widehat{\gamma}^\beta_0t_0 + \widehat{\gamma}^\beta_1 \frac{t_0^2}{2} + \widehat{\gamma}^\beta_2 \frac{t_0^3}{3}\right)\right).
	\end{equation*}
	where $\widehat{\boldsymbol{\theta}}^\beta = (\widehat{\gamma}_0^\beta, \widehat{\gamma}_1^\beta,\widehat{\gamma}_2^\beta,\widehat{a}_1^\beta)^T$ is the MDPDE of the parameter $\boldsymbol{\theta}.$ The properties of the MDPDE for the reliability, $R_{\widehat{\boldsymbol{\theta}}_\beta}(t_0),$ will be inherited from the properties of $\widehat{\boldsymbol{\theta}}^\beta,$  and positive values of the tuning parameter will provide robust estimates of the reliability. Applying the Delta method to the function $R_{\boldsymbol{\theta}}(t_0),$ we obtain
	$$\sqrt{N}\left(R_{\widehat{\boldsymbol{\theta}}_\beta}(t_0)- R_{\boldsymbol{\theta}^\ast}(t_0) \right) \xrightarrow[N\rightarrow \infty]{L} \mathcal{N} \left(\boldsymbol{0}, \sigma_\beta(R_{\boldsymbol{\theta}^\ast}(t_0))\right),
	$$
	where $\boldsymbol{\theta}^{\ast} = (\gamma_0^{\ast}, \gamma_1^{\ast}, a_0^{\ast}, a_1^{\ast})$ is the true value of the model parameters,
	$$ \sigma^2_\beta(R_{\boldsymbol{\theta}^\ast}(t_0)) = \left[\frac{\partial R_{\boldsymbol{\theta}}(t_0) }{\partial \boldsymbol{\theta}}\right]^T_{\boldsymbol{\theta} = \boldsymbol{\theta}^{\ast}} \boldsymbol{\Sigma}_\beta(\boldsymbol{\theta}^{\ast}) \left[\frac{\partial R_{\boldsymbol{\theta}}(t_0) }{\partial \boldsymbol{\theta}}\right]_{\boldsymbol{\theta} = \boldsymbol{\theta}^{\ast}},$$
	with $\boldsymbol{\Sigma}_\beta(\boldsymbol{\theta}) $ being defined in Equation (\ref{eq:Sigmamatrix}) 
	and 
	\begin{equation}
		\frac{\partial R_{\boldsymbol{\theta}}(t_0)}{\partial \boldsymbol{\theta}} = -R_{\boldsymbol{\theta}}(t_0)\exp\left( a_1 x_0\right)
		\left(t_0,  \frac{t_0^2}{2}, \frac{t_0^3}{3},  \left( \gamma_0t_0 + \gamma_1 \frac{t_0^2}{2} \gamma_2 \frac{t_0^3}{3}\right)x_0,   \right)^T.
	\end{equation} 
	 Therefore, we can build an approximate $100(1-\alpha)\%$ confidence interval as 
	$$IC_{1-\alpha}\left[ R_{\boldsymbol{\theta}}(t_0) \right] = \left[ R_{\widehat{\boldsymbol{\theta}}_\beta}(t_0) \pm z_{\alpha/2} \sigma_\beta(R_{\widehat{\boldsymbol{\theta}}_\beta}(t_0))\right],$$
	where $z_{\alpha}$ denotes the upper $\alpha$ quantile of a standard normal distribution and $\sigma_\beta(R_{\widehat{\boldsymbol{\theta}}_\beta}(t_0))$ is a robust (for $\beta >0$) and consistent estimator of the standard error.
	As before, the approximate interval bounds need to be truncated to $[0,1]$ if their estimated values exceed these values, as the reliability of a device can not exceed these limits.

	\subsubsection{Quantiles} \label{sec:char_quantile_quadratic}
	
	The upper $\alpha$-quantile of the cumulative distribution function under quadratic baseline hazard, denoted by $Q_{\boldsymbol{\theta}}(\alpha),$ should satisfy the equation
	\begin{equation*} 
		1- \alpha  = 	1- \exp\left(-\exp\left( a_1 x_0\right) \left( \gamma_0Q_{\boldsymbol{\theta}}(\alpha) + \gamma_1 \frac{Q_{\boldsymbol{\theta}}(\alpha)^2}{2} + \gamma_2 \frac{Q_{\boldsymbol{\theta}}(\alpha)^3}{3}\right)\right),
	\end{equation*}
	and so, after some algebra, we can establish that the upper $\alpha$-quantile of the lifetime must satisfy the equation 
	\begin{equation*} 
	    \frac{ \gamma_2}{3}Q_{\boldsymbol{\theta}}(\alpha)^3 + \frac{\gamma_1}{2} Q_{\boldsymbol{\theta}}(\alpha)^2 +\gamma_0Q_{\boldsymbol{\theta}}(\alpha)  + \log(\alpha) \exp\left(- a_1 x_0\right) = 0.
	\end{equation*}
	This equation may have up to three real solutions (and at least one). We will choose the least positive real solution. 
	While we do not have  an explicit formula for the $\alpha$-quantile, it is a differentiable function with respect to $\boldsymbol{\theta}$ and so 
	we can apply the Delta-method to derive its asymptotic distribution, yielding  the  result:
	$$
	\sqrt{N}\left(Q_{\widehat{\boldsymbol{\theta}}_\beta}(\alpha)- Q_{\boldsymbol{\theta}^\ast}(\alpha) \right) \xrightarrow[N\rightarrow \infty]{L} \mathcal{N} \left(\boldsymbol{0}, \sigma_\beta(Q_{\boldsymbol{\theta}^\ast}(\alpha))\right),
	$$
	where $\boldsymbol{\theta}^{\ast} = (\gamma_0^{\ast}, \gamma_1^{\ast}, a_0^{\ast}, a_1^{\ast})$ is the true value of the model parameters,
	$$ \sigma^2_\beta(Q_{\boldsymbol{\theta}^\ast}(\alpha)) = \left[\frac{\partial Q_{\boldsymbol{\theta}}(\alpha) }{\partial \boldsymbol{\theta}}\right]^T_{\boldsymbol{\theta} = \boldsymbol{\theta}^{\ast}} \boldsymbol{\Sigma}_\beta(\boldsymbol{\theta}^{\ast}) \left[\frac{\partial Q_{\boldsymbol{\theta}}(\alpha) }{\partial \boldsymbol{\theta}}\right]_{\boldsymbol{\theta} = \boldsymbol{\theta}^{\ast}},$$
	with $\boldsymbol{\Sigma}_\beta(\boldsymbol{\theta}) $ begin as defined  in Equation (\ref{eq:Sigmamatrix}) 
	and 
	\begin{equation}
		\frac{\partial Q_{\boldsymbol{\theta}}(\alpha)}{\partial \boldsymbol{\theta}} =  - \frac{ Q_{\boldsymbol{\theta}}(\alpha) }{k_1 (Q_{\boldsymbol{\theta}}(\alpha))}\left(1,  \frac{Q_{\boldsymbol{\theta}}(\alpha)}{2}, \frac{Q_{\boldsymbol{\theta}}(\alpha)^2}{3}, k_2(Q_{\boldsymbol{\theta}}(\alpha))x_0 \right)^T,
	\end{equation}
	with $k_1(t) = \gamma_2t^2+\gamma_1t+\gamma_0 $ and $ k_2(t) = \frac{\gamma_2}{3}t^2 + \frac{\gamma_1}{2}t + \gamma_0.$
	Details of the calculations are presented in Appendix \ref{app:quantilequadratic}. Further, using the asymptotic distribution of the $\alpha_0$-quantile and the consistency of the MDPDE, a $100(1-\alpha)\%$ confidence interval can be given as
	$$IC_{1-\alpha}\left[Q_{\boldsymbol{\theta}^\ast}(\alpha_0) \right] = \left[ Q_{\widehat{\boldsymbol{\theta}}_\beta}(\alpha_0) \pm z_{\alpha/2} \sigma_\beta(Q_{\widehat{\boldsymbol{\theta}}_\beta}(\alpha_0))\right].$$
	where $z_{\alpha}$ denotes the upper $\alpha$ quantile of a standard normal distribution.
	Again, the lower bound may be truncated if necessary.
	
	It is worthy to note the similarity of the formulas for the linear and quadratic baseline hazard cases. Indeed, if we set $\gamma_2=0$ and suppress the components corresponding to that parameter in all vector and matrices involved, we would recover the formulas presented for the linear case earlier in Section \ref{sec:char_linear}.

	\section{Monte Carlo simulation study \label{sec:simulation}}
	
	We examine the performance of the proposed MDPDEs through an extensive simulation study, considering both linear and quadratic baseline functions. We evaluate the accuracy of the estimates in terms of their mean squared error (MSE),  computed as
	\begin{equation*}
		MSE(\widehat{\gamma}^\beta_i) = ||\widehat{\gamma}^\beta_i-\gamma_i ||_2, \hspace{0.3cm} i=0,1,2, \hspace{0.3cm} \text{ and }  \hspace{0.3cm}	MSE(\widehat{a}^\beta_1) = ||\widehat{a}^\beta_1-a_1 ||_2.
	\end{equation*}
	Moreover, we examine the accuracy of the estimated lifetime characteristics based on the MDPDE, which will inherit the properties of the estimators. The accuracy of the estimation of the lifetime characteristics is also measured in terms of their MSE.
	
	To generate the failure counts, we consider a simple step-stress ALT assuming PH lifetimes and we then generate the count of failure within the interval $(t_{j-1}, t_j]$ using a conditional binomial distribution with probability  given by
	$$\widetilde{\pi}_j(\boldsymbol{\theta}) = \frac{\pi_j(\boldsymbol{\theta})}{1-G_T(t_j)},$$
	where $\pi_j(\boldsymbol{\theta})$ is the multinomial probability defined in Equation (\ref{eq:th}) for linear and Equation (\ref{eq:thquadratic}) for quadratic baseline hazard functions. Although both conditional binomial and multinomial models are equivalent, we use the first to generate the data so we can increase the probability of failure in one cell, and simulate the effect on the subsequent inspections.

	To evaluate the robustness of the estimators, we introduce some contamination in data.
	Generating outlying failure times may not necessarily result in an increased count of failures deemed as outliers. So, in contrast to the notion of outliers based on individual observations, \cite{barnett1992unusual} and \cite{victoria1997robust} pointed out that an outlier from grouped data is identified within a class whose associated probability (according to the underlying model) is notably smaller than the observed frequency.  This can be formalized by defining the ``adjusted residuals'' (see also \cite{fuchs1980test}) as follows:
	$$ r_j = \sqrt{N}\frac{\widehat{p}_j- \pi_j(\boldsymbol{\theta})}{\sqrt{\pi_j(\boldsymbol{\theta})(1-\pi_j(\boldsymbol{\theta}))}},$$
	 where $\pi_j(\boldsymbol{\theta})$  is the expected probability of class $j$ according to the underlying model and $\widehat{p}_j = n_j/N$ is the observed relative frequency in class $j$. Because $\pi_j(\boldsymbol{\theta})$ depends on the unknown model parameter $\boldsymbol{\theta}$, a robust estimation of the parameter is important to avoid the inflation of the denominator of the residual and the consequent misleading effect.
	 Due to the step-stress scheme in the generation of the failure times, increasing or decreasing the probability of failure in a cell will also influence the observed probabilities in the subsequent cells, but not the observed failures of the previous ones. Therefore, the adjusted residuals will increase at all cells from the contaminated cell onward.
		
	On the other hand, as all baseline hazard parameters are assumed to be positive, we reparameterize the hazard parameters for easier computation by exponentiating them, i.e., $\gamma_i' = \log(\gamma_i)$ and the hazard function coefficients are then $\exp(\boldsymbol{\gamma})$. This ensures that all estimated parameters are positive. If a parameter estimated is lower than $10^{-7}$, we considered it as zero.
	
	\subsection{Linear baseline  hazard}
	
	We consider linear hazard with true parameter vector $\boldsymbol{\theta}_0= (\exp(-4), \exp(-5.3), 0.5) $ and inspection times $2,  4,  6,  8, 10, 12, 14, 16, 18, 20$ and $22.$ The stress in increased at $\tau_1=14$ from $x_1=0.5$ to $x_2=2.5.$
	The hazard function of the true underlying model at constant stress $x$ is then
	$$h(t,x) = -\exp(0.5x)(\exp(-4)+ \exp(-5.3)t).$$
	We contaminate the model by increasing the probability of failure within the $10$-th cell by $(1+\varepsilon)$ times its probability, which will result in the decrease of failures in the subsequent cells. Figure \ref{fig:residualslinear} shows the increase on the residuals for the contaminated cells for increasing contamination. Note that the percentage of contaminated cells remains constant at $27\%,$ but the strength of contamination increases with $\varepsilon.$
	\begin{figure}[htb]
		\centering
		\includegraphics[width=9cm, height=5cm]{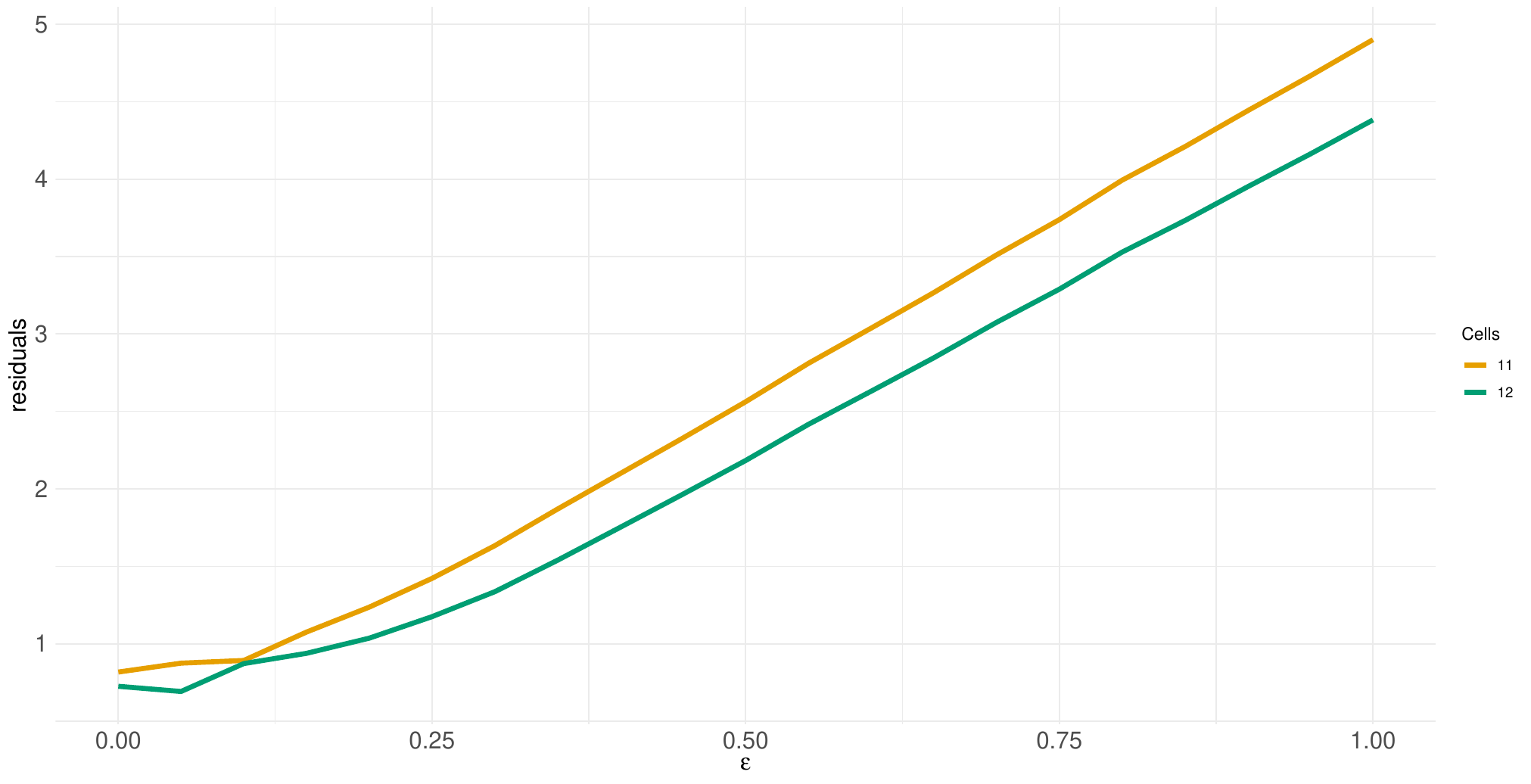}
		\caption{Residuals of contaminated cells under linear baseline hazard}
		\label{fig:residualslinear}
	\end{figure}
	
	Figure \ref{fig:MSElinear} presents the RMSE on the estimation of the three model parameters for increasing contamination of the model. As seen there, the MLE outperforms all its competitors in the absence of contamination, but rapidly worsens with contamination. In contrast, the MDPDE with positive values of $\beta$ are competitive to the MLE in the absence of contamination and are also not  heavily  influenced by contamination in data.
		\begin{figure}[htb]
		\centering
		\begin{subfigure}{0.3\textwidth}
			\includegraphics[scale=0.3]{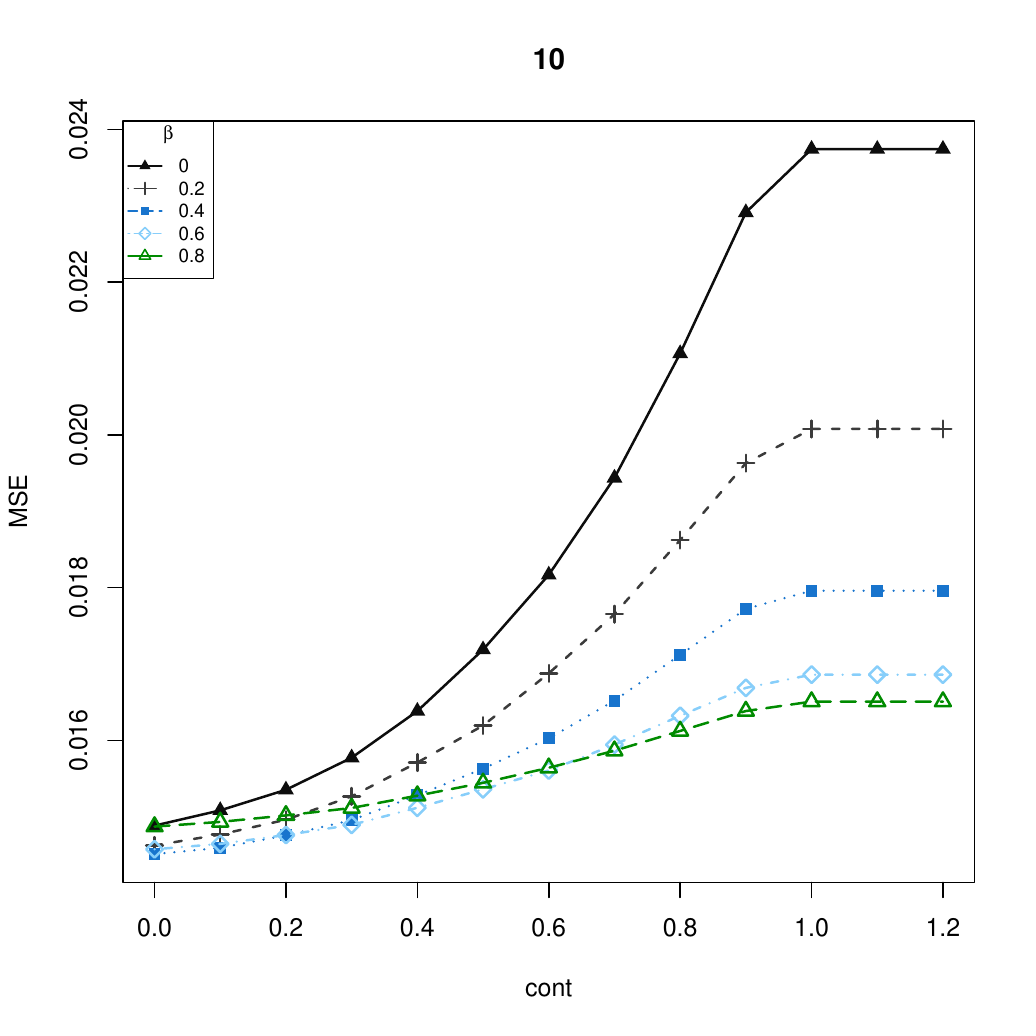}
			\subcaption{RMSE($\gamma_0$)}
		\end{subfigure}
		\begin{subfigure}{0.3\textwidth}
			\includegraphics[scale=0.3]{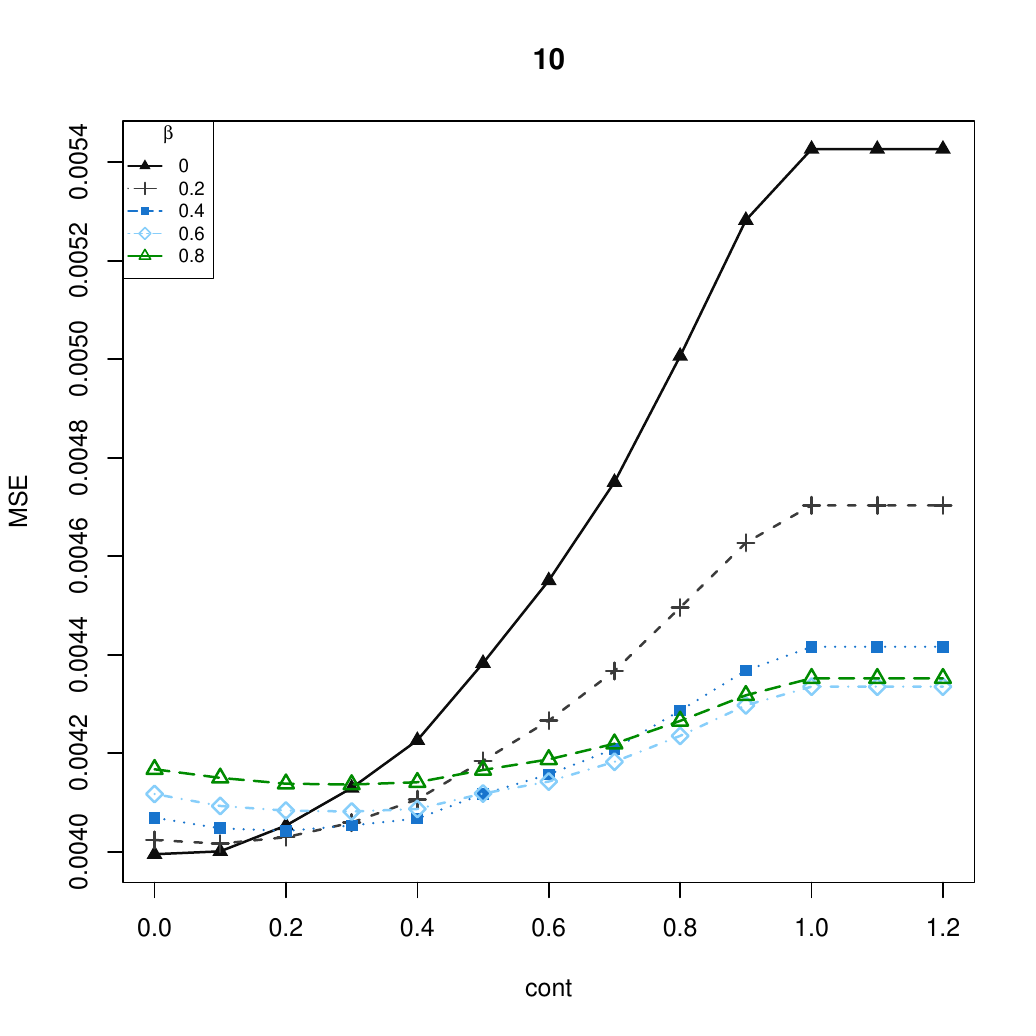}
			\subcaption{RMSE($\gamma_1$)}
		\end{subfigure}\begin{subfigure}{0.3\textwidth}
			\includegraphics[scale=0.3]{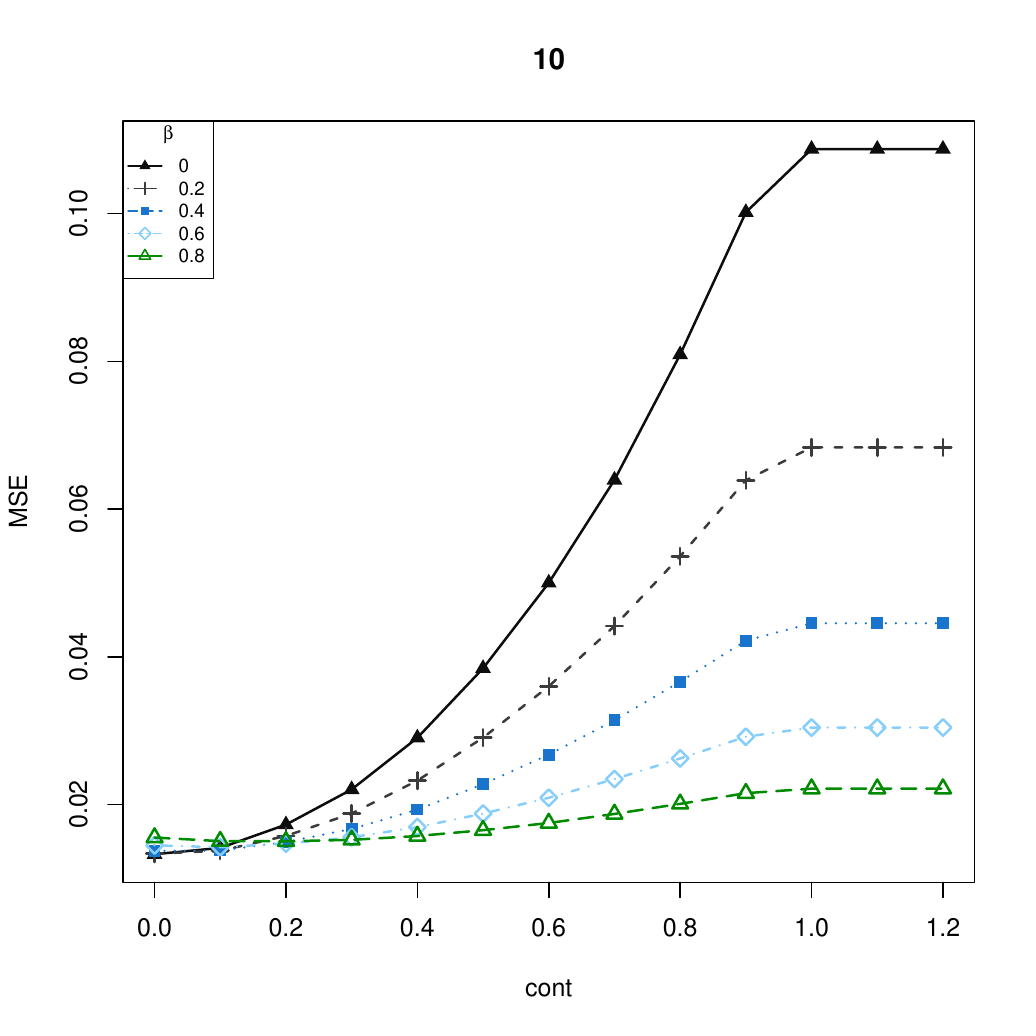}
			\subcaption{RMSE($a_1$)}
		\end{subfigure}
		\caption{Root mean squared error (RMSE) of the estimates under linear baseline hazard}
		\label{fig:MSElinear}
	\end{figure}
	 Recall that the MDPDE estimating equations defined in (\ref{eq:estimatedreliabilitylinearhazard}) down-weigh the observations with low probability (say $\pi$) by a factor of $\pi^\beta.$ Because the theoretical probability of failing within the last  two intervals is quite low, the MDPDE gets less influenced by an abnormally high number of failures.

	The performance of the estimators would naturally influence  the accuracy of the corresponding lifetime characteristics, which are of natural interest in many reliability analysis.
	For example, one may be interested in inferring specific quantiles, mean lifetime, hazard rate or the reliability of the product under NOC  at certain mission times.
	Figure \ref{fig:MSE_char_linear} presents the accuracy on the estimation of the median ($50\%$-quantile), mean lifetime, hazard rate and reliability at time $t_0=5$ under a lower stress (set at $x_0=0.3$ for illustrative purposes) representing the NOC.
	All efficiency and robustness properties of the MDPDEs are inherited by the estimates of the lifetime characteristics  and all estimates based on the MLE are seen to be heavily influenced by contamination present in the data.
	\begin{figure}[H]
	\centering
	\begin{subfigure}{0.35\textwidth}
		\includegraphics[scale=0.3]{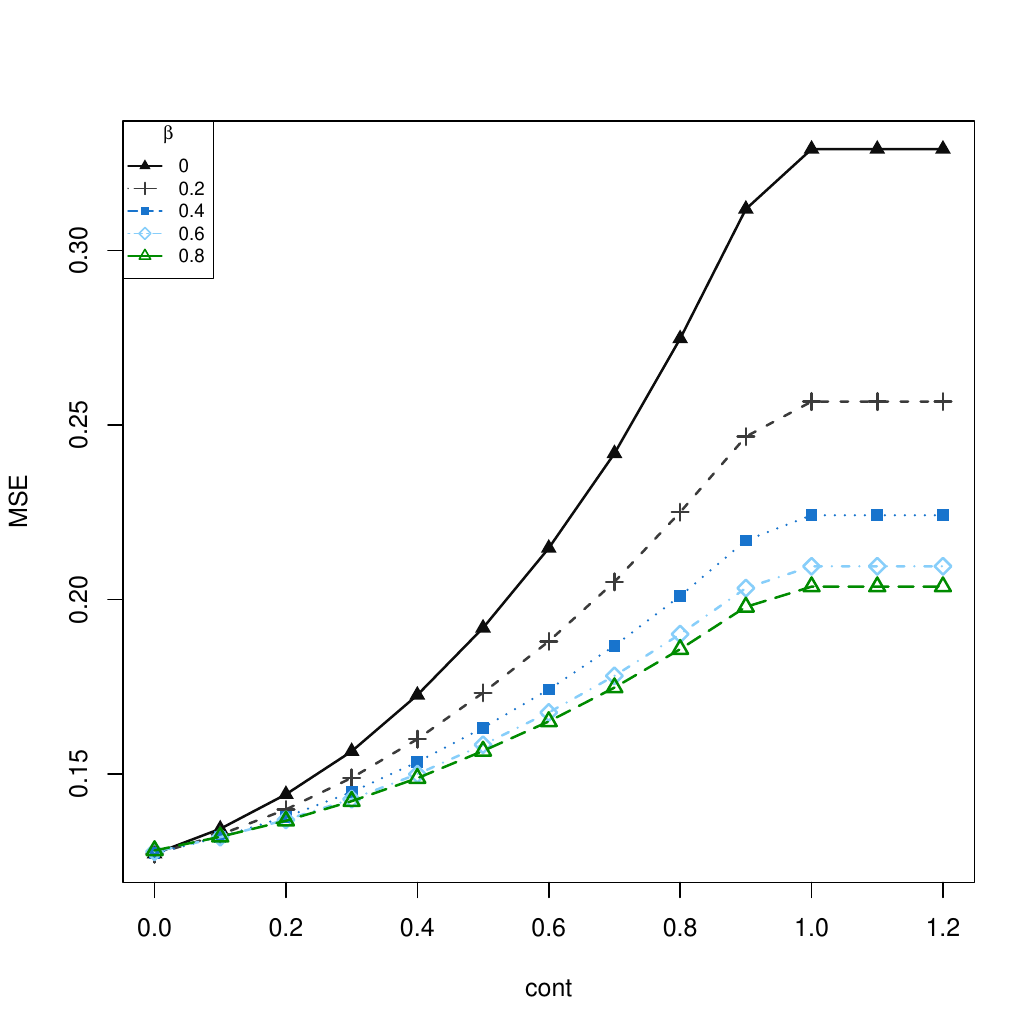}
		\subcaption{RMSE of the median}
	\end{subfigure}
	\begin{subfigure}{0.35\textwidth}
		\includegraphics[scale=0.3]{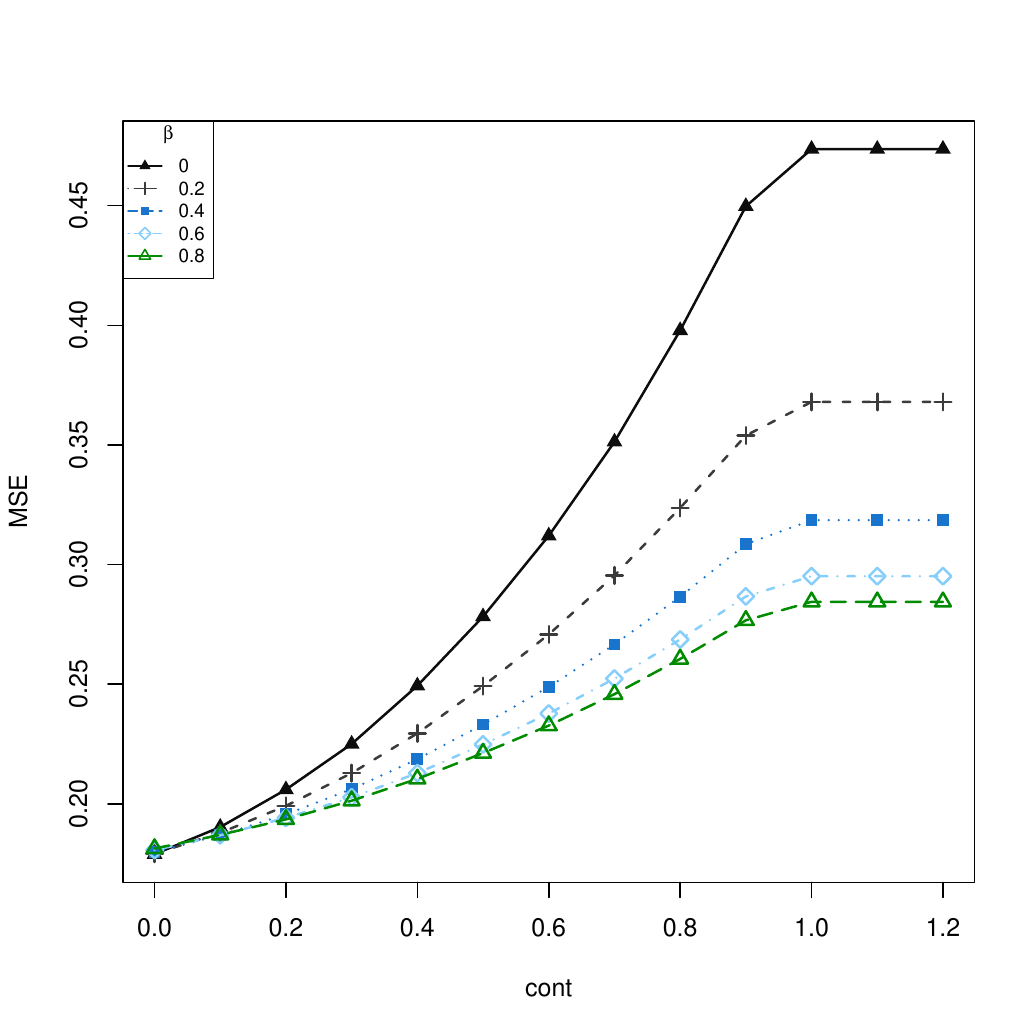}
		\subcaption{RMSE of the mean}
	\end{subfigure}
	\begin{subfigure}{0.35\textwidth}
		\includegraphics[scale=0.3]{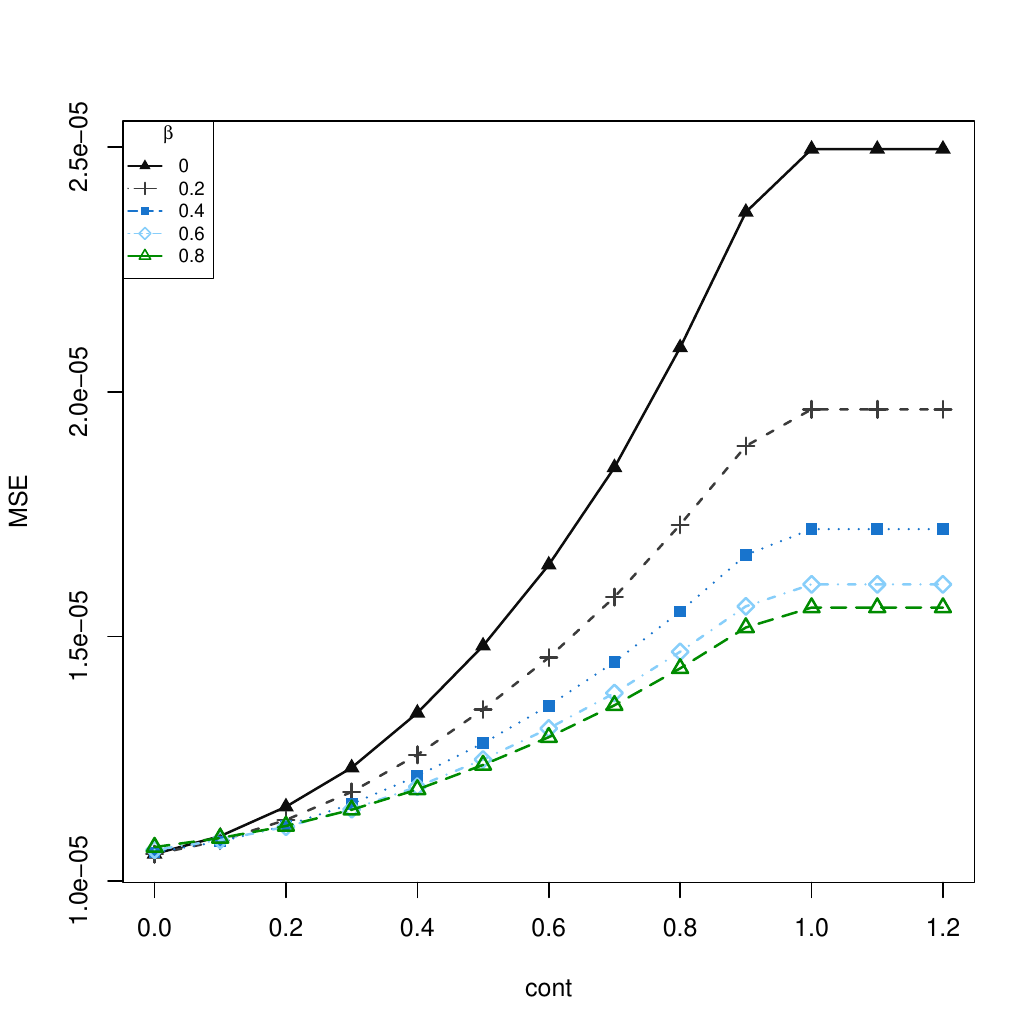}
		\subcaption{RMSE of the hazard}
	\end{subfigure}
	\begin{subfigure}{0.35\textwidth}
	\includegraphics[scale=0.35]{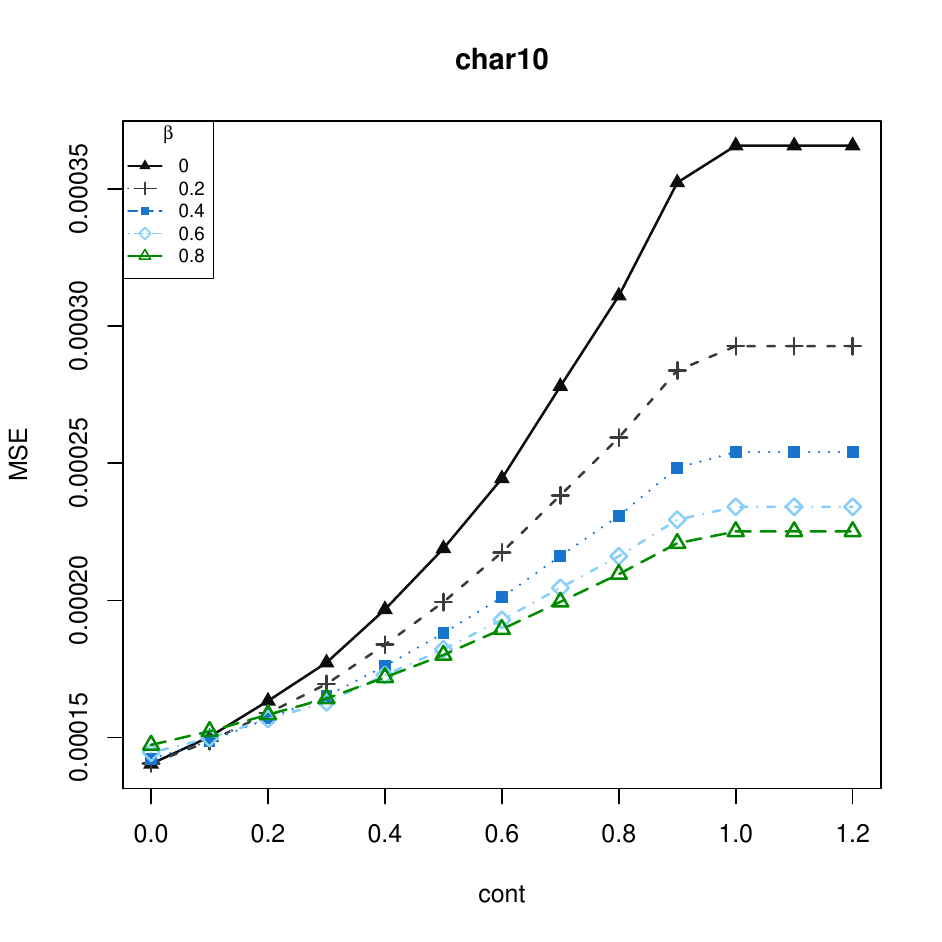}
	\subcaption{RMSE of the reliability}
	\end{subfigure}
	\caption{Root mean squared error (RMSE)  of the estimates of lifetime characteristics  under linear baseline hazard}
	\label{fig:MSE_char_linear}
	\end{figure}
		
	\subsection{Quadratic baseline hazard}
	
	We now consider a quadratic hazard baseline with true parameter $\boldsymbol{\theta}_0= (\exp(-4), 0  , \exp(-6.0) ,  0.5) $ so that the underlying hazard function is given by
	$$h(t,x) = -\exp(0.5x)(\exp(-4)+ \exp(-6)t^2).$$
	
	For quadratic hazard, the probability of failure will increase more rapidly in time and so we consider a shorter experiment with inspection times $1,  2,  3,  4,  5,  6,  7,  8 , 9, 10, 11$ and $12,$ when the experiment terminates. The same stress levels $x_1=0.5$ and $x_2=2.5$ are considered and the time of stress change is set at $\tau_1=8.$ We contaminate the penultimate cell as before, here corresponding to cell $11$ by increasing its probability of failure by $\varepsilon.$ 
	Figure \ref{fig:residualsquadratic} plots the increase in the residuals at the contaminated cells for increasing contamination amount. It can be seen that all three cells, including the event of surviving after the experiment terminates, are contaminated and consequently their residuals increases with $\varepsilon.$

	\begin{figure}[htb]
		\centering
		\includegraphics[width=9cm, height=5cm]{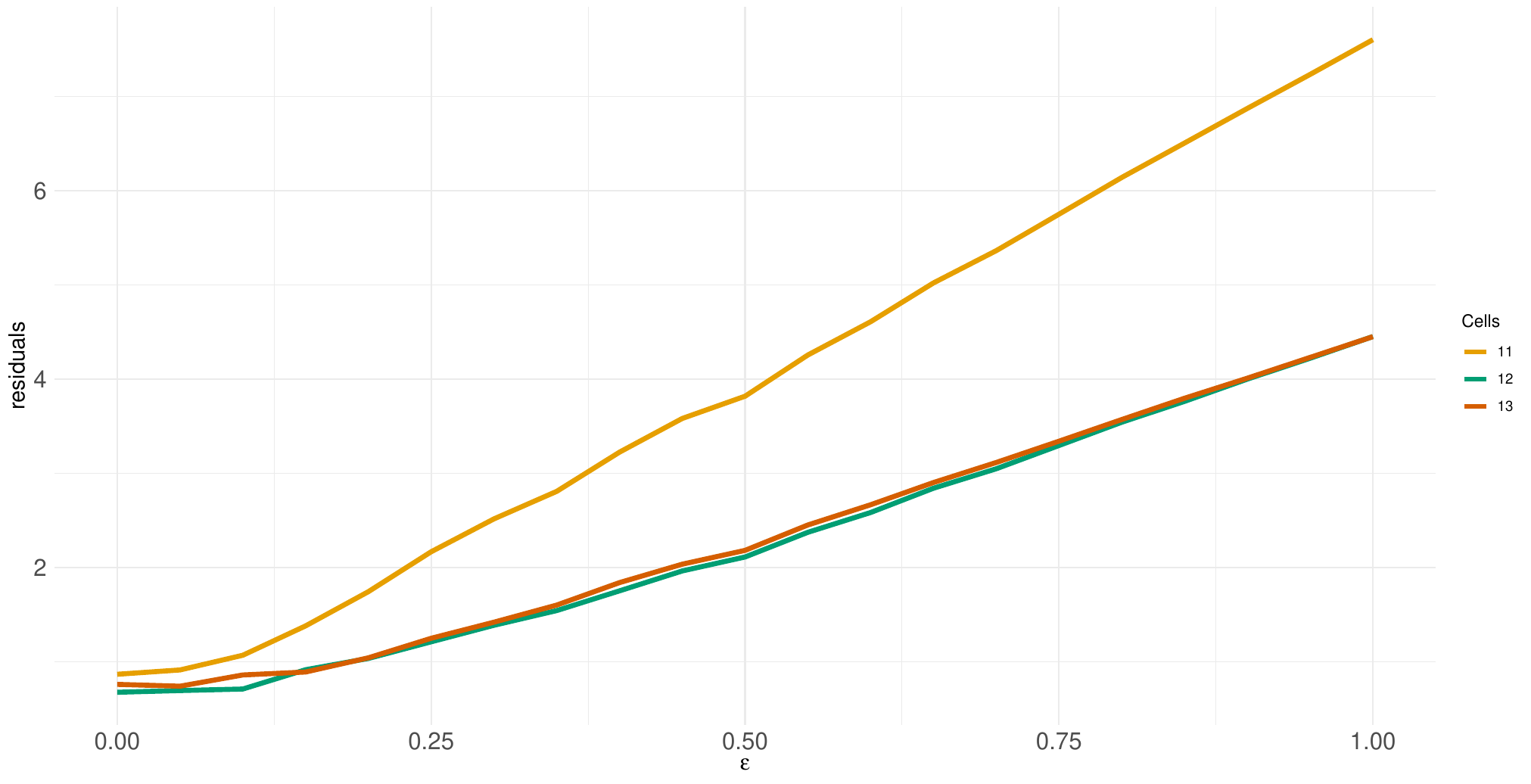}
		\caption{Residuals of contaminated cells under quadratic baseline hazard}
		\label{fig:residualsquadratic}
	\end{figure}
	 
	 As in the previous section, we evaluate the accuracy and robustness of the MDPDEs with different values of the tuning parameter. Figure \ref{fig:MSEquadratic} presents the root mean square error of the estimation of each non-zero model parameter $\gamma_0, \gamma_2,$ and $a_1.$ For the parameter $\gamma_1$ (which is null in the true parameter vector), all MDPDEs estimate it to be zero.

	 \begin{figure}[htb]
	 	\centering
	 	\begin{subfigure}{0.3\textwidth}
	 		\includegraphics[scale=0.3]{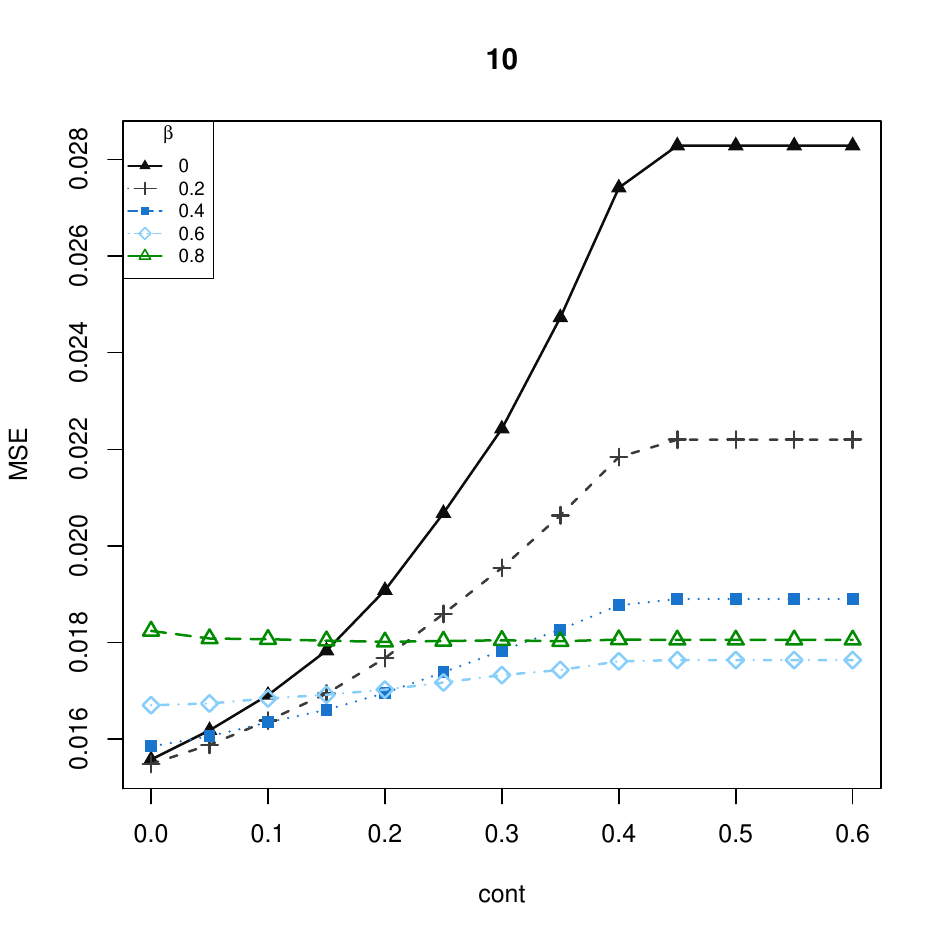}
	 		\subcaption{RMSE($\gamma_0$)}
	 	\end{subfigure}
	 	\begin{subfigure}{0.3\textwidth}
	 		\includegraphics[scale=0.3]{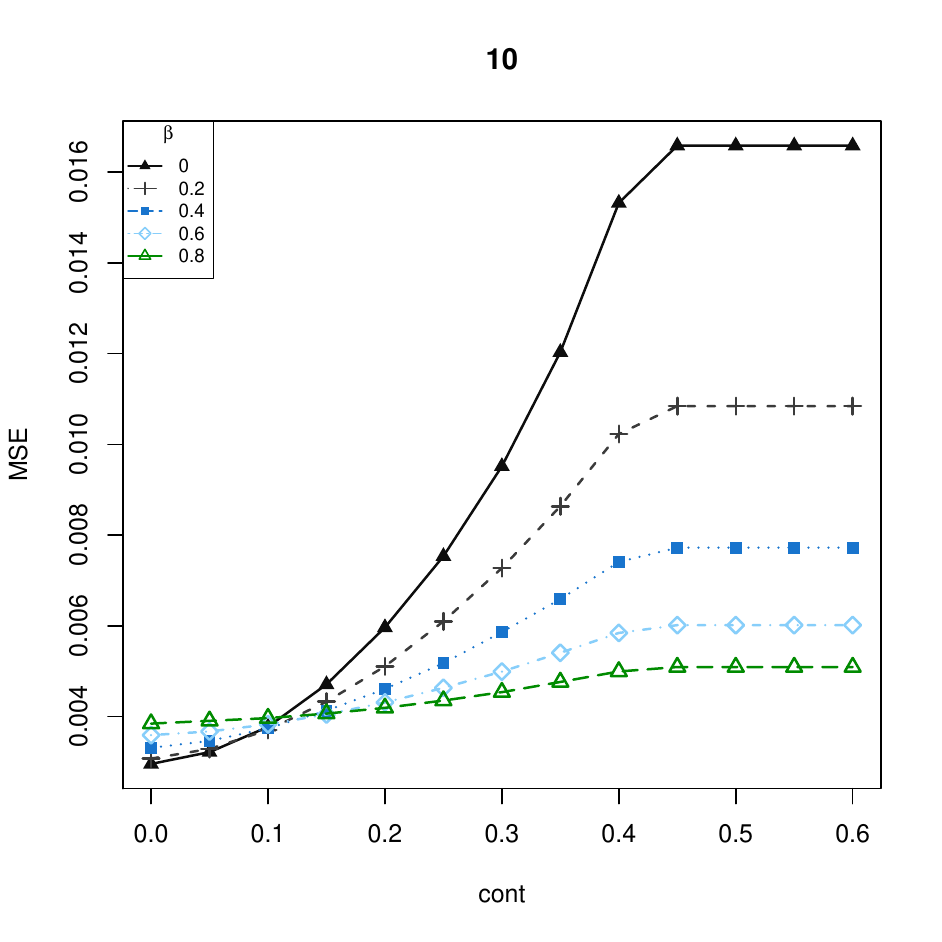}
	 		\subcaption{RMSE($\gamma_2$)}
	 	\end{subfigure}\begin{subfigure}{0.3\textwidth}
	 		\includegraphics[scale=0.3]{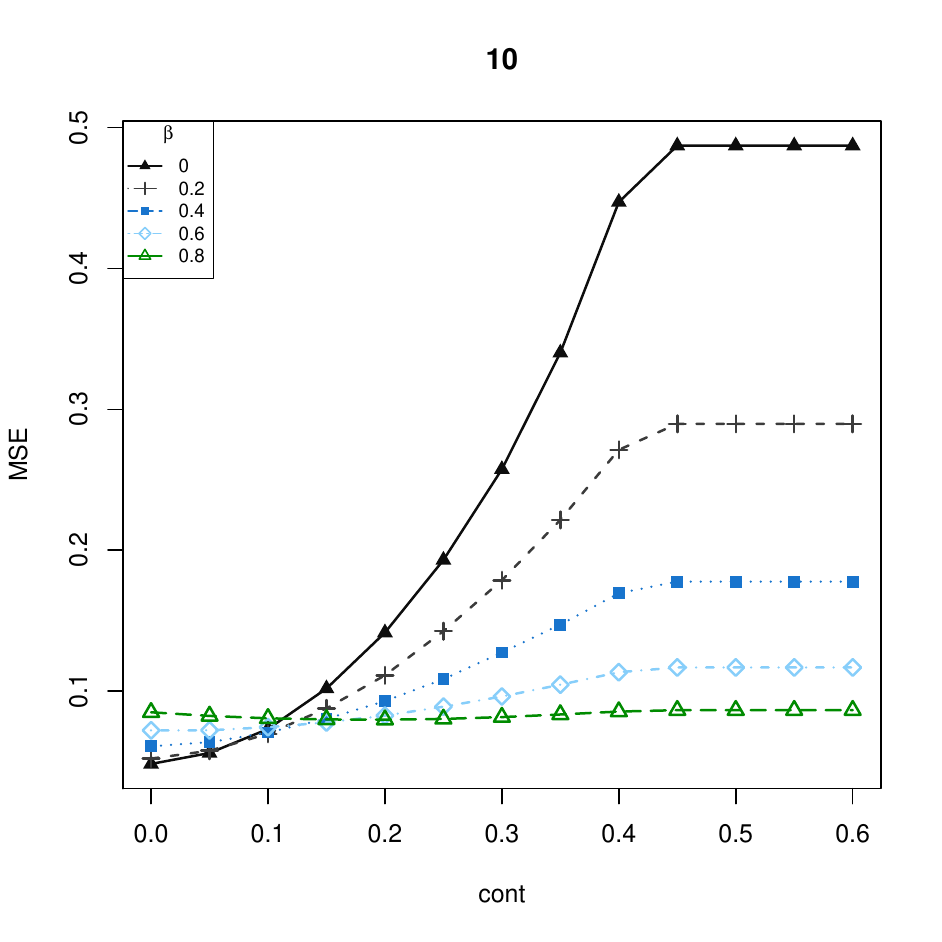}
	 		\subcaption{RMSE($a_1$)}
	 	\end{subfigure}
	 	\caption{Root mean squared error (RMSE) of the estimates of model parameters under quadratic baseline hazard}
	 	\label{fig:MSEquadratic}
	 \end{figure}

	Once again, the MLE is heavily influenced by an abnormal increase in failures at cells with low theoretical probability (i.e., outliers in grouped data), while the MDPDE with strictly positive values of the tuning parameter $\beta$ are much more robust against outlying cells.
	
%	\subsection{Performance comparisons between linear and quadratic hazard rates }
	
	\section{Illustrative Data Analysis \label{sec:realdata}}
	
	We revisit an example discussed in \cite{elsayed2002optimal}, wherein a two-step-stress ALT was conducted for Metal Oxide Semiconductor (MOS) capacitors in order to estimate its hazard rate over a 10 year period of time	at design temperature of $T=50^\circ C.$ The test needed to be completed in $T=300$ hours and the total number of test units under test was $N= 200.$ To avoid any unexpected change in failure mechanisms within the design temperature range, it was decided by engineers' judgment that the testing temperatures could not exceed a maximum temperature of	$T=250^\circ C.$
	Therefore, two stress levels were considered, $x_1=145^\circ C$ and $x_2=250^\circ C.$
	
	According to the Arrehenius model, the lifetime of a product or a chemical reaction and temperature to which it got exposed are related  as follows:
	$$T = A \exp\left(-\frac{E_a}{kx}\right),$$
	where $T$ represents the  lifetime of the product, 
	$A$ is the pre-exponential factor, to be determined for each product,
	$E_a$	is the activation energy, which varies by failure mechanism,
	$k =  8.36 \times 10^{-5}e V/^\circ K$ is an universal Boltzman's constant and
	$x$ is the absolute temperature in Kelvin.
	
	The Arrhenius law states that at higher temperatures, chemical reactions or degradation processes occur more rapidly. Conversely, at lower temperatures, the rates of these processes decrease. 
	The Arrhenius equation may help in predicting how long-term exposure to higher temperatures affects the product's lifetime. Therefore, accelerated life tests can be applied  on products by subjecting them to higher temperatures than they would typically experience during normal use, and then the failure behaviour can be extrapolated to the normal use temperature using the previous equation.
	 
	Then, re-parametrizing the stress factor to $s_i = -\frac{1}{x_i},$ the Arrhenius  Equation can be rewritten in a log-linear relation form as
	$$T =  \exp\left( \log(A) -\left(\frac{E_a}{k}\right)s\right);$$
	that is, the two stress are $x_1= -2.3914 \times 10^3$ and $x_2= -1.9114 \times 10^3. $
	 \cite{elsayed2002optimal} assumed linear baseline hazard functions and used the initial values 	$\gamma_0 = 0.0001, \gamma_1 = 0.5$ and $a_1 = 3,800$ for selecting optimal designs for the temperature and time of stress change. We used their suggested parameters to generate a real data-based sample with inspection times  $40, 60, 90, 110, 130, 150, 170, 183, 190, 210, 220$ and $250.$ Table \ref{table:realdataestimates} presents the MDPDEs for different values of the tuning parameter in the absence of contamination. The estimates are averages of the estimated parameters over $R=1,000$ simulated samples.
	 
	 \begin{table}[ht]
	 	\centering
	 	\begin{tabular}{rrrrrrr}
	 		\hline
	 		$\beta$ & 0 & 0.2 & 0.4 & 0.6 & 0.8 & 1 \\ 
	 		\hline
	 		$\widehat{\gamma}_0^\beta$ & 0.00012 & 0.00021 & 0.00011 & 0.00010 & 0.00009 & 0.00007 \\ 
	 		$\widehat{\gamma}_1^\beta$ & 0.4830 & 0.4785 & 0.4751 & 0.4728 & 0.4714 & 0.4716 \\ 
	 		 $\widehat{a}_1^\beta$ & 3,790 & 3,790 & 3,790 & 3,780 & 3,780 & 3,780 \\ 
	 		\hline
 		\end{tabular}
 	\caption{Model parameter estimates for the real-data based example }
 	\label{table:realdataestimates}
 	\end{table}
	 		
	 All MDPDE estimates, for the different values of $\beta,$ are quite similar and near the true value of the parameter, illustrating their competitive performance in the absence of contamination.
	 
	 Now, we consider a contaminated scenario. Following the description in Section \ref{sec:simulation}, we increase the probability of failure within the penultimate cell (and the subsequent cells are consequently also contaminated). Figure \ref{fig:MSE_char_realdata_linear} presents the RMSE of the estimates of the median, mean, reliability and hazard rate at $t_0=60$ under NOC based on the MDPDE with different values of $\beta$. As in the previous numerical analyses, the MLE gets heavily influenced by the abnormal count of failures and consequently its performance worsens rapidly as the amount of contamination increases.  Meanwhile, the MDPDEs with positive $\beta$s show a clear advantage over the MLE in terms of robustness. 
	 
	\begin{figure}[htb]
		\centering
		\begin{subfigure}{0.35\textwidth}
			\includegraphics[scale=0.3]{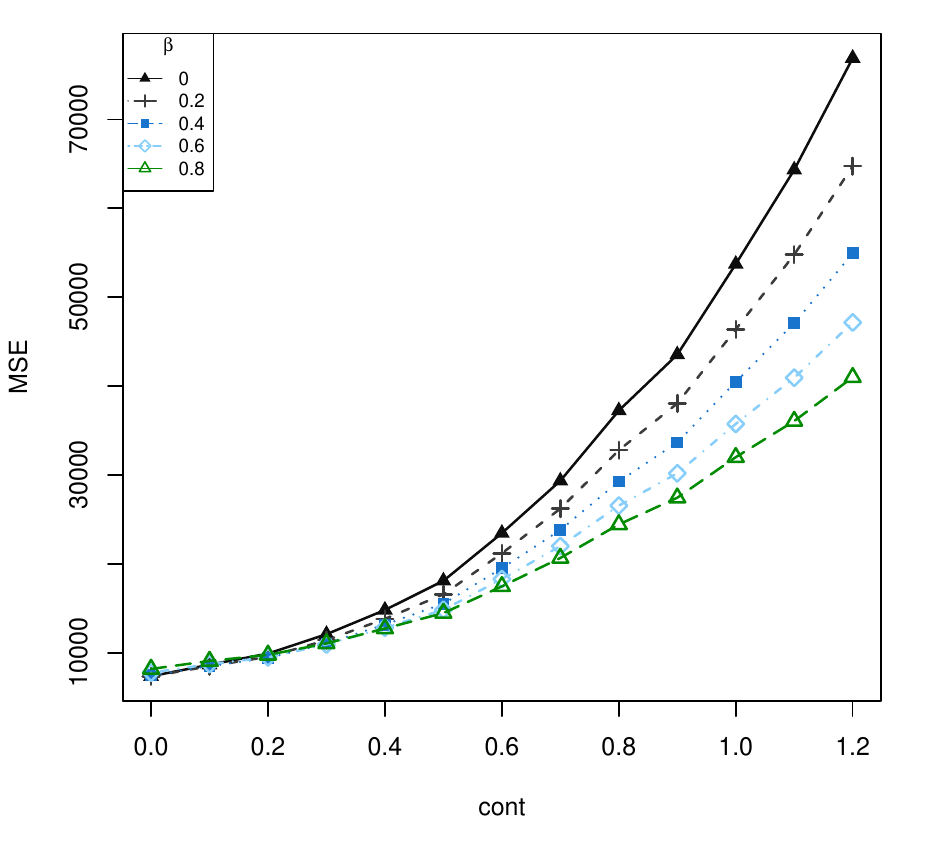}
			\subcaption{RMSE of the median}
		\end{subfigure}
		\begin{subfigure}{0.35\textwidth}
			\includegraphics[scale=0.3]{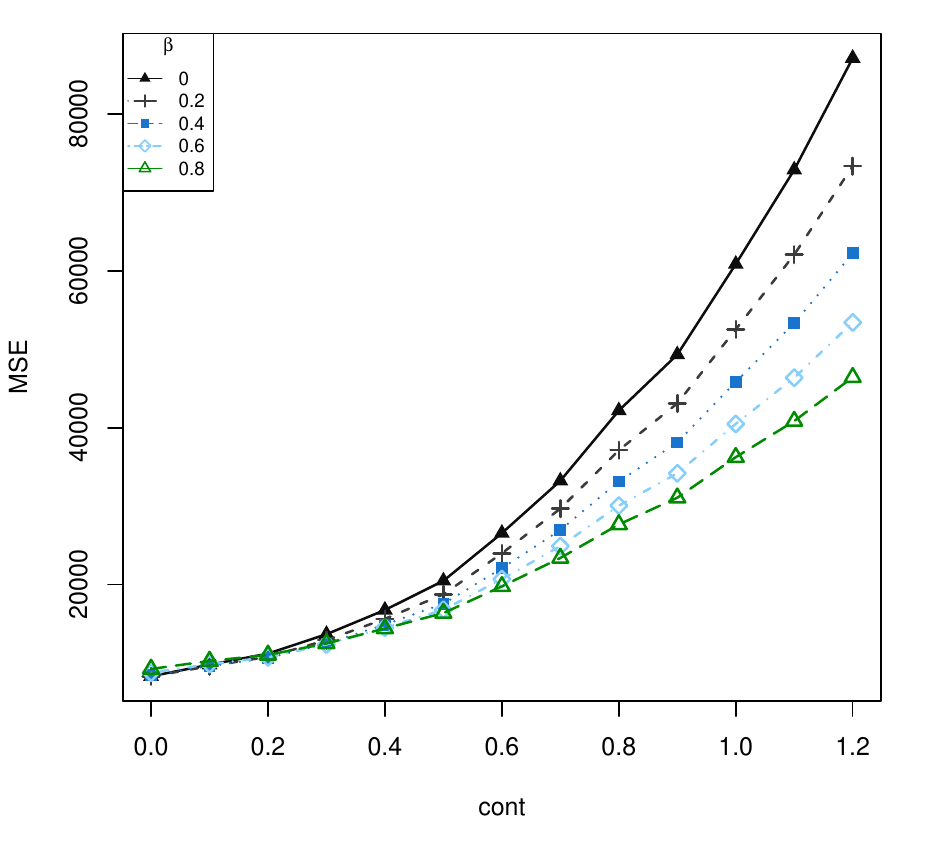}
			\subcaption{RMSE of the mean}
		\end{subfigure}
		\begin{subfigure}{0.35\textwidth}
			\includegraphics[scale=0.3]{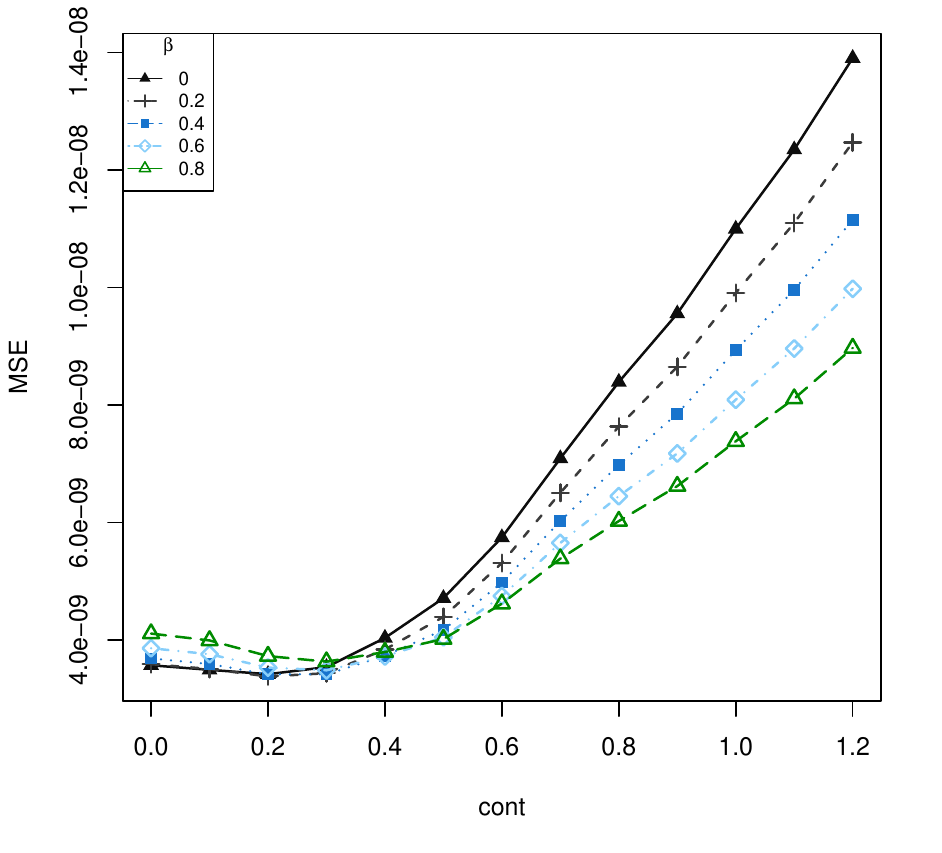}
			\subcaption{RMSE of the hazard}
		\end{subfigure}
		\begin{subfigure}{0.35\textwidth}
			\includegraphics[scale=0.3]{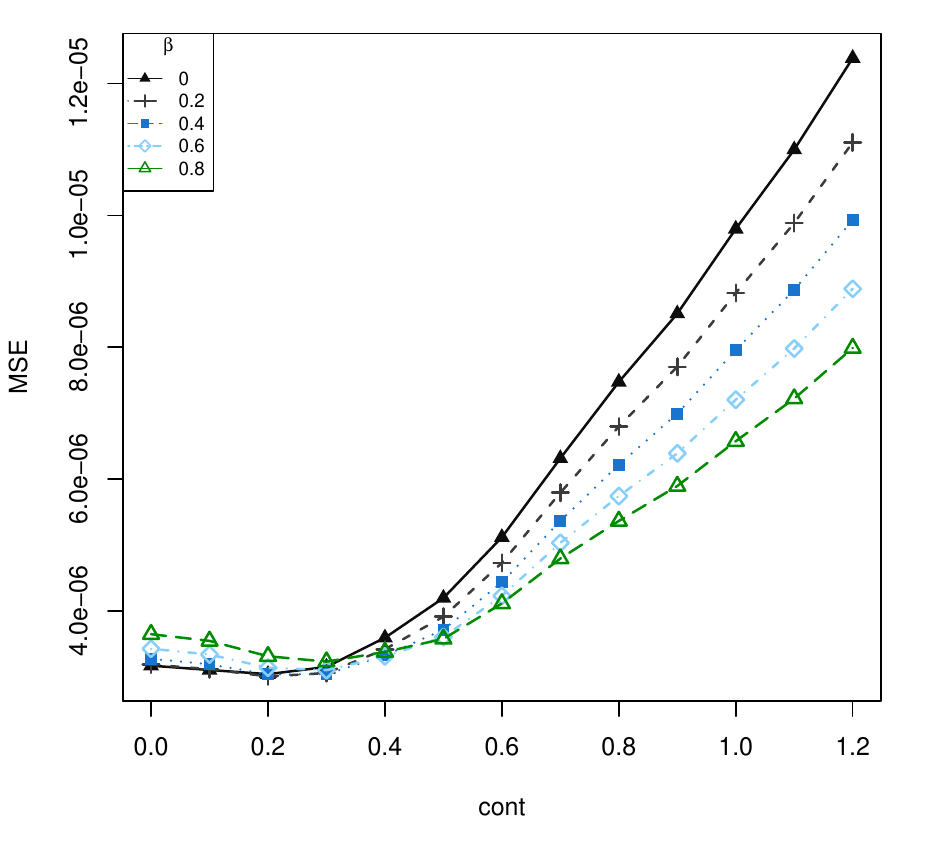}
			\subcaption{RMSE of the reliability}
		\end{subfigure}
		\caption{Root mean squared error (RMSE) of the estimates of lifetime characteristics for the illustrative example}
		\label{fig:MSE_char_realdata_linear}
	\end{figure}
%
%https://pubs.aip.org/avs/jvb/article/38/6/064001/588956/Experimental-reliability-study-of-cumulative 
%	Their resulting Q-optimal low accelerated stress level is 145C.
	
	\section{Conclusions \label{sec:concluding}}
	
	This paper introduces inferential methods for step-stress ALTs under proportional hazards for interval-monitored reliability experiments. Specifically, we have investigated estimation procedures under linear and quadratic baseline functions.
	Due to the lack of robustness of the MLE, we have presented a robust family of estimators known as the MDPDEs. This family is indexed by a tuning parameter, $\beta$, which controls the trade-off between the efficiency of the estimator and the robustness. Indeed, the MDPDE family includes the MLE as an extreme case with maximum efficiency, but minimal (indeed, none) robustness. %The MDPDE family has been shown to be asymptotically normal.
	Furthermore, from the estimates of the lifetime characteristics, we have derived point estimates and confidence intervals for the main lifetime characteristics, including the mean lifetime and distribution quantiles,  under both linear and quadratic forms. 
	
	All proposed estimators haven been evaluated  through a simulation study, demonstrating a significant advantage of the MDPDEs in terms of robustness, with a small compromise in efficiency. The performance of the robust estimators has  also been evaluated in a real application for inferring the reliability of Metal Oxide Semiconductor (MOS) capacitors, as described in \cite{elsayed2002optimal}. In both analyses, the MDPDEs are quite similar in the absence of contamination (demonstrating their competitive performance even with no contamination), but the MLE rapidly worsens when contamination gets introduced in the grouped data.
	
	While both linear and quadratic baseline forms exhibit flexibility in approximating the general baseline hazards, one would naturally be interested in preferable form for the baseline, based on the data at hand. In our future research, we intend to explore hypothesis tests for the baseline form.
	
	\section*{Acknowledgements}
	This work was supported by the Spanish Grant PID2021-124933NB-I00  and Natural Sciences and Engineering Research Council of Canada (of the first author) through an Individual Discovery Grant (No. 20013416).
	M. Jaenada and L. Pardo are members of the Interdisciplinary Mathematics Institute (IMI).
	
	\bibliographystyle{abbrvnat}
	\bibliography{bibliography.bib}

\begin{thebibliography}{16}
\providecommand{\natexlab}[1]{#1}
\providecommand{\url}[1]{\texttt{#1}}
\expandafter\ifx\csname urlstyle\endcsname\relax
  \providecommand{\doi}[1]{doi: #1}\else
  \providecommand{\doi}{doi: \begingroup \urlstyle{rm}\Url}\fi

\bibitem[AGA-EEI(1942)]{AGAEEI}
AGA-EEI.
\newblock An appraisal of methods for estimating service lives of utility
  properties, 1942.

\bibitem[Bagdonavicius(1978)]{Bagdonavicius1978}
V.~Bagdonavicius.
\newblock Testing the hypothesis of additive accumulation of damages.
\newblock \emph{Theory of Probability and its Applications}, 23:\penalty0
  403--408, 1978.

\bibitem[Balakrishnan et~al.(2023{\natexlab{a}})Balakrishnan, Castilla,
  Jaenada, and Pardo]{balakrishnan2023robust}
N.~Balakrishnan, E.~Castilla, M.~Jaenada, and L.~Pardo.
\newblock Robust inference for nondestructive one-shot device testing under
  step-stress model with exponential lifetimes.
\newblock \emph{Quality and Reliability Engineering International}, 39\penalty0
  (4):\penalty0 1192--1222, 2023{\natexlab{a}}.

\bibitem[Balakrishnan et~al.(2023{\natexlab{b}})Balakrishnan, Jaenada, and
  Pardo]{balakrishnan2023Non}
N.~Balakrishnan, M.~Jaenada, and L.~Pardo.
\newblock Non-destructive one-shot device test under step-stress experiment
  with lognormal lifetime distribution.
\newblock \emph{Journal of Computational and Applied Mathematics}, page 115483,
  2023{\natexlab{b}}.

\bibitem[Balakrishnan et~al.(2023{\natexlab{c}})Balakrishnan, Jaenada, and
  Pardo]{balakrishnan2023step}
N.~Balakrishnan, M.~Jaenada, and L.~Pardo.
\newblock Step-stress tests for interval-censored data under gamma lifetime
  distribution.
\newblock \emph{Quality Engineering}, pages 1--18, 2023{\natexlab{c}}.

\bibitem[Barnett(1992)]{barnett1992unusual}
V.~Barnett.
\newblock Unusual outliers.
\newblock \emph{Data Analysis and Statistical Inference. Koln: Verlag}, 1992.

\bibitem[Bhattacharyya and Soejoeti(1989)]{bhattacharyya1989tampered}
G.~K. Bhattacharyya and Z.~Soejoeti.
\newblock A tampered failure rate model for step-stress accelerated life test.
\newblock \emph{Communications in statistics-Theory and methods}, 18\penalty0
  (5):\penalty0 1627--1643, 1989.

\bibitem[Cox(1972)]{cox1972regression}
D.~R. Cox.
\newblock Regression models and life-tables.
\newblock \emph{Journal of the Royal Statistical Society: Series B
  (Methodological)}, 34\penalty0 (2):\penalty0 187--202, 1972.

\bibitem[DeGroot and Goel(1979)]{degroot1979bayesian}
M.~H. DeGroot and P.~K. Goel.
\newblock Bayesian estimation and optimal designs in partially accelerated life
  testing.
\newblock \emph{Naval Research Logistics Quarterly}, 26\penalty0 (2):\penalty0
  223--235, 1979.

\bibitem[Elsayed and Jiao(2002)]{elsayed2002optimal}
E.~Elsayed and L.~Jiao.
\newblock Optimal design of proportional hazards based accelerated life testing
  plans.
\newblock \emph{International Journal of Materials and Product Technology},
  17\penalty0 (5-6):\penalty0 411--424, 2002.

\bibitem[Fuchs and Kenett(1980)]{fuchs1980test}
C.~Fuchs and R.~Kenett.
\newblock A test for detecting outlying cells in the multinomial distribution
  and two-way contingency tables.
\newblock \emph{Journal of the American Statistical Association}, 75\penalty0
  (370):\penalty0 395--398, 1980.

\bibitem[Krane(1963)]{krane1963analysis}
S.~A. Krane.
\newblock Analysis of survival data by regression techniques.
\newblock \emph{Technometrics}, 5\penalty0 (2):\penalty0 161--174, 1963.

\bibitem[NARUC(1938)]{NARUC}
NARUC.
\newblock Report of the committee on depreciation, 1938.

\bibitem[Nelson(1990)]{nelson1990accelerated}
W.~Nelson.
\newblock \emph{Accelerated Testing: Statistical Models, Test Plans, and Data
  Analysis}.
\newblock John Wiley \& Sons, New York, 1990.

\bibitem[Sedyakin(1966)]{Sedyakin1966}
N.~M. Sedyakin.
\newblock On one physical principle in reliability theory (in russian).
\newblock \emph{Technical Cybernetics}, 3:\penalty0 80--87, 1966.

\bibitem[Victoria-Feser and Ronchetti(1997)]{victoria1997robust}
M.-P. Victoria-Feser and E.~Ronchetti.
\newblock Robust estimation for grouped data.
\newblock \emph{Journal of the American Statistical Association}, 92\penalty0
  (437):\penalty0 333--340, 1997.

\end{thebibliography}
	
	\appendix
	\section{Proofs of the main results}
	
	\subsection{Proof of Result \ref{thm:estimatinglinear}} \label{app:proof1}
	\begin{proof}
		Taking derivatives in the DPD loss defined in Eq. (\ref{eq:DPDloss}) and equating them to zero, we get
		\begin{equation*}
			\frac{\partial d_{\beta}\left( \widehat{\boldsymbol{p}},\boldsymbol{\pi}\left(\boldsymbol{\theta}\right)\right)}{\partial \boldsymbol{\theta}}   = (\beta+1)\sum_{j=1}^{L+1} \left(  \pi_j(\boldsymbol{\theta})^{\beta-1} \left(\pi_j(\boldsymbol{\theta}) -  \widehat{p}_j\right) \frac{\partial \pi_j(\boldsymbol{\theta})}{\partial \boldsymbol{\theta}}  \right) = \boldsymbol{0}_3,
		\end{equation*}
		where the derivative of each probability of failure $\pi_j(\boldsymbol{\theta})$ is given by
		
		\begin{equation*}
			\frac{\partial\pi_j(\boldsymbol{\theta})}{\partial \boldsymbol{\theta}}
			=   \frac{\partial R_{\boldsymbol{\theta}}(t_{j-1})}{\partial \boldsymbol{\theta}} -\frac{\partial R_{\boldsymbol{\theta}}(t_j)}{\partial \boldsymbol{\theta}}.
			%		\begin{cases}
				%			\exp \left( -\exp(a_1x_1) \left( \gamma_0t + \gamma_1 \frac{t^2}{2}\right) \right) & 0 < t < \tau,\\
				%			\exp\left(-\exp(a_1x_2) \left( \gamma_0(t+s) + \gamma_1 \frac{(t+s)^2}{2}\right)\right) &  \tau < t < \infty,
				%		\end{cases}
		\end{equation*}
		Let us write the reliability function of the device at stress level $x$ as  $$R_{\boldsymbol{\theta}}(t, x) = 
		\begin{cases}
			\exp\left(-\exp(a_1x_1) \left( \gamma_0t + \gamma_1 \frac{t^2}{2}\right)\right), & x=x_1,\\
			\exp\left(-\exp(a_1x_2) \left( \gamma_0(t+s) + \gamma_1 \frac{(t+s)^2}{2}\right)\right), & x=x_2,
		\end{cases}$$
		where $s$ is the corresponding shifting time  corresponding to the stress change from $x_1$ to $x_2,$  which depends on the model parameters and satisfies Eq. (\ref{eq:secondorderequation}).
		
		Firstly, we need to compute the derivatives of the shifting time $s$ (as a function of the parameters $\boldsymbol{\theta}$).  Taking implicit derivatives in Equation (\ref{eq:secondorderequation}) with respect to each parameter $(\gamma_0, \gamma_1,a_0, a_1),$ we obtain the following system of linear equations:
		\begin{equation*}
			\begin{aligned}
				\gamma_1(\tau+s)\frac{\partial s}{\partial \gamma_0} + \gamma_0 \frac{\partial s}{\partial \gamma_0} &= - (\tau+s) + \frac{\exp(a_1x_1)}{\exp(a_1x_2)}\tau , \\
				\gamma_1(\tau+s)\frac{\partial s}{\partial \gamma_1} + \gamma_0 \frac{\partial s}{\partial \gamma_1} &= - \frac{1}{2}(\tau+s)^2 + \frac{\exp(a_1x_1)}{\exp(a_1x_2)} \frac{1}{2}\tau^2,  \\
				% \gamma_1(\tau+s)\frac{\partial s}{\partial a_0} + \gamma_0 \frac{\partial s}{\partial a_0} &= 0, \\
				\gamma_1(\tau+s)\frac{\partial s}{\partial a_1} + \gamma_0 \frac{\partial s}{\partial a_1} &= \frac{\exp(a_1x_1)}{\exp(a_1x_2)}(x_1-x_2) \left(\frac{\gamma_1}{2}\tau^2 + \gamma_0\tau\right).
			\end{aligned}
		\end{equation*}
		Let us define $$k_1(t) = \gamma_1t+\gamma_0 \hspace{1cm}\text{and}\hspace{1cm} k_2(t) = \frac{\gamma_1}{2}t + \gamma_0.$$
		Then, solving the above equations, we obtain the derivatives of the shifting time as follows:
		\begin{equation*}
			\begin{aligned}
				\frac{\partial s}{\partial \gamma_0} & =  \frac{\gamma_1}{2} \frac{ (\tau+s) s}{k_2(\tau)}\frac{1}{k_1(\tau+s)}, \\
				\frac{\partial s}{\partial \gamma_1}  & = - \frac{\gamma_0}{2}\frac{(\tau+s)s}{k_2(\tau)} \frac{1}{k_1(\tau+s)}, \\
				%\frac{\partial s}{\partial a_0}  & =  0 &
				\frac{\partial s}{\partial a_1}  & =  \frac{(\tau+s)}{\tau}\frac{k_2(\tau+s)}{k_1(\tau+s)}(x_1-x_2),
				% \gamma_1(\tau+s)\frac{\partial s}{\partial a_1} + \gamma_0 \frac{\partial s}{\partial a_1} = \frac{1}{2}\frac{\exp(a_1x_1)}{\exp(a_1x_2)}\frac{\gamma_1\tau^2 + 2\gamma_0\tau}{\gamma_1(\tau+s)+\gamma_0}(x_1-x_2) 
			\end{aligned}
		\end{equation*}
		where we have used the equality
		$$\frac{\exp(a_1x_1)}{\exp(a_1x_2)} = \frac{(\tau+s)}{ \tau}\frac{k_2(\tau+s)}{k_2(\tau)}$$
		to simplify the expressions.
		Now, we compute separately the derivatives of the reliability under stress levels $x_1$ and $x_2.$ Under the first stress level $x_1$, we have 
		
		\begin{equation*}
			\begin{aligned}
				\frac{\partial  R_{\boldsymbol{\theta}}(t,x_1)}{\partial \gamma_0} &= -	R_{\boldsymbol{\theta}}(t,x_1)
				\exp(a_1x_1) t ,\\
				\frac{\partial R_{\boldsymbol{\theta}}(t,x_1)}{\partial \gamma_1} &= -R_{\boldsymbol{\theta}}(t,x_1)
				\exp(a_1x_1) \frac{t^2}{2} ,\\
				%
				%\frac{\partial R_{\boldsymbol{\theta}}(t,x_1)}{\partial a_0} &= 	-R_{\boldsymbol{\theta}}(t,x_2) \exp(a_1x_1) tk_2(t),\\
				%
				\frac{\partial R_{\boldsymbol{\theta}}(t,x_1)}{\partial a_1} &= 	-R_{\boldsymbol{\theta}}(t,x_1) \exp(a_1x_1)  tk_2(t) x_2.\\
			\end{aligned}
		\end{equation*}
		
		%REVISAR PARCIAL A_1
		On the other hand, the derivatives of the reliability under the increased stress $x_2$ with respect to each parameter are as follows:
		\begin{equation*}
			\begin{aligned}
				\frac{\partial R_{\boldsymbol{\theta}}(t,x_2)}{\partial \gamma_0} &= 	\exp\left(-\exp(a_1x_2) \left( \gamma_0(t+s) + \gamma_1 \frac{(t+s)^2}{2}\right)\right)\\
				&  \hspace{0.5cm} \times \left[
				- \exp(a_1x_2) \left((t+s)+\gamma_0\frac{\partial s}{\partial \gamma_0} + \gamma_1(t+s)\frac{\partial s}{\partial \gamma_0}\right)\right],\\
				\frac{\partial R_{\boldsymbol{\theta}}(t,x_2)}{\partial \gamma_1} &= 	\exp\left(-\exp(a_1x_2) \left( \gamma_0(t+s) + \gamma_1 \frac{(t+s)^2}{2}\right)\right)\\
				&  \hspace{0.5cm} \times \left[
				- \exp(a_1x_2) \left(\frac{(t+s)^2}{2}+\gamma_0\frac{\partial s}{\partial \gamma_1} + \gamma_1(t+s)\frac{\partial s}{\partial \gamma_1}\right)\right],\\
				%
				%\frac{\partial R_{\boldsymbol{\theta}}(t,x_2)}{\partial a_0} &= 	\exp\left(-\exp(a_1x_2) \left( \gamma_0(t+s) + \gamma_1 \frac{(t+s)^2}{2}\right)\right)\\
			%	&  \hspace{0.5cm} \times \left[-\exp(a_1x_2) \left( \gamma_0(t+s) + \gamma_1 \frac{(t+s)^2}{2}\right) 
			%	- \exp(a_1x_2) \left(\gamma_0\frac{\partial s}{\partial a_0} + \gamma_1(t+s)\frac{\partial s}{\partial a_0}\right)\right],\\
				%
				\frac{\partial R_{\boldsymbol{\theta}}(t,x_2)}{\partial a_1} &= 	\exp\left(-\exp(a_1x_2) \left( \gamma_0(t+s) + \gamma_1 \frac{(t+s)^2}{2}\right)\right)\\
				&  \hspace{0.5cm} \times \left[-\exp(a_1x_2) \left( \gamma_0(t+s) + \gamma_1 \frac{(t+s)^2}{2}\right) x_2
				- \exp(a_1x_2) \left(\gamma_0\frac{\partial s}{\partial a_1} + \gamma_1(t+s)\frac{\partial s}{\partial a_1}\right)\right],\\
			\end{aligned}
		\end{equation*}
		Upon substituting the formulas of the shifting time derivatives in the above expressions, we can rewrite the above derivatives as follows:
		\begin{equation*}
			\begin{aligned}
				\frac{\partial  R_{\boldsymbol{\theta}}(t,x_2)}{\partial \gamma_0} &= -	R_{\boldsymbol{\theta}}(t,x_2) 
				\exp(a_1x_2) \left( (t+s)+ \frac{ \gamma_1}{2} 
				\frac{(\tau+s) s}{k_2(\tau)}\frac{k_1(t+s)}{k_1(\tau+s)}
				\right),\\
				\frac{\partial R_{\boldsymbol{\theta}}(t,x_2)}{\partial \gamma_1} &= -R_{\boldsymbol{\theta}}(t,x_2)
				\exp(a_1x_2) \left(\frac{(t+s)^2}{2}-    \frac{\gamma_0}{2} \frac{(\tau+s)s}{k_2(\tau)} \frac{k_1(t+s)}{k_1(\tau+s)}          \right),\\
				%
				%\frac{\partial R_{\boldsymbol{\theta}}(t,x_2)}{\partial a_0} &= 	- R_{\boldsymbol{\theta}}(t,x_2) \exp(a_1x_2)(t+s)k_2(t+s),\\
				%
				\frac{\partial R_{\boldsymbol{\theta}}(t,x_2)}{\partial a_1} &= 	- R_{\boldsymbol{\theta}}(t,x_2) \exp(a_1x_2) \left[ 
				(t+s)k_2(t+s) x_2
				 +  \frac{k_1(t+s)k_2(\tau+s)}{k_1(\tau+s)} (\tau+s)(x_1-x_2) \right].
			\end{aligned}
		\end{equation*}
		
		Now, defining  $$\boldsymbol{z}_j = \begin{cases}
			\frac{\partial R_{\boldsymbol{\theta}}(t_{j-1},x_1)}{\partial \boldsymbol{\theta}} - \frac{\partial R_{\boldsymbol{\theta}}(t_{j},x_1)}{\partial \boldsymbol{\theta}} & t_{j-1} < \tau \\
			\frac{\partial R_{\boldsymbol{\theta}}(t_{j-1},x_2)}{\partial \boldsymbol{\theta}} - \frac{\partial R_{\boldsymbol{\theta}}(t_{j},x_2)}{\partial \boldsymbol{\theta}} & \tau \leq t_{j-1} 
		\end{cases},$$
		we obtain the required formula.		
	\end{proof}

%%%%%%%%%%%%%%%%%%%%%%%%%%%%%%%%%%%%%%%

\subsection{Proof of Result \ref{thm:estimatingquadratic}} \label{app:proof2}
\begin{proof}
	Taking derivatives in the DPD loss defined in Eq. (\ref{eq:DPDloss}) and equating them to zero, we get
	\begin{equation*}
		\frac{\partial d_{\beta}\left( \widehat{\boldsymbol{p}},\boldsymbol{\pi}\left(\boldsymbol{\theta}\right)\right)}{\partial \boldsymbol{\theta}}   = (\beta+1)\sum_{j=1}^{L+1} \left(  \pi_j(\boldsymbol{\theta})^{\beta-1} \left(\pi_j(\boldsymbol{\theta}) -  \widehat{p}_j\right) \frac{\partial \pi_j(\boldsymbol{\theta})}{\partial \boldsymbol{\theta}}  \right) = \boldsymbol{0}_4.
	\end{equation*}
	where the derivative of each probability of failure $\pi_j(\boldsymbol{\theta})$ is given by
	\begin{equation*}
		\frac{\partial\pi_j(\boldsymbol{\theta})}{\partial \boldsymbol{\theta}}
		=   \frac{\partial R_{\boldsymbol{\theta}}(t_{j-1})}{\partial \boldsymbol{\theta}} -\frac{\partial R_{\boldsymbol{\theta}}(t_j)}{\partial \boldsymbol{\theta}},
		%		\begin{cases}
			%			\exp \left( -\exp(a_1x_1) \left( \gamma_0t + \gamma_1 \frac{t^2}{2}\right) \right) & 0 < t < \tau,\\
			%			\exp\left(-\exp(a_1x_2) \left( \gamma_0(t+s) + \gamma_1 \frac{(t+s)^2}{2}\right)\right) &  \tau < t < \infty,
			%		\end{cases}
	\end{equation*}
	with $R_{\boldsymbol{\theta}}(t_{j-1})$ denoting the reliability under quadratic hazard. The explicit formula of $\pi_j(\boldsymbol{\theta})$ is as given in (\ref{pijquadratic}).
	Let us write the reliability function of the device at stress level $x$ as  $$R_{\boldsymbol{\theta}}(t, x) = 
	\begin{cases}
		\exp\left(-\exp(a_1x_1) \left( \gamma_0t + \gamma_1 \frac{t^2}{2} + \gamma_2 \frac{t^3}{3}\right)\right), & x=x_1,\\
		\exp\left(-\exp(a_1x_2) \left( \gamma_0(t+s) + \gamma_1 \frac{(t+s)^2}{2} + \gamma_2 \frac{(t+s)^3}{3} \right)\right), & x=x_2,
	\end{cases}$$
	where $s$ is the shifting time  corresponding to the stress change from $x_1$ to $x_2,$  which depends on the model parameters and satisfies third-order equation defined in Eq. (\ref{eq:thirdorderequation}).
	
	Firstly, we need to compute the derivatives of the shifting time $s$ (as a function of the parameters $\boldsymbol{\theta}$).  Taking implicit derivatives in Eq. (\ref{eq:thirdorderequation}) with respect to each parameter $(\gamma_0, \gamma_1, \gamma_2, a_1),$ we obtain the following system of linear equations:
	\begin{equation*}
		\begin{aligned}
			\gamma_2(\tau+s)^2 \frac{\partial s}{\partial \gamma_0} +\gamma_1(\tau+s)\frac{\partial s}{\partial \gamma_0} + \gamma_0 \frac{\partial s}{\partial \gamma_0} &= - (\tau+s) + \frac{\exp(a_1x_1)}{\exp(a_1x_2)}\tau , \\
			\gamma_2(\tau+s)^2 \frac{\partial s}{\partial \gamma_1} + \gamma_1(\tau+s)\frac{\partial s}{\partial \gamma_1} + \gamma_0 \frac{\partial s}{\partial \gamma_1} &= - \frac{1}{2}(\tau+s)^2 + \frac{\exp(a_1x_1)}{\exp(a_1x_2)} \frac{1}{2}\tau^2,  \\
			\gamma_2(\tau+s)^2 \frac{\partial s}{\partial \gamma_2} + \gamma_1(\tau+s)\frac{\partial s}{\partial \gamma_2} + \gamma_0 \frac{\partial s}{\partial \gamma_2} &= - \frac{1}{3}(\tau+s)^3 + \frac{\exp(a_1x_1)}{\exp(a_1x_2)} \frac{1}{3}\tau^3,  \\
			%\gamma_2(\tau+s)^2 \frac{\partial s}{\partial a_0} + \gamma_1(\tau+s)\frac{\partial s}{\partial a_0} + \gamma_0 \frac{\partial s}{\partial a_0} &= 0, \\
			\gamma_2(\tau+s)^2 \frac{\partial s}{\partial a_1}+ \gamma_1(\tau+s)\frac{\partial s}{\partial a_1} + \gamma_0 \frac{\partial s}{\partial a_1} &= \frac{\exp(a_1x_1)}{\exp(a_1x_2)}(x_1-x_2) \left(\frac{\gamma_2}{3} \tau^3 + \frac{\gamma_1}{2}\tau^2 + \gamma_0\tau\right).
		\end{aligned}
	\end{equation*}
	Let us define  $$k_1(t) = \gamma_2t^2+\gamma_1t+\gamma_0 \hspace{1cm}\text{and}\hspace{1cm} k_2(t) = \frac{\gamma_2}{3}t^2+ \frac{\gamma_1}{2}t + \gamma_0.$$
	Then, the system of equations can be rewritten as
	\begin{equation*}
		\begin{aligned}
			k_1(\tau+s) \frac{\partial s}{\partial \gamma_0} &= - (\tau+s) + \frac{\exp(a_1x_1)}{\exp(a_1x_2)}\tau , \\
			k_1(\tau+s) \frac{\partial s}{\partial \gamma_1} &= - \frac{1}{2}(\tau+s)^2 + \frac{\exp(a_1x_1)}{\exp(a_1x_2)} \frac{1}{2}\tau^2,  \\
			k_1(\tau+s) \frac{\partial s}{\partial \gamma_2} &= - \frac{1}{3}(\tau+s)^3 + \frac{\exp(a_1x_1)}{\exp(a_1x_2)} \frac{1}{3}\tau^3,  \\
			%k_1(\tau+s) \frac{\partial s}{\partial a_0} &= 0, \\
			k_1(\tau+s) \frac{\partial s}{\partial a_1} &= \frac{\exp(a_1x_1)}{\exp(a_1x_2)}(x_1-x_2) \tau k_2(\tau).
		\end{aligned}
	\end{equation*}

	Solving the above equations, we obtain the derivatives of the shifting time as follows:
	\begin{equation*}
		\begin{aligned}
			\frac{\partial s}{\partial \gamma_0} & = \frac{(\tau+s)}{k_1(\tau+s)}\left(\frac{k_2(\tau+s)}{k_2(\tau)}-1\right),   \\
			\frac{\partial s}{\partial \gamma_1}  & = \frac{1}{2}\frac{(\tau+s)}{k_1(\tau+s)}\left(\tau\frac{k_2(\tau+s)}{k_2(\tau)}-(\tau+s)\right),  \\
				\frac{\partial s}{\partial \gamma_2}  & =  \frac{1}{3}\frac{(\tau+s)}{k_1(\tau+s)}\left(\tau^2\frac{k_2(\tau+s)}{k_2(\tau)}-(\tau+s)^2\right),   \\
		%	\frac{\partial s}{\partial a_0}  & =  0 &
			\frac{\partial s}{\partial a_1}  & =  (\tau+s)\frac{k_2(\tau+s)}{k_1(\tau+s)}(x_1-x_2),
			% \gamma_1(\tau+s)\frac{\partial s}{\partial a_1} + \gamma_0 \frac{\partial s}{\partial a_1} = \frac{1}{2}\frac{\exp(a_1x_1)}{\exp(a_1x_2)}\frac{\gamma_1\tau^2 + 2\gamma_0\tau}{\gamma_1(\tau+s)+\gamma_0}(x_1-x_2) 
		\end{aligned}
	\end{equation*}
	where we have used the equality
	$$\frac{\exp(a_1x_1)}{\exp(a_1x_2)} = \frac{(\tau+s)}{ \tau}\frac{k_2(\tau+s)}{k_2(\tau)}$$
	to simplify the expressions.
	Now, we compute separately the derivatives of the reliability under stress levels $x_1$ and $x_2.$ Under the first stress level $x_1$, we have 
	
	\begin{equation*}
		\begin{aligned}
			\frac{\partial  R_{\boldsymbol{\theta}}(t,x_1)}{\partial \gamma_0} &= -	R_{\boldsymbol{\theta}}(t,x_1)
			\exp(a_1x_1) t ,\\
			\frac{\partial R_{\boldsymbol{\theta}}(t,x_1)}{\partial \gamma_1} &= -R_{\boldsymbol{\theta}}(t,x_1)
			\exp(a_1x_1) \frac{t^2}{2} ,\\
			\frac{\partial R_{\boldsymbol{\theta}}(t,x_1)}{\partial \gamma_2} &= -R_{\boldsymbol{\theta}}(t,x_1)
			\exp(a_1x_1) \frac{t^3}{3} ,\\
			\frac{\partial R_{\boldsymbol{\theta}}(t,x_1)}{\partial a_0} &= 	-R_{\boldsymbol{\theta}}(t,x_1) \exp(a_1x_1)tk_2(t),\\
			\frac{\partial R_{\boldsymbol{\theta}}(t,x_1)}{\partial a_1} &= 	-R_{\boldsymbol{\theta}}(t,x_1) \exp(a_1x_1)tk_2(t) x_2.\\
		\end{aligned}
	\end{equation*}
	
	%REVISAR PARCIAL A_1
	On the other hand, the derivatives of the reliability under the increased stress $x_2$ with respect to each parameter are as follows:
	\begin{equation*}
		\begin{aligned}
			\frac{\partial R_{\boldsymbol{\theta}}(t,x_2)}{\partial \gamma_0} &= 	\exp\left(-\exp(a_1x_2) \left( \gamma_0(t+s) + \gamma_1  \frac{(t+s)^2}{2} + \gamma_2\frac{(t+s)^3}{3}\right)\right)\\
			&  \hspace{0.5cm} \times \left[
			- \exp(a_1x_2) \left((t+s)+\gamma_0\frac{\partial s}{\partial \gamma_0} + \gamma_1(t+s)\frac{\partial s}{\partial \gamma_0} + \gamma_2(t+s)^2\frac{\partial s}{\partial \gamma_0}\right)\right],\\
			\frac{\partial R_{\boldsymbol{\theta}}(t,x_2)}{\partial \gamma_1} &= 	\exp\left(-\exp(a_1x_2) \left( \gamma_0(t+s) + \gamma_1 \frac{(t+s)^2}{2} + \gamma_2\frac{(t+s)^3}{3} \right)\right)\\
			&  \hspace{0.5cm} \times \left[
			- \exp(a_1x_2) \left(\frac{(t+s)^2}{2}+\gamma_0\frac{\partial s}{\partial \gamma_1} + \gamma_1(t+s)\frac{\partial s}{\partial \gamma_1} + \gamma_1(t+s)\frac{\partial s}{\partial \gamma_0} + \gamma_2(t+s)^2\frac{\partial s}{\partial \gamma_1} \right)\right],\\
			\frac{\partial R_{\boldsymbol{\theta}}(t,x_2)}{\partial \gamma_2} &= 	\exp\left(-\exp(a_1x_2) \left( \gamma_0(t+s) + \gamma_1  \frac{(t+s)^2}{2} + \gamma_2\frac{(t+s)^3}{3}\right)\right)\\
			&  \hspace{0.5cm} \times \left[
			- \exp(a_1x_2) \left(\frac{(t+s)^3}{3}+\gamma_0\frac{\partial s}{\partial \gamma_2} + \gamma_1(t+s)\frac{\partial s}{\partial \gamma_2} + \gamma_2(t+s)^2\frac{\partial s}{\partial \gamma_2}\right)\right],\\
			\frac{\partial R_{\boldsymbol{\theta}}(t,x_2)}{\partial a_0} &= 	\exp\left(-\exp(a_1x_2) \left( \gamma_0(t+s) + \gamma_1 \frac{(t+s)^2}{2} + \gamma_2\frac{(t+s)^3}{3} \right)\right)\\
			&  \hspace{0.5cm} \times \left[-\exp(a_1x_2) \left( \gamma_0(t+s) + \gamma_1 \frac{(t+s)^2}{2} + \gamma_2\frac{(t+s)^3}{3} \right) \right.
			\\	&  \hspace{1cm} 	
			\left. - \exp(a_1x_2) \left(\gamma_0\frac{\partial s}{\partial a_0} + \gamma_1(t+s)\frac{\partial s}{\partial a_0} + \gamma_2(t+s)^2\frac{\partial s}{\partial a_0}\right)\right],\\
			\frac{\partial R_{\boldsymbol{\theta}}(t,x_2)}{\partial a_1} &= 	\exp\left(-\exp(a_1x_2) \left( \gamma_0(t+s) + \gamma_1 \frac{(t+s)^2}{2} + \gamma_2\frac{(t+s)^3}{3} \right)\right)\\
			&  \hspace{0.5cm} \times \left[-\exp(a_1x_2) \left( \gamma_0(t+s) + \gamma_1 \frac{(t+s)^2}{2}+ \gamma_2\frac{(t+s)^3}{3}\right) x_2 \right.
			\\	&  \hspace{1cm} 	
			\left. 	- \exp(a_1x_2) \left(\gamma_0\frac{\partial s}{\partial a_1} + \gamma_1(t+s)\frac{\partial s}{\partial a_1} + \gamma_2(t+s)^2\frac{\partial s}{\partial a_1}\right)\right].\\
		\end{aligned}
	\end{equation*}
	Upon substituting the formulas of the shifting time derivatives in the above expressions, we can rewrite the above derivatives as follows:
	\begin{equation*}
		\begin{aligned}
			\frac{\partial  R_{\boldsymbol{\theta}}(t,x_2)}{\partial \gamma_0} &= -	R_{\boldsymbol{\theta}}(t,x_2) 
			\exp(a_1x_2) \left( (t+s)+  (\tau+s)\frac{k_1(t+s) }{k_1(\tau+s)}\left(\frac{k_2(\tau+s)}{k_2(\tau)}-1\right)
			\right),\\
			\frac{\partial R_{\boldsymbol{\theta}}(t,x_2)}{\partial \gamma_1} &= -R_{\boldsymbol{\theta}}(t,x_2)
			\exp(a_1x_2) \left(\frac{(t+s)^2}{2}+\frac{(\tau+s)}{2}\frac{ k_1(t+s)}{k_1(\tau+s)}\left(\tau\frac{k_2(\tau+s)}{k_2(\tau)}-(\tau+s)\right)          \right),\\
			\frac{\partial R_{\boldsymbol{\theta}}(t,x_2)}{\partial \gamma_1} &= -R_{\boldsymbol{\theta}}(t,x_2)
			\exp(a_1x_2) \left(\frac{(t+s)^3}{3}+  \frac{(\tau+s)}{3}\frac{k_1(t+s)}{k_1(\tau+s)}\left(\tau^2\frac{k_2(\tau+s)}{k_2(\tau)}-(\tau+s)^2\right)        \right),\\
			%
			%\frac{\partial R_{\boldsymbol{\theta}}(t,x_2)}{\partial a_0} &= 	-R_{\boldsymbol{\theta}}(t,x_2)\exp(a_1x_2)(t+s)k_2(t+s),\\
			%
			\frac{\partial R_{\boldsymbol{\theta}}(t,x_2)}{\partial a_1} &=  - R_{\boldsymbol{\theta}}(t,x_2) \exp(a_1x_2) \left[(t+s)k_2(t+s) x_2 +
			  \frac{k_1(t+s)}{k_1(\tau+s)}(\tau+s)k_2(\tau+s)(x_1-x_2) \right].\
		\end{aligned}
	\end{equation*}
	
	Now, defining  $$\boldsymbol{w}_j = \begin{cases}
		\frac{\partial R_{\boldsymbol{\theta}}(t_{j-1},x_1)}{\partial \boldsymbol{\theta}} - \frac{\partial R_{\boldsymbol{\theta}}(t_{j},x_1)}{\partial \boldsymbol{\theta}}, & t_{j-1} < \tau, \\
		\frac{\partial R_{\boldsymbol{\theta}}(t_{j-1},x_2)}{\partial \boldsymbol{\theta}} - \frac{\partial R_{\boldsymbol{\theta}}(t_{j},x_2)}{\partial \boldsymbol{\theta}}, & \tau \leq t_{j-1},
	\end{cases}$$
	we obtain the  required formula.		
\end{proof}

	\section{Calculations for the lifetime characteristics under linear baseline hazard }
	
	\subsection{Computation of the mean lifetime and its derivatives \label{app:meanlinear}}
	
	Let us consider a linear baseline hazard function of the form in (\ref{eq:baselinehazardlinear}) and a constant stress $x_0.$ Then, the cumulative distribution function of the products in this case is 
	\begin{equation*} 
		F_T(t) = 	1- \exp\left[-\exp\left( a_1 x_0\right) \left( \gamma_0t + \gamma_1 \frac{t^2}{2}\right)\right], \hspace{0.3cm} 0 < t < \infty,
	\end{equation*}
	and the corresponding distribution density function is
		\begin{equation*} 
		f_T(t) =  -\exp\left( a_1 x_0\right)(\gamma_0+\gamma_1t)\exp\left[\exp \left(a_0 + a_1 x_0\right) \left( \gamma_0t + \gamma_1 \frac{t^2}{2}\right)\right], \hspace{0.3cm} 0 < t < \infty,
	\end{equation*}
	where $\exp\left( a_1 x_0\right)$ represents the multiplicative effect of the stress level.
	Then, the mean lifetime of the product can be obtained as follows:
	\begin{equation}
		\begin{aligned}
			E_{\boldsymbol{\theta}}[T] 
			&= \int_0^\infty t f_T(t) dt\\
			 &= 	\int_0^\infty  \exp\left( a_1 x_0\right)(\gamma_0t+\gamma_1t^2)\exp\left[-\exp \left[a_0 + a_1 x_0\right] \left( \gamma_0t + \gamma_1 \frac{t^2}{2}\right)\right] dt \\
			&= 	\int_0^\infty \exp\left[-\exp\left( a_1 x_0\right) \left( \gamma_0t + \gamma_1 \frac{t^2}{2}\right)\right] dt \\
			&=  \int_0^\infty \exp\left[-\frac{\gamma_1}{2}\exp\left( a_1 x_0\right) \left( 2\frac{\gamma_0}{\gamma_1}t +  t^2\right)\right] dt \\
			&= \exp\left(\frac{\gamma_1}{2}\exp(a_1x_0)\frac{\gamma_0^2}{\gamma_1^2}\right) \int_0^\infty \exp\left[-\frac{\gamma_1}{2}\exp\left( a_1 x_0\right) \left( t+ \frac{\gamma_0}{\gamma_1}\right)^2\right] dt \\
			&= \exp\left(\exp(a_1x_0)\frac{\gamma_0^2}{2\gamma_1}\right)\sqrt{\frac{2\pi}{\gamma_1 \exp(a_1x_0)}} \\
			& \hspace{0.5cm} \times \int_0^\infty  \sqrt{\frac{\gamma_1 \exp(a_1x_0)}{2\pi}} \exp\left[-\frac{\gamma_1}{2}\exp\left( a_1 x_0\right) \left( t+ \frac{\gamma_0}{\gamma_1}\right)^2\right] dt \\
			&= \exp\left(\exp(a_1x_0)\frac{\gamma_0^2}{2\gamma_1}\right)\sqrt{\frac{2\pi}{\gamma_1 \exp(a_1x_0)}} \\
			& \hspace{0.5cm} \times \int_{\frac{\gamma_0}{\gamma_1}\sqrt{\gamma_1 \exp(a_1x_0)}}^\infty  \frac{1}{\sqrt{2\pi}} \exp\left(-\frac{1}{2}  u^2 \right) du \\
			&= \exp\left[\exp(a_1x_0)\frac{\gamma_0^2}{2\gamma_1}\right]\sqrt{\frac{2\pi}{\gamma_1 \exp(a_1x_0)}} \Phi\left(-\frac{\gamma_0}{\sqrt{\gamma_1}}, \exp\left(\frac{a_1x_0}{2}\right)\right),
		\end{aligned}
	\end{equation}
	where $\Phi$ denotes the cumulative distribution function of a standard normal variable.
	
	Now, the derivatives of  $E_{\boldsymbol{\theta}}[T]$ with respect to every parameter are as follows:
	\begin{equation}
		\begin{aligned}
			\frac{\partial E_{\boldsymbol{\theta}}[T] }{\partial \gamma_0} &=  \sqrt{\frac{2\pi}{\gamma_1 \exp(a_1x_0)}} \left[\exp\left(\frac{\exp(a_1x_0)\gamma_0^2}{2\gamma_1}\right) \frac{\exp(a_1x_0)\gamma_0}{\gamma_1} \Phi\left(-\frac{\gamma_0}{\sqrt{\gamma_1}}\exp\left(\frac{a_1x_0}{2}\right)\right) \right.\\
			\hspace{0.5cm} & 
			\left. - \exp\left(\frac{\exp(a_1x_0)\gamma_0^2}{2\gamma_1}\right) \phi\left(-\frac{\gamma_0}{\sqrt{\gamma_1}} \exp\left(\frac{a_1x_0}{2}\right)\right) \frac{\exp(a_1x_0/2)}{\sqrt{\gamma_1}}
			 \right] \\
			&= \frac{\gamma_0}{\gamma_1} E_{\boldsymbol{\theta}}[T] \exp(a_1x_0) - \frac{\sqrt{2\pi}}{\gamma_1 } \exp\left[\frac{\exp(a_1x_0)\gamma_0^2}{2\gamma_1}\right] \phi\left(-\frac{\gamma_0}{\sqrt{\gamma_1}} \exp\left(\frac{a_1x_0}{2}\right)\right),\\
			\frac{\partial E_{\boldsymbol{\theta}}[T] }{\partial \gamma_1}& =  \sqrt{\frac{2\pi}{ \exp(a_1x_0)}}\exp\left(\frac{\exp(a_1x_0)\gamma_0^2}{2\gamma_1}\right) \left[-\frac{1}{2}\gamma_1^{-3/2} - \frac{1}{\sqrt{\gamma_1}}\frac{\exp(a_1x_0)\gamma_0^2}{2\gamma_1^2}\right]\Phi\left(-\frac{\gamma_0}{\sqrt{\gamma_1}}\sqrt{ \exp(a_1x_0)}\right) \\
			\hspace{0.5cm} & +   \exp\left(\exp(a_1x_0)\frac{\gamma_0^2}{2\gamma_1}\right)\sqrt{\frac{2\pi}{\gamma_1 \exp(a_1x_0)}} \phi\left(-\frac{\gamma_0}{\sqrt{\gamma_1}}\sqrt{ \exp(a_1x_0)}\right)  \frac{\gamma_0}{2\gamma_1^{3/2}}\sqrt{ \exp(a_1x_0)} \\
			&= \frac{1}{2\gamma_1}E_{\boldsymbol{\theta}}[T]\left[-1 - \exp(a_1x_0)\frac{\gamma_0^2}{\gamma_1}\right] + 
			\exp\left(\exp(a_1x_0)\frac{\gamma_0^2}{2\gamma_1}\right) \phi\left(-\frac{\gamma_0}{\sqrt{\gamma_1}}\sqrt{ \exp(a_1x_0)}\right)  \frac{\sqrt{2\pi} \gamma_0}{2\gamma_1^2 },\\
			%
			%\frac{\partial E_{\boldsymbol{\theta}}[T] }{\partial a_0} & = \frac{1}{2} \exp\left(\frac{\exp(a_1x_0)\gamma_0^2}{2\gamma_1}\right)\sqrt{\frac{2\pi}{\gamma_1}}\left[\frac{-1}{2\sqrt{\exp(a_1x_0)}} + \frac{\sqrt{\exp(a_1x_0)}\gamma_0^2}{2\gamma_1} \right] \\
			%&= \frac{1}{2} E_{\boldsymbol{\theta}}[T] \left[ -1 + \frac{\gamma_0^2}{\gamma_1}\exp(a_1x_0) \right].\\
			%
			\frac{\partial E_{\boldsymbol{\theta}}[T] }{\partial a_1} &= \frac{x_0}{2} \exp\left(\frac{\exp(a_1x_0)\gamma_0^2}{2\gamma_1}\right)\sqrt{\frac{2\pi}{\gamma_1}}\left[\frac{-1}{2\sqrt{\exp(a_1x_0)}} + \frac{\sqrt{\exp(a_1x_0)}\gamma_0^2}{2\gamma_1} \right]\Phi\left(-\frac{\gamma_0}{\sqrt{\gamma_1}}\sqrt{ \exp(a_1x_0)}\right)\\
			&\hspace{0.7cm}  + \exp\left(\exp(a_1x_0)\frac{\gamma_0^2}{2\gamma_1}\right)\sqrt{\frac{2\pi}{\gamma_1 \exp(a_1x_0)}} \phi\left(-\frac{\gamma_0}{\sqrt{\gamma_1}} \exp\left(\frac{a_1x_0}{2}\right) \right) \left(-\frac{\gamma_0}{\sqrt{\gamma_1}} \exp\left(\frac{a_1x_0}{2}\right) \frac{x_0}{2}\right)	\\
			&= E_{\boldsymbol{\theta}}[T]\Phi\left(-\frac{\gamma_0}{\sqrt{\gamma_1}}\sqrt{ \exp(a_1x_0)}\right)\left[ -1 + \frac{\gamma_0^2}{\gamma_1}\exp(a_1x_0)\right]\frac{x_0}{2}\\
			 & \hspace{0.7cm} - \exp\left(\exp(a_1x_0)\frac{\gamma_0^2}{2\gamma_1}\right) \phi\left(-\frac{\gamma_0}{\sqrt{\gamma_1}} \exp\left(\frac{a_1x_0}{2}\right) \right) \frac{\sqrt{2\pi}\gamma_0}{\gamma_1}   \frac{x_0}{2} .
		\end{aligned}
	\end{equation}

	\subsection{Computation of the $\alpha-$quantile and its derivatives  \label{app:quantilelinear}}
	
	As  discussed in Section \ref{sec:char}, the $\alpha-$quantile of the distribution should satisfy the second-order equation
	\begin{equation}  \label{eq:appquantileec}
		\frac{\gamma_1 }{2}Q_{\boldsymbol{\theta}}(\alpha)^2 + \gamma_0Q_{\boldsymbol{\theta}}(\alpha) + \log(\alpha)\exp\left( -a_1^\beta x_0\right) = 0. 
	\end{equation}
Now, taking derivatives with respect to every parameter $(\gamma_0, \gamma_1, a_0, a_1),$ we obtain the following system of equations:
%	\begin{equation}
%		Q_{\boldsymbol{\theta}}(\alpha) = \frac{-\gamma_0 + \sqrt{\gamma_0^2-2\gamma_1 \log(\alpha)\exp\left( -a_1^\beta x_0\right)}}{\gamma_1}
%	\end{equation}
	\begin{equation*}
		\begin{aligned}
			\gamma_1 Q_{\boldsymbol{\theta}}(\alpha) \frac{\partial Q_{\boldsymbol{\theta}}(\alpha)}{\partial \gamma_0} +  Q_{\boldsymbol{\theta}}(\alpha) + \gamma_0\frac{\partial Q_{\boldsymbol{\theta}}(\alpha)}{\partial \gamma_0}  &= 0, \\
			\frac{1}{2}  Q_{\boldsymbol{\theta}}(\alpha)^2 + \gamma_1 Q_{\boldsymbol{\theta}}(\alpha) \frac{\partial Q_{\boldsymbol{\theta}}(\alpha)}{\partial \gamma_1}  \gamma_0\frac{\partial Q_{\boldsymbol{\theta}}(\alpha)}{\partial \gamma_1}  &= 0, \\
			\gamma_1 Q_{\boldsymbol{\theta}}(\alpha) \frac{\partial Q_{\boldsymbol{\theta}}(\alpha)}{\partial a_0} + \gamma_0\frac{\partial Q_{\boldsymbol{\theta}}(\alpha)}{\partial a_0} -\log(\alpha)\exp\left( -a_1^\beta x_0\right)  &= 0, \\
			\gamma_1 Q_{\boldsymbol{\theta}}(\alpha) \frac{\partial Q_{\boldsymbol{\theta}}(\alpha)}{\partial a_1} +   \gamma_0\frac{\partial Q_{\boldsymbol{\theta}}(\alpha)}{\partial a_1} - \log(\alpha)\exp\left( -a_1^\beta x_0\right)x_0 &= 0. \\
		\end{aligned}
		\end{equation*}
Upon solving the above equations, we readily obtain the derivatives of the $\alpha$-quantile function as follows:
	\begin{equation*}
	\begin{aligned}
		 \frac{\partial Q_{\boldsymbol{\theta}}(\alpha)}{\partial \gamma_0} &= -  \frac{ Q_{\boldsymbol{\theta}}(\alpha) }{\gamma_1  Q_{\boldsymbol{\theta}}(\alpha)  + \gamma_0}, \\
			 \frac{\partial Q_{\boldsymbol{\theta}}(\alpha)}{\partial \gamma_1} &= -  \frac{ Q_{\boldsymbol{\theta}}(\alpha)^2 }{2\gamma_1  Q_{\boldsymbol{\theta}}(\alpha)  + 2\gamma_0}. \\
		 \frac{\partial Q_{\boldsymbol{\theta}}(\alpha)}{\partial a_0} &=  \frac{ \log(\alpha)\exp\left( -a_1^\beta x_0\right) }{\gamma_1  Q_{\boldsymbol{\theta}}(\alpha)  + \gamma_0}, \\
			 \frac{\partial Q_{\boldsymbol{\theta}}(\alpha)}{\partial a_1} &=   \frac{ \log(\alpha)\exp\left( -a_1^\beta x_0\right) }{\gamma_1  Q_{\boldsymbol{\theta}}(\alpha)  + \gamma_0}x_0. \\
	\end{aligned}
\end{equation*}
Now, by using Eq. (\ref{eq:appquantileec}), we can rewrite $-\log(\alpha)\exp\left( -a_1^\beta x_0\right) = \frac{\gamma_1 }{2}Q_{\boldsymbol{\theta}}(\alpha)^2 + \gamma_0Q_{\boldsymbol{\theta}}(\alpha),$ and thus the expressions presented in Section \ref{sec:char_quantile_linear} are obtained.
%\begin{equation}
%	\frac{\partial Q_{\boldsymbol{\theta}}(\alpha)}{\partial \boldsymbol{\theta}} =  - \frac{ Q_{\boldsymbol{\theta}}(\alpha) }{\gamma_1  Q_{\boldsymbol{\theta}}(\alpha)  + \gamma_0}\left(1,  \frac{Q_{\boldsymbol{\theta}}(\alpha)}{2}, \frac{\gamma_1 }{2}Q_{\boldsymbol{\theta}}(\alpha) + \gamma_0, \frac{\gamma_1 }{2}Q_{\boldsymbol{\theta}}(\alpha)x_0 + \gamma_0x_0 \right),
%\end{equation}

\section{Calculations for the lifetime characteristics under quadratic baseline hazard }

\subsection{Computation of the mean lifetime and its derivatives  \label{app:meanquadratic}}

Let us consider the mean lifetime of a device under the quadratic hazard assumption and constant stress level $x_0,$ given by
\begin{equation*}
	E_{\boldsymbol{\theta}}[T] = \exp\left( a_1 x_0\right) \int_{0}^{\infty}  \left( \gamma_0t + \gamma_1 t^2 + \gamma_2 t^3 \right) \exp\left[-\exp\left( a_1 x_0\right) \left( \gamma_0t + \gamma_1 \frac{t^2}{2} + \gamma_2 \frac{t^3}{3}\right)\right] dt.
\end{equation*}
Taking derivatives with respect to every model parameter $\boldsymbol{\theta} = \left( \gamma_0, \gamma_1, \gamma_2, a_0, a_1\right),$ we get

\begin{equation}
	\begin{aligned}
	\frac{\partial E_{\boldsymbol{\theta}}[T] }{\partial \gamma_0} &=  \exp\left( a_1 x_0\right) \int_{0}^{\infty}  t \exp\left[-\exp\left( a_1 x_0\right) \left( \gamma_0t + \gamma_1 \frac{t^2}{2} + \gamma_2 \frac{t^3}{3}\right)\right] dt \\
	& \hspace{0.5cm}-  \exp\left( a_1 x_0\right)^2 \int_{0}^{\infty}  \left( \gamma_0t + \gamma_1 t^2 + \gamma_2 t^3 \right)t \exp\left[-\exp\left( a_1 x_0\right) \left( \gamma_0t + \gamma_1 \frac{t^2}{2} + \gamma_2 \frac{t^3}{3}\right)\right] dt\\
	&= \exp\left( a_1 x_0\right) \int_{0}^{\infty}  tR_{\boldsymbol{\theta}}(t)dt - \int_{0}^{\infty}  t^2h_{\boldsymbol{\theta}}(t)R_{\boldsymbol{\theta}}(t)dt,\\
	\frac{\partial E_{\boldsymbol{\theta}}[T] }{\partial \gamma_1} &=  \exp\left( a_1 x_0\right) \int_{0}^{\infty}  t^2 \exp\left[-\exp\left( a_1 x_0\right) \left( \gamma_0t + \gamma_1 \frac{t^2}{2} + \gamma_2 \frac{t^3}{3}\right)\right]dt \\
	& \hspace{0.5cm} -  \exp\left( a_1 x_0\right)^2 \int_{0}^{\infty}  \left( \gamma_0t + \gamma_1 t^2 + \gamma_2 t^3 \right)\frac{t^2}{2} \exp\left[-\exp\left( a_1 x_0\right) \left( \gamma_0t + \gamma_1 \frac{t^2}{2} + \gamma_2 \frac{t^3}{3}\right)\right] dt\\
	&= \exp\left( a_1 x_0\right) \int_{0}^{\infty}  t^2 R_{\boldsymbol{\theta}}(t)dt - \int_{0}^{\infty}  \frac{t^3}{2}h_{\boldsymbol{\theta}}(t)R_{\boldsymbol{\theta}}(t)dt,\\
	\frac{\partial E_{\boldsymbol{\theta}}[T] }{\partial \gamma_2} &=  \exp\left( a_1 x_0\right) \int_{0}^{\infty}  t^3 \exp\left[-\exp\left( a_1 x_0\right) \left( \gamma_0t + \gamma_1 \frac{t^2}{2} + \gamma_2 \frac{t^3}{3}\right)\right] dt \\
	& \hspace{0.5cm} -  \exp\left( a_1 x_0\right)^2 \int_{0}^{\infty}  \left( \gamma_0t + \gamma_1 t^2 + \gamma_2 t^3 \right)\frac{t^3}{3} \exp\left[-\exp\left( a_1 x_0\right) \left( \gamma_0t + \gamma_1 \frac{t^2}{2} + \gamma_2 \frac{t^3}{3}\right)\right] dt\\
	&= \exp\left( a_1 x_0\right) \int_{0}^{\infty}  t^3 R_{\boldsymbol{\theta}}(t)dt - \int_{0}^{\infty}  \frac{t^4}{3}h_{\boldsymbol{\theta}}(t)R_{\boldsymbol{\theta}}(t)dt\\
	%
%	\frac{\partial E_{\boldsymbol{\theta}}[T] }{\partial a_0} &=  \exp\left( a_1 x_0\right) \int_{0}^{\infty}  \left( \gamma_0t + \gamma_1 t^2 + \gamma_2 t^3 \right) \exp\left(-\exp\left( a_1 x_0\right) \left( \gamma_0t + \gamma_1 \frac{t^2}{2} + \gamma_2 \frac{t^3}{3}\right)\right) dt \\
%	& \hspace{0.5cm} -  \exp\left( a_1 x_0\right)^2 \int_{0}^{\infty}  \left( \gamma_0t + \gamma_1 t^2 + \gamma_2 t^3 \right) \exp\left(-\exp\left( a_1 x_0\right) \left( \gamma_0t + \gamma_1 \frac{t^2}{2} + \gamma_2 \frac{t^3}{3}\right)\right) dt\\
	&= (1-\exp\left( a_1 x_0\right)) E_{\boldsymbol{\theta}}[T],\\
	\frac{\partial E_{\boldsymbol{\theta}}[T] }{\partial a_1} &=   (1-\exp\left( a_1 x_0\right)) E_{\boldsymbol{\theta}}[T]  x_0.
	\end{aligned}
\end{equation}

\subsection{Computation of the $\alpha-$quantile and its derivatives  \label{app:quantilequadratic}}

The $\alpha-$quantile of the distribution should satisfy the  third-order equation
\begin{equation*}  
	\frac{\gamma_1 }{3}Q_{\boldsymbol{\theta}}(\alpha)^3 + \frac{\gamma_1 }{2}Q_{\boldsymbol{\theta}}(\alpha)^2 + \gamma_0Q_{\boldsymbol{\theta}}(\alpha) + \log(\alpha)\exp\left( -a_1^\beta x_0\right) = 0. 
\end{equation*}
Now, taking derivatives with respect every parameter $(\gamma_0, \gamma_1, \gamma_2, a_0, a_1),$ we obtain the following system of equations:
%	\begin{equation}
	%		Q_{\boldsymbol{\theta}}(\alpha) = \frac{-\gamma_0 + \sqrt{\gamma_0^2-2\gamma_1 \log(\alpha)\exp\left( -a_1^\beta x_0\right)}}{\gamma_1}
	%	\end{equation}
\begin{equation*}
	\begin{aligned}
		\gamma_2 Q_{\boldsymbol{\theta}}(\alpha)^2 \frac{\partial Q_{\boldsymbol{\theta}}(\alpha)}{\partial \gamma_0} + \gamma_1 Q_{\boldsymbol{\theta}}(\alpha) \frac{\partial Q_{\boldsymbol{\theta}}(\alpha)}{\partial \gamma_0} +  Q_{\boldsymbol{\theta}}(\alpha) + \gamma_0\frac{\partial Q_{\boldsymbol{\theta}}(\alpha)}{\partial \gamma_0}  &= 0, \\
		\gamma_2 Q_{\boldsymbol{\theta}}(\alpha)^2 \frac{\partial Q_{\boldsymbol{\theta}}(\alpha)}{\partial \gamma_1} +\frac{1}{2}  Q_{\boldsymbol{\theta}}(\alpha)^2 + \gamma_1 Q_{\boldsymbol{\theta}}(\alpha) \frac{\partial Q_{\boldsymbol{\theta}}(\alpha)}{\partial \gamma_1}  \gamma_0\frac{\partial Q_{\boldsymbol{\theta}}(\alpha)}{\partial \gamma_1}  &= 0, \\
		\frac{1}{3}  Q_{\boldsymbol{\theta}}(\alpha)^3  + \gamma_2 Q_{\boldsymbol{\theta}}(\alpha)^2 \frac{\partial Q_{\boldsymbol{\theta}}(\alpha)}{\partial \gamma_2} + \gamma_1 Q_{\boldsymbol{\theta}}(\alpha) \frac{\partial Q_{\boldsymbol{\theta}}(\alpha)}{\partial \gamma_2}  \gamma_0\frac{\partial Q_{\boldsymbol{\theta}}(\alpha)}{\partial \gamma_2}  &= 0, \\
		%
		%\gamma_2 Q_{\boldsymbol{\theta}}(\alpha)^2 \frac{\partial Q_{\boldsymbol{\theta}}(\alpha)}{\partial a_0} +\gamma_1 Q_{\boldsymbol{\theta}}(\alpha) \frac{\partial Q_{\boldsymbol{\theta}}(\alpha)}{\partial a_0} + \gamma_0\frac{\partial Q_{\boldsymbol{\theta}}(\alpha)}{\partial a_0} -\log(\alpha)\exp\left( -a_1^\beta x_0\right)  &= 0. \\
		%
		\gamma_2 Q_{\boldsymbol{\theta}}(\alpha)^2 \frac{\partial Q_{\boldsymbol{\theta}}(\alpha)}{\partial a_1} +\gamma_1 Q_{\boldsymbol{\theta}}(\alpha) \frac{\partial Q_{\boldsymbol{\theta}}(\alpha)}{\partial a_1} +   \gamma_0\frac{\partial Q_{\boldsymbol{\theta}}(\alpha)}{\partial a_1} - \log(\alpha)\exp\left( -a_1^\beta x_0\right)x_0 &= 0, \\
	\end{aligned}
\end{equation*}
Upon solving the above equations, we obtain the derivatives of the $\alpha$-quantile function as follows:
\begin{equation*}
	\begin{aligned}
		\frac{\partial Q_{\boldsymbol{\theta}}(\alpha)}{\partial \gamma_0} &= -  \frac{ Q_{\boldsymbol{\theta}}(\alpha) }{\gamma_2 Q_{\boldsymbol{\theta}}(\alpha)^2+ \gamma_1  Q_{\boldsymbol{\theta}}(\alpha)  + \gamma_0}, \\
		\frac{\partial Q_{\boldsymbol{\theta}}(\alpha)}{\partial \gamma_1} &= - \frac{1}{2} \frac{ Q_{\boldsymbol{\theta}}(\alpha)^2 }{\gamma_2 Q_{\boldsymbol{\theta}}(\alpha)^2+ \gamma_1  Q_{\boldsymbol{\theta}}(\alpha)  + \gamma_0}, \\
		\frac{\partial Q_{\boldsymbol{\theta}}(\alpha)}{\partial \gamma_2} &= - \frac{1}{3} \frac{ Q_{\boldsymbol{\theta}}(\alpha)^3 }{\gamma_2 Q_{\boldsymbol{\theta}}(\alpha)^2+ \gamma_1  Q_{\boldsymbol{\theta}}(\alpha)  + \gamma_0}, \\
		%
	%	\frac{\partial Q_{\boldsymbol{\theta}}(\alpha)}{\partial a_0} &=  \frac{ \log(\alpha)\exp\left( -a_1^\beta x_0\right) }{\gamma_2 Q_{\boldsymbol{\theta}}(\alpha)^2+ \gamma_1  Q_{\boldsymbol{\theta}}(\alpha)  + \gamma_0}. \\
		%
		\frac{\partial Q_{\boldsymbol{\theta}}(\alpha)}{\partial a_1} &=   \frac{ \log(\alpha)\exp\left( -a_1^\beta x_0\right) }{\gamma_2 Q_{\boldsymbol{\theta}}(\alpha)^2+ \gamma_1  Q_{\boldsymbol{\theta}}(\alpha)  + \gamma_0}x_0. \\
	\end{aligned}
\end{equation*}
Now, by using Eq. (\ref{eq:appquantileec}), we can rewrite $-\log(\alpha)\exp\left( -a_1^\beta x_0\right) = \frac{\gamma_2 }{3}Q_{\boldsymbol{\theta}}(\alpha)^3+ \frac{\gamma_1 }{2}Q_{\boldsymbol{\theta}}(\alpha)^2 + \gamma_0Q_{\boldsymbol{\theta}}(\alpha),$ and thus expressions presented in Section \ref{sec:char_quantile_quadratic} are obtained.
%\begin{equation}
	%	\frac{\partial Q_{\boldsymbol{\theta}}(\alpha)}{\partial \boldsymbol{\theta}} =  - \frac{ Q_{\boldsymbol{\theta}}(\alpha) }{\gamma_1  Q_{\boldsymbol{\theta}}(\alpha)  + \gamma_0}\left(1,  \frac{Q_{\boldsymbol{\theta}}(\alpha)}{2}, \frac{\gamma_1 }{2}Q_{\boldsymbol{\theta}}(\alpha) + \gamma_0, \frac{\gamma_1 }{2}Q_{\boldsymbol{\theta}}(\alpha)x_0 + \gamma_0x_0 \right),
	%\end{equation}
\end{document}